\documentclass[11pt,reqno]{amsart}

\usepackage{amsmath,amsfonts,amssymb,amscd,amsthm,amsbsy,bbm, epsf,calc,  comment, caption, enumerate}
\usepackage{color}
\usepackage[usenames,dvipsnames]{xcolor}
\usepackage{datetime}
\usepackage{soul}

\usepackage{thmtools, thm-restate}
\usepackage{thm-patch, thm-kv, thm-autoref}
\usepackage[colorlinks=true, urlcolor=blue, linkcolor=blue, citecolor=black]{hyperref}

\usepackage[pdftex]{graphicx}
\usepackage{wrapfig}
\usepackage{geometry}
\numberwithin{equation}{section}
\newcounter{hours}\newcounter{minutes}

\theoremstyle{plain}
\declaretheorem[title=Theorem, parent=section]{theorem}
\declaretheorem[title=Lemma,sibling=theorem]{lemma}
\declaretheorem[title=Corollary,sibling=theorem]{corollary}
\declaretheorem[title=Proposition,sibling=theorem]{proposition}
\declaretheorem[title=Definition,sibling=theorem]{definition}
\declaretheorem[title=Remark,sibling=theorem]{rem}

\newtheorem*{prop*}{Proposition}

\makeatletter
\newcommand\RedeclareMathOperator{%
  \@ifstar{\def\rmo@s{m}\rmo@redeclare}{\def\rmo@s{o}\rmo@redeclare}%
}
\newcommand\rmo@redeclare[2]{%
  \begingroup \escapechar\m@ne\xdef\@gtempa{{\string#1}}\endgroup
  \expandafter\@ifundefined\@gtempa
     {\@latex@error{\noexpand#1undefined}\@ehc}%
     \relax
  \expandafter\rmo@declmathop\rmo@s{#1}{#2}}
\newcommand\rmo@declmathop[3]{%
  \DeclareRobustCommand{#2}{\qopname\newmcodes@#1{#3}}%
}
\@onlypreamble\RedeclareMathOperator
\makeatother

\def\ep{\varepsilon}
\def\al{\alpha}
\def\del{\delta}

\def\om{\omega}

\def\gam{\gamma} 
\def\Gam{\Gamma}

\def\grad{\nabla}

\def\real{\mathbb R}

\def\I{\mathcal I}
\def\K{\mathcal K}

\def\R{\mathbb R}
\def\Rd{{\mathbb R}^d}
\def\Rmc{\mathcal R}
\def\S{\mathcal S}

\def\Id{\textnormal{Id}}

\def\Union{\bigcup}
\def\intersect{\cap}

\def\Indicator{{\mathbbm{1}}}

\def\Id{\textnormal{Id}}
\def\dini{\textnormal{Dini}}

\DeclareMathOperator*{\osc}{osc}
\RedeclareMathOperator{\div}{\textnormal{div}}

\def\polhk#1{\setbox0=\hbox{#1}{\ooalign{\hidewidth
	    \lower1.5ex\hbox{`}\hidewidth\crcr\unhbox0}}}

\newcommand{\abs}[1]{\left| #1 \right|}

\newcommand{\norm}[1]{\lVert#1\rVert}

\begin{document}

\title{Well-posedness for viscosity solutions of the one-phase Muskat problem in all dimensions}

\author{Russell Schwab}
\author{Son Tu} 
\author{Olga Turanova}

\address{Department of Mathematics\\
Michigan State University\\
619 Red Cedar Road \\
East Lansing, MI 48824}
\email{rschwab@math.msu.edu, tuson@msu.edu, turanova@msu.edu}

\begin{abstract}
In this article, we apply the viscosity solutions theory for integro-differential equations to the \emph{one-phase} Muskat equation (also known as the Hele-Shaw problem with gravity). We prove global well-posedness for the corresponding Hamilton-Jacobi-Bellmann equation with bounded, uniformly continuous initial data, in all dimensions.

\end{abstract}

\date{\today}

\thanks{R. Schwab acknowledges support from the Simons Foundation for a Travel Support for Mathematicians grant. O. Turanova  and S. Tu acknowledge support from NSF DMS grant 2204722.}
\subjclass[2020]{
%35J99,      %pde other
35B51, %Comparison principles in context of PDEs
35R09,  	%Integro-partial differential equations
35R35, %Free boundary problems for PDEs
45K05,  	%Integro-partial differential equations 
%46T99,   	%None of the above, but in this section
47G20,      %Integro-differential operators 
49L25,  	%Viscosity solutions to Hamilton-Jacobi equations in optimal control and differential games
%49N70,  	%Differential games
%60J75,      %jump processes
76D27, %	Other free boundary flows; Hele-Shaw flows
76S05 %Flows in porous media; filtration; seepage
%93E20       %optimal stoch. control
%
35Q35 %PDEs in connection with fluid mechanics
76B03 %Existence, uniqueness, and regularity theory for incompressible inviscid fluids 
35D35 %Strong solutions to PDEs
35D40 %Viscosity solutions to PDEs
}

\maketitle

%\markboth{HJB for Muskat}{HJB for Muskat}
\markboth{Well-posedness for viscosity solutions of the one-phase Muskat problem}{Well-posedness for viscosity solutions of the one-phase Muskat problem}

%%%%%%%%%%%%%%%%%%%%%%%%%%%%%%%%%%%%%%%%%%%%%%
%%%%%%%%%%%%%%%%%%%%%%%%%%%%%%%%%%%%%%%%%%%%%%

%%%%%%%%%%%%%%%%%%%%%%%%%%%%%%%%%%%%%%%%%%%%%%%%%
%%%%%%%%%%%%%%%%%%%%%%%%%%%%%%%%%%%%%%%%%%%%%%%%%
%%%%%%%%%%%%%%%%%%%%%%%%%%%%%%%%%%%%%%%%%%%%%%%%%
%%%%%%%%%%%%%%%%%%%%%%%%%%%%%%%%%%%%%%%%%%%%%%%%%
%%%%%%%%%%%%%%%%%%%%%%%%%%%%%%%%%%%%%%%%%%%%%%%%%
%%%%%%%%%%%%%%%%%%%%%%%%%%%%%%%%%%%%%%%%%%%%%%%%%
%%%%%%%%%%%%%%%%%%%%%%%%%%%%%%%%%%%%%%%%%%%%%%%%%
%%%%%%%%%%%%%%%%%%%%%%%%%%%%%%%%%%%%%%%%%%%%%%%%%
%%%%%%%%%%%%%%%%%%%%%%%%%%%%%%%%%%%%%%%%%%%%%%%%%

%%%%%%%%%%%%%%%%%%%%%%%%%%%%%%%%%%%%%%%%%%%%%%%%%
%%%%%%%%%%%%%%%%%%%%%%%%%%%%%%%%%%%%%%%%%%%%%%%%%
%%%%%%%%%%%%%%%%%%%%%%%%%%%%%%%%%%%%%%%%%%%%%%%%%
%%%%%%%%%%%%%%%%%%%%%%%%%%%%%%%%%%%%%%%%%%%%%%%%%
%%%%%%%%%%%%%%%%%%%%%%%%%%%%%%%%%%%%%%%%%%%%%%%%%
%%%%%%%%%%%%%%%%%%%%%%%%%%%%%%%%%%%%%%%%%%%%%%%%%
%%%%%%%%%%%%%%%%%%%%%%%%%%%%%%%%%%%%%%%%%%%%%%%%%
%%%%%%%%%%%%%%%%%%%%%%%%%%%%%%%%%%%%%%%%%%%%%%%%%
%%%%%%%%%%%%%%%%%%%%%%%%%%%%%%%%%%%%%%%%%%%%%%%%%

\section{Introduction}\label{sec:Intro}

We establish global well-posedness for viscosity solutions to the integro-differential  equation that governs the \emph{one-phase} Muskat problem when the fluid region is the subgraph of a function. 
The one-phase Muskat problem, which is a well know and long studied equation, describes the evolution of a fluid under the effect of gravity.  

We will use the following notation. For a function $f:\Rd\rightarrow \R$, we denote its subgraph and graph, respectively, by
\begin{align}\label{eqIntro:DfAndGamf}
	D_f = \{ (x,x_{d+1})\in\real^{d+1}\ :\ x_{d+1}<f(x) \}\ \ \ \textnormal{with}\ \ \ 
	\Gam_f = \partial D_f.
\end{align}
The outward unit normal to $\Gam_f$ at a point $(x,f(x))$ is denoted $n(x)$ and is given by
\begin{align}\label{eqIntro:NormalVector}
	n(x) = \frac{N(x)}{|N(x)|}\ \textnormal{ where }  N(x) = (-\grad f(x),1).
\end{align}
With this in hand, we proceed to describe the one-phase Muskat problem. For each time $t$, the fluid region and free boundary, respectively, are given as the subgraph $D_{f(\cdot, t)}$ and graph $\Gam_{f(\cdot, t)}$ of a function $f:\Rd\times[0,T)\rightarrow \real$.  
Darcy's law and the influence of gravity are encoded in  the assumption that the fluid velocity $u$ and the pressure $p$ (which is taken to be zero on $\partial D_{f(\cdot, t)}$) satisfy

\begin{equation}
\label{eqIntro:rhograv}
\nabla \cdot u = 0, \quad u= -\nabla(p+\rho_{\mathrm{grav}} x_{d+1}) \text{ in }D_{f(\cdot, t)},
\end{equation}
where $\rho_{\mathrm{grav}}$ is the gravitational constant and from now on will be taken to equal 1.   The normal velocity $V$ of the free boundary at time $t$ and at the point $(x,f(x,t))$ is taken to equal the normal component of the velocity:
\begin{equation}
\label{eq:V}
 V= u\cdot n. 
 \end{equation}
Letting  $\phi$ denote the associated potential, so that $u=\nabla\phi$, yields that $\phi(\cdot, t)$ is harmonic in $D_{f(\cdot, t)}$ for each $t$. Moreover,  for $(x,x_{d+1})\in \partial D_{f(\cdot, t)}$ the equality $\phi(x,x_{d+1},t) = x_{d+1} =f(x,t)$ holds. Thus,  the one-phase Muskat problem can be summarized as
\begin{equation}
\label{eqIntro:MuskatInTermsOfOmegatNormalVelocity}
\begin{cases}
\Delta \phi = 0  &\text{ in }  D_{f(\cdot, t)},\\
\phi = f  &\text{ on } \Gamma_{f(\cdot, t)},\\
V = \partial_n\phi, &
\end{cases}
\end{equation}
where the third line is obtained by combining (\ref{eq:V}) with $u\cdot n = \nabla \phi\cdot n = \partial_n\phi$.

This problem can be phrased entirely in terms of a nonlocal evolution equation for $f$. To see this, we use the nonlinear Dirichlet-to-Neumann operator, denoted $G(f)$, which is defined for $f\in C^{1,\gamma}(\R^d)$ as follows.  Given $g\in C^{1,\gamma}(\R^d)$, consider  the unique bounded classical solution $\Phi_{f,g}\in C^2(D_f)\intersect C^{1,\gam}(\overline{D}_f)$ of
 \begin{align}\label{eqIntro:DtoNPotential}
\begin{cases}
         \Delta \Phi_{f,g}= 0    &\textnormal{in}\;D_f,\\
		     \Phi_{f,g} =g  &\textnormal{on}\;\Gamma_f,\\
			 \norm{\Phi_{f,g}}_{L^\infty(D_f)}<\infty, & 
\end{cases}
\end{align}
and, for any $x\in\real^d$, define $[G(f)g](x)$ by
\begin{align}\label{eqIntro:G-NonlinDtoN-NonNormalized}
	[G(f)g](x) &= \lim_{h\to0^-} \frac{1}{h}
	\left[ \Phi_{f,g}\left( (x,f(x)) +hN(x) \right) - \Phi_{f,g}\left( (x,f(x)) \right) \right] 
	\\&= 	\left(\sqrt{1+\abs{\grad f(x)}^2}\right)\partial_{n} \Phi_{f,g}(x,f(x)). \nonumber	
\end{align}
Furthermore, we denote, for $f\in C^{1,\gamma}(\R^d)$,
\begin{align}
\label{eqIntro:M def}
M(f)=-G(f)f.
\end{align}

Returning to the one-phase Muskat problem (\ref{eqIntro:MuskatInTermsOfOmegatNormalVelocity}), we note that, by the graph assumption, the velocity $V$ of the free boundary can be expressed explicitly in terms of $f$; combined with the third line in (\ref{eqIntro:MuskatInTermsOfOmegatNormalVelocity}), this yields
\[
\frac{\partial_t f}{\sqrt{1+|\nabla_x f|^2}} =V = \partial_n\phi.
\]
Noticing that $\phi$ is the unique solution of (\ref{eqIntro:DtoNPotential}) and recalling the definition of the Dirichlet-to-Neumann operator as well as of the operator $M$ allows the previous line to be expressed as:
\begin{align}
\label{eqIntro:MuskatHJB}
	\begin{cases}
		\partial_t f =  M(f)\ &\text{in}\ \real^d\times (0,T)\\
		f(\cdot,0) = f_0\ &\text{on}\ \real^d\times\{ 0 \},
	\end{cases}
&\end{align}
 for an initial datum $f_0$. The relationship between (\ref{eqIntro:MuskatInTermsOfOmegatNormalVelocity}) and  (\ref{eqIntro:MuskatHJB}) has been utilized in \cite{Alazard-2021ConvexityAndHeleShaw-WaterWaves},
\cite{AlazardMeunierSmets-2020LyapounovCauchyProbHeleShaw-CMP},
\cite{DongGancedoNguyen-2023MuskatWP3D-arXiv}, \cite{DongGancedoNguyen-2023MuskatOnePhase2DWellPose-CPAM},
\cite{NguyenPausader-2020ParadifferentialWellPoseMuskatARMA}.

The main result of our paper is the well-posedness of viscosity solutions (see Definition \ref{defViscSol:ViscositySolutions}) to (\ref{eqIntro:MuskatHJB}) for initial datum $f_0$ that is bounded and uniformly continuous on $\real^d$, which we will denote as $f_0\in BUC(\real^d)$.

\begin{theorem}\label{thmIntro:MuskatExistUnique}
	Given any $f_0\in BUC(\real^d)$, there exists a unique viscosity solution, $f\in C^0(\real^d\times [0,T))$, to (\ref{eqIntro:MuskatHJB}).  Furthermore, if $f_0$ has a modulus of continuity, $\om$, then for each $t\in[0,T)$, $f(\cdot, t)$ has the same modulus of continuity.
\end{theorem}

This is the first well-posedness result for viscosity solutions of the one-phase Muskat problem --- more precisely,   to equation (\ref{eqIntro:MuskatHJB}) --- with bounded and uniformly continuous initial datum that holds in general dimension $d$ and with no periodicity assumptions. 
A key advantage of working with viscosity solutions is that the equation holds pointwise wherever the solution is \emph{locally} regular enough. We refer to this as the pointwise evaluation property; see  Section \ref{sec:PointwiseEvaluationH}, Proposition \ref{propViscSol:pointwise-testing}, and Section \ref{ss:punctually}.

We note that    a classical solution,  $f\in C^1(\real^d\times (0,T))$ with a Dini modulus for its first order spacial derivatives, satisfies \eqref{eqIntro:MuskatHJB} if and only if the pair $(\phi,f)$ satisfies  the original Muskat problem \eqref{eqIntro:MuskatInTermsOfOmegatNormalVelocity}.
In this work we only treat (\ref{eqIntro:MuskatHJB}), and leave the connection between weak solutions of (\ref{eqIntro:MuskatHJB}) and (\ref{eqIntro:MuskatInTermsOfOmegatNormalVelocity}),   for future work.

\begin{rem}
\label{remark}
\normalfont
We now address the most immediate and pertinent relationships between Theorem \ref{thmIntro:MuskatExistUnique} and existing works, with a focus on the notion of solution used in each.  We provide further details on this in  Section \ref{sec:ComparisonWithOtherResults};  specifically, in  Subsections \ref{ss:potential} and \ref{ss:NotionsOfSolutions}.
\begin{itemize} 
\item The works 
\cite{Alazard-2021ConvexityAndHeleShaw-WaterWaves}, \cite{AlazardMeunierSmets-2020LyapounovCauchyProbHeleShaw-CMP}, \cite{NguyenPausader-2020ParadifferentialWellPoseMuskatARMA}
 concern \eqref{eqIntro:MuskatHJB} in the setting of periodic initial data in $H^s$ with $s>d/2+1$. In that setting, $M(f)$ is classically defined, and solutions are smooth.  Of these, Theorem 2.1 of \cite{AlazardMeunierSmets-2020LyapounovCauchyProbHeleShaw-CMP}, which concerns local in time well-posedness of \eqref{eqIntro:MuskatHJB}, is  most closely related to Theorem \ref{thmIntro:MuskatExistUnique}. However, Theorem \ref{thmIntro:MuskatExistUnique} does not imply the results in \cite{AlazardMeunierSmets-2020LyapounovCauchyProbHeleShaw-CMP} and, likewise, the results in \cite{AlazardMeunierSmets-2020LyapounovCauchyProbHeleShaw-CMP} do not imply Theorem \ref{thmIntro:MuskatExistUnique}. In addition, \cite{AlazardMeunierSmets-2020LyapounovCauchyProbHeleShaw-CMP}  employs different methods from the ones developed here: namely, variational techniques inspired by water-wave theory.

\item Theorem \ref{thmIntro:MuskatExistUnique} holds in all dimensions and for general initial data.  Therefore,  the results in  \cite{DongGancedoNguyen-2023MuskatWP3D-arXiv}, \cite{DongGancedoNguyen-2023MuskatOnePhase2DWellPose-CPAM} 
that pertain to existence and uniqueness of viscosity solutions
are contained in Theorem \ref{thmIntro:MuskatExistUnique}, as they are posed in dimension $d=1$ and $d=2$ with initial data that is Lipschitz and periodic, hence in $BUC(\real^d)$.  
We note that the aspect of \cite{DongGancedoNguyen-2023MuskatOnePhase2DWellPose-CPAM}, \cite{DongGancedoNguyen-2023MuskatWP3D-arXiv} that is focused on $L^2$ regularity of $\partial_t f$ and the strong solution property is not contained in our result here --- Theorem \ref{thmIntro:MuskatExistUnique} gives no information about the equation in the strong sense.

\item   After this work was finished, we learned of    \cite{AlazardKoch-2023HeleShawSemiFlow-ArXiV}. The results in \cite{AlazardKoch-2023HeleShawSemiFlow-ArXiV} give well-posedness for a different, variational,   notion of solution to (\ref{eqIntro:MuskatHJB}), one which relies upon  formulating \eqref{eqIntro:MuskatInTermsOfOmegatNormalVelocity} in terms of an obstacle problem (see Subsection \ref{ss:NotionsOfSolutions} for more details). If the initial data is regular enough, then classical solutions, viscosity solutions, and solutions in the sense of \cite{AlazardKoch-2023HeleShawSemiFlow-ArXiV}  coincide.  However, when the data is only in $BUC(\real^d)$, it is not known whether solutions in the sense of \cite{AlazardKoch-2023HeleShawSemiFlow-ArXiV} agree with viscosity solutions, nor whether any results in \cite{AlazardKoch-2023HeleShawSemiFlow-ArXiV} imply Theorem \ref{thmIntro:MuskatExistUnique} or vice versa.  Moreover, the methods of proof of  \cite{AlazardKoch-2023HeleShawSemiFlow-ArXiV} are completely different from the approach taken here, as the results of \cite{AlazardKoch-2023HeleShawSemiFlow-ArXiV} hinge on the transformation of the Hele-Shaw problem into the aforementioned obstacle problem, which in turn is understood as a variational inequality.  Another difference is that in the previously listed works, including ours, the solution $f$  to (\ref{eqIntro:MuskatHJB})  is constructed  using the properties of equation (\ref{eqIntro:MuskatHJB}) itself, whereas the variational solutions from \cite{AlazardKoch-2023HeleShawSemiFlow-ArXiV} arise via  an auxiliary problem in $\real^{d+1}$.  We expand upon this in Subsection \ref{ss:NotionsOfSolutions}. 

 \end{itemize}

\end{rem}

%%%%%%%%%%%%%%%%%%%%%%%%%%%%%%%%%%%%%%%%%%%%%%%%%
%%%%%%%%%%%%%%%%%%%%%%%%%%%%%%%%%%%%%%%%%%%%%%%%%
%%%%%%%%%%%%%%%%%%%%%%%%%%%%%%%%%%%%%%%%%%%%%%%%%

\subsection{Method of proof}\label{sec:MainObservation}
Our proof of Theorem \ref{thmIntro:MuskatExistUnique} relies on a connection between (\ref{eqIntro:MuskatHJB}) and the Hele-Shaw problem without gravity. In fact, most of our work concerns an integro-differential operator $H$, which we introduce  below, in \eqref{eqIntro:H}, and define rigorously in Definition \ref{defAuxOp:TheOperatorH}, associated to the Hele-Shaw problem. We prove well-posedness for viscosity solutions to the evolution equation driven by $H$, and deduce Theorem \ref{thmIntro:MuskatExistUnique} from that result. 

In the absence of a graph assumption the Hele-Shaw problem is,
\begin{align}
\label{eqIntro:HS1}
\begin{cases}
\Delta p = 0\quad \text{ in }\{p>0\},\\
\frac{\partial_t p}{|\nabla p|} = |\nabla p| \quad \text{ on }\{p=0\},
\end{cases}
\end{align}
which is coupled with a boundary condition either in the the interior of $\{p>0\}$ or at $x_{d+1}=-\infty$. 
We remark that the quantity $\frac{\partial_t p}{|\nabla p|} $ is exactly the outward normal velocity of the zero level set of $p$. In Section \ref{sec:Literature}, we discuss literature on this widely-studied problem; for now, we mention that viscosity solutions for (\ref{eqIntro:HS1}) were introduced in  \cite{Kim-2003UniquenessAndExistenceHeleShawStefanARMA}. 

In the graph setting, the Hele-Shaw problem corresponds to  taking the gravitational constant $\rho_{\mathrm{grav}}$ to be zero in (\ref{eqIntro:rhograv}), yielding  (\ref{eqIntro:MuskatInTermsOfOmegatNormalVelocity}) but with $\phi=0$ on $\Gamma_f$, and with a different boundary condition at $x_{d+1}=-\infty$, which  is needed to ensure the solution $\phi$ is non-trivial.  There are several ways to reformulate this free boundary problem  in terms of nonlocal evolution equations. One of them, utilized in \cite{ChangLaraGuillenSchwab-2019SomeFBAsNonlocalParaboic-NonlinAnal}, is to consider the operator $I$ defined as follows. For $f:\real^d \rightarrow \real$ with $f>0$ and nice enough, let $U_f$ be the unique solution to the Dirichlet problem on a strip,
\begin{align}
\label{eqIntro:HS2}
\begin{cases}
\Delta U_f= 0\quad \text{in }D_f,\\
U_f=0 \quad\text{on }\Gamma_f,\\
U_f=1 \quad \text{on }\{x_{d+1}=0\},
\end{cases}
\end{align}
and define 
\begin{equation}
\label{eqIntro:I}
I(f,x)=\sqrt{1+|\nabla f|^2} (\partial_{(-n)} U_f)(x,f(x)). 
\end{equation}
A consequence of the results in \cite{ChangLaraGuillenSchwab-2019SomeFBAsNonlocalParaboic-NonlinAnal} (which apply to more general, nonlinear equations),  is that viscosity solutions to the integro-differential equation 
\begin{align}
\label{eqIntro:Ieqn}
	\begin{cases}
		\partial_t f = I(f)\ &\textnormal{in}\ \real^d\times(0,T)\\
		f(\cdot,0) = f_0\ &\textnormal{on}\ \real^d\times\{0\},
	\end{cases}
\end{align}
with $f_0\in BUC(\real^d)$ are well-posed and can be used in the graph setting to give viscosity solutions to (\ref{eqIntro:HS1}) that correspond to those in the sense of \cite{Kim-2003UniquenessAndExistenceHeleShawStefanARMA}. 

A second way to reformulate the Hele-Shaw problem with graph assumption in terms of integro-differential equations is to replace the potential $U_f$ by the potential $W_f$, defined in all of $D_f$ as the solution to
\begin{align}
\label{eqIntro:Wf}
	\begin{cases}
		\Delta W_f = 0 \qquad \ &\textnormal{in}\ D_f\\
		W_f = 0 \qquad \ &\textnormal{on}\ \Gam_f\\
		\norm{W_f-\ell}_{L^\infty(D_f)}<\infty, &
	\end{cases}
\end{align}
where $\ell$ is the linear function,
\begin{align}
\label{eq:linear}
	\ell(x,x_{d+1}) = -x_{d+1}.
\end{align}
The operator $H$ is defined via,
\begin{align}
\label{eqIntro:H}
	H(f,x) = \left( \sqrt{1+\abs{\grad f(x)}^2} \right)\partial_{(-n)}W_f(x,f(x)),
\end{align}
and the corresponding nonlocal Cauchy problem for $f$ is, 
\begin{align}\label{eqIntro:HeleShawHJB}
	\begin{cases}
		\partial_t f = H(f)\ &\textnormal{in}\ \real^d\times(0,T)\\
		f(\cdot,0) = f_0\ &\textnormal{on}\ \real^d\times\{0\}.
	\end{cases}
\end{align}
Since $U_f$ and $W_f$ satisfy the same boundary condition on $\Gamma_f$, this integro-differential equation  corresponds to the Hele-Shaw problem in the graph setting as well. However, the operator $H$ is  closely related to $M$, and hence more useful in the present context. Indeed, 
we note that $X\mapsto W_f(X)-\ell(X)$ solves (\ref{eqIntro:DtoNPotential}) for $g=f$, and hence by uniqueness of solutions of (\ref{eqIntro:DtoNPotential}), we have $W_f-\ell=\Phi_{f,f}$.  
Thus, 
\begin{align*}
	\partial_{n} \Phi_{f,f} = \partial_{(n)}W_f - \partial_{n}\ell
	= -\partial_{(-n)}W_f +\frac{1}{\sqrt{1+\abs{\grad f}^2}} \quad \text{on }\Gam_f.
\end{align*}
Thus, we conclude,
\begin{align}\label{eqIntro:CorrespondenceHandGff}
	M(f) = -\left(\sqrt{1+\abs{\grad f}^2}\right)\partial_{n} \Phi_{f,f}
	= H(f) - 1.
\end{align}
In this paper, we adapt the methods of \cite{ChangLaraGuillenSchwab-2019SomeFBAsNonlocalParaboic-NonlinAnal} to to establish well-posedness of (\ref{eqIntro:HeleShawHJB}). 
Then, we use relationship (\ref{eqIntro:CorrespondenceHandGff}) between $H$ and $M$ to deduce well-posedness of (\ref{eqIntro:MuskatHJB}), thus establishing our main result.

%%%%%%%%%%%%%%%%%%%%%%%%%%%%%%%%%%%%%%%%%%%%%%%%%
%%%%%%%%%%%%%%%%%%%%%%%%%%%%%%%%%%%%%%%%%%%%%%%%%
%%%%%%%%%%%%%%%%%%%%%%%%%%%%%%%%%%%%%%%%%%%%%%%%%

\subsection{Outline of the paper}

In Section \ref{sec:Literature}, we review some literature related to the Muskat problem and Theorem \ref{thmIntro:MuskatExistUnique}.  Section \ref{sec:Notation} is a collection of notation used throughout the paper.  In Section \ref{sec:AuxiliaryOperatorProperties} we prove various properties of the operator $H$.  Section \ref{sec:ViscositySolsAndMainProof} is devoted to the definitions  and properties of, as well as existence and uniqueness of, viscosity solutions for (\ref{eqIntro:MuskatHJB}) and (\ref{eqIntro:HeleShawHJB}), and concludes with the proof of Theorem \ref{thmIntro:MuskatExistUnique} (Section \ref{subsecVS:ProofOfMainTheorem}.) 
 Section \ref{sec:ComparisonWithOtherResults} contains further discussion of our results and techniques, with a focus on how they relate to existing literature.  The Appendix contains an existence and uniqueness result for harmonic functions which appears to be well-known, and yet for which there is no obvious or easily accessible reference to cite.

%%%%%%%%%%%%%%%%%%%%%%%%%%%%%%%%%%%%%%%%%%%%%%%%%
%%%%%%%%%%%%%%%%%%%%%%%%%%%%%%%%%%%%%%%%%%%%%%%%%
%%%%%%%%%%%%%%%%%%%%%%%%%%%%%%%%%%%%%%%%%%%%%%%%%

\subsection{Literature and earlier results}\label{sec:Literature}

There is a vast collection of works involving the Muskat  and  Hele-Shaw problems.  While both  have one-phase and two-phase versions, in our discussion of literature, we will focus on those that apply to \emph{one-phase} problems. Similarly, although a wide variety of techniques have been used to establish well-posedness, we focus our literature review on works that use viscosity solution techniques. 
We also refer the reader to complementary expositions in \cite[Section 2]{AbedinSchwab-2020RegularityTwoPhaseHSParrabolic-JFA} and \cite[Section 4]{ChangLaraGuillenSchwab-2019SomeFBAsNonlocalParaboic-NonlinAnal}.

\subsubsection{Viscosity solutions and integro-differential equations}
 Viscosity solutions  were first introduced in  \cite{CrandallLions-83ViscositySolutionsHJE-TAMS}, \cite{CrandallEvansLions-1984SomePropertiesOfViscositySolutionsTAMS} (similar to the approach in \cite{Evans-1980OnSolvingPDEAccretive}) to obtain existence and uniqueness to first order Hamilton-Jacobi equations (the name was an artifact of using the vanishing viscosity method to obtain existence of solutions in those original works).  Even though the notion  and definition of viscosity solutions  only relies on operators obeying the global comparison property (abbreviated as GCP; see Definition \ref{defAuxOp:GCP}), showing existence and uniqueness of such solutions is much harder.  It took some years from the time the first order Hamilton-Jacobi results appeared until existence and uniqueness was obtained for fully nonlinear second order elliptic equations in \cite{Jensen-1988UniquenessARMA}, \cite{Ishii-1989UniqueViscSolSecondOrderCPAM}.  A useful review article is \cite{CrandalIshiLions-92UsersGuide}.

Just as for the local setting,  there is  a deeply and broadly developed theory for viscosity solutions for general, \emph{nonlocal}, Hamilton-Jacobi-Bellman equations:
see  \cite{BaIm-07}, \cite{CaSi-09RegularityIntegroDiff}, \cite{GuillenMouSwiech-2019CouplingLevyMeasureAndComparisonPrincipleTAMS} \cite{JakobsenKarlsen-2006MaxPrincipleIntDiffNODEA} for a small sampling of existence and uniqueness results and  \cite{BaChIm-11Holder}, \cite{CaSi-09RegularityIntegroDiff}, \cite{ChangLaraDavila-2016HolderNonlocalParabolicDriftJDE}, \cite{Kriventsov-2013RegRoughKernelsCPDE}, \cite{SchwabSilvestre-2014RegularityIntDiffVeryIrregKernelsAPDE}, \cite{Silv-2011DifferentiabilityCriticalHJ} for a small sampling of the regularity theory.
   It turns out, thanks to a characterization of operators with the GCP in \cite{GuSc-2019MinMaxNonlocalCALCVARPDE}, \cite{GuSc-2019MinMaxEuclideanNATMA}, that under reasonable assumptions, general  operators with the GCP that act on functions in $C^{1,\gam}$   always enjoy a min-max structure with integro-differential ingredients, and thus are amenable to adaptations of viscosity solutions to the nonlocal framework, like those in \cite{BaIm-07}, \cite{CaSi-09RegularityIntegroDiff}, \cite{JakobsenKarlsen-2006MaxPrincipleIntDiffNODEA}, \cite{Silv-2011DifferentiabilityCriticalHJ}, where such a min-max structure is always assumed.  Whether or not one can actually obtain existence and uniqueness for these solutions  depends upon more detailed properties of the operator, beyond just the GCP.

\subsubsection{The Hele-Shaw problem without gravity}
\label{sss:HS}

Some of the earliest works for short time existence and uniqueness of the Hele-Shaw problem without gravity  are \cite{ElliottJanovsky-1981VariationalApproachHeleShaw} and \cite{EscherSimonett-1997ClassicalSolutionsHeleShaw-SIAM}:  a type of variational problem is studied in \cite{ElliottJanovsky-1981VariationalApproachHeleShaw} and a classical solution (for short time) is produced in \cite{EscherSimonett-1997ClassicalSolutionsHeleShaw-SIAM}.  For the one-phase problem, under a smoothness and convexity assumption, \cite{DaskaLee-2005AllTimeSmoothSolHeleShawStefan-CPDE} gives global in time smooth solutions. 

Since the  Hele-Shaw problem without gravity   enjoys a comparison principle, it is amenable to viscosity solution methods: well-posedness for viscosity solutions of  (\ref{eqIntro:HS1}) was established in \cite{Kim-2003UniquenessAndExistenceHeleShawStefanARMA}. For the closely related Stefan problem, this was done in \cite{AthanaCaffarelliSalsa-1996RegFBParabolicPhaseTransitionACTA}; subsequent properties were studied in \cite{KimPozar-2011ViscSolTwoPhaseStefanCPDE}.  These notions of viscosity solutions for the Stefan and Hele-Shaw problems are generalizations of those for the two-phase stationary problem given in \cite{Caffarelli-1987HarnackInqualityApproachFBPart1RevMatIbero}.

In \cite{ChangLaraGuillenSchwab-2019SomeFBAsNonlocalParaboic-NonlinAnal}, it was proved that viscosity solutions of a rather large class of one-phase and two-phase problems under the graph condition, including the Hele-Shaw problem without gravity, are equivalent to viscosity solutions of the Hamilton-Jacobi-Bellman equation for the function that parameterizes the graph of the free boundary. (A priori, these are two different notions of viscosity solution.)  Existence, uniqueness, and propagation of modulus of continuity of solutions was also proved in \cite{ChangLaraGuillenSchwab-2019SomeFBAsNonlocalParaboic-NonlinAnal}.

Beyond the smooth initial data case in \cite{EscherSimonett-1997ClassicalSolutionsHeleShaw-SIAM}, and the convex case in \cite{DaskaLee-2005AllTimeSmoothSolHeleShawStefan-CPDE}, there are a number of works on regularity of the Hele-Shaw problem without gravity  (\ref{eqIntro:HS1}).    Under a non-degeneracy assumption on the free boundary, \cite{Kim-2006RegularityFBOnePhaseHeleShaw-JDE} showed that  Lipschitz free boundaries become $C^1$ in space-time. Long time regularity, involving propagation of a Lipschitz modulus, was obtained in \cite{Kim-2006LongTimeRegularitySolutionsHeleShaw-NonlinAnalysis}.  Subsequently, the extra condition on the space-time non-degeneracy in \cite{Kim-2006RegularityFBOnePhaseHeleShaw-JDE} was removed in the work of \cite{ChoiJerisonKim-2007RegHSLipInitialAJM}, where it was shown that, under a dimensionally small Lipschitz condition on the initial free boundary, Lipschitz free boundaries must be $C^1$ in space-time and hence classical.  More precise results were proved in  \cite{ChoiJerisonKim-2009LocalRegularizationOnePhaseINDIANA} in the case that  the solution starts from a global Lipschitz graph.  We note that the restriction on the size of the Lipschitz norm of the free boundary is indeed necessary for a \emph{regularizing} result like \cite{ChoiJerisonKim-2007RegHSLipInitialAJM} to hold: as shown in \cite{KingLaceyVazquez-1995PersistenceOfCornersHeleShaw},  there are pathological solutions for which some free boundary points remain stationary for some time before immediately jumping into motion.   

In \cite{ChangLaraGuillen-2016FreeBondaryHeleShawNonlocalEqsArXiv}, local $C^{1,\gamma}$ space-time regularity for  solutions of  (\ref{eqIntro:HS1}) was established under  a space-time flatness condition on the free boundary. This was done using blow-up limits, similar to the strategy of \cite{Savin-2007SmallPerturbationCPDE} and \cite{DeSilva-2011FBProblemWithRHS}, whereas \cite{Kim-2006RegularityFBOnePhaseHeleShaw-JDE},  \cite{ChoiJerisonKim-2007RegHSLipInitialAJM},\cite{ChoiJerisonKim-2009LocalRegularizationOnePhaseINDIANA} employed techniques originating in \cite{Caffarelli-1987HarnackInqualityApproachFBPart1RevMatIbero}, \cite{Caffarelli-1989HarnackForFBFlatAreLipCPAM}.  The integro-differential structure of a two-phase version of the Hele-Shaw problem under the graph assumption was used in \cite{AbedinSchwab-2020RegularityTwoPhaseHSParrabolic-JFA} to show that solutions that are $C^{1,\dini}$ must in fact become $C^{1,\gam}$ for a universal $\gam$.

\subsubsection{The one-phase Muskat problem}
\label{ssIntro:Muskat}
A number of works  address the well-posedness of the one-phase Muskat problem by reducing it to the equation for the graph of the boundary, (\ref{eqIntro:MuskatHJB}),  including \cite{Alazard-2021ConvexityAndHeleShaw-WaterWaves}, \cite{AlazardMeunierSmets-2020LyapounovCauchyProbHeleShaw-CMP}, \cite{DongGancedoNguyen-2023MuskatWP3D-arXiv}, \cite{DongGancedoNguyen-2023MuskatOnePhase2DWellPose-CPAM},  \cite{NguyenPausader-2020ParadifferentialWellPoseMuskatARMA}.  The works \cite{Alazard-2021ConvexityAndHeleShaw-WaterWaves}, \cite{AlazardMeunierSmets-2020LyapounovCauchyProbHeleShaw-CMP}, \cite{NguyenPausader-2020ParadifferentialWellPoseMuskatARMA} established local in time well-posedness of classical solutions for periodic initial data in $H^s$ with $s>1+\frac{d}{2}$.  (Note that if $f\in H^s$ for $s>1+\frac{d}{2}$ then $f\in C^{1,\gam}$, for any $\gam>0$, and hence $M(f)$ is classically defined.) These works 
 focused on the parabolicity of and energy estimates for $M(f)$, creating short time  solutions. 
    Global well-posedness of viscosity solutions of (\ref{eqIntro:MuskatHJB}) for Lipschitz and periodic initial data was established in \cite{DongGancedoNguyen-2023MuskatWP3D-arXiv}, \cite{DongGancedoNguyen-2023MuskatOnePhase2DWellPose-CPAM}.  There is a very recent result in \cite{AlazardKoch-2023HeleShawSemiFlow-ArXiV} that gives global in time smooth solutions for regular data as well as well-posedness of variational solutions to a closely related time-dependent obstacle problem with merely bounded data.

	The one-phase Muskat problem is closely connected with a family of equations called Hele-Shaw with a drift and a source. A few works that address the connection of Hele-Shaw with a drift with variations of the Porous Medium equation are \cite{AlexanderKimYao-QuasiStaticAndCongestedCrowdControl-Nonlinearity2014}, \cite{KimPozarWoodhouse-SingularLimitPMEWithDrift-2019AdvMath}, \cite{KimZhang-PorousMediumWithDriftRegularity-ARMA2021}.  Some regularity results for the free boundary for Hele-Shaw with a drift are in \cite{KimZhang-RegularityHeleShawWithDrift-2022ArXiV}.

\subsubsection{The two-phase Muskat problem}
The two-phase version of (\ref{eqIntro:MuskatHJB}) is  more complicated  than the one-phase version, and there are a variety of methods that have been used to address well-posedness and regularity.  
The reformulation of the two-phase Muskat problem in terms of integro-differential equations goes back to \cite{Ambrose-2004WellPoseHeleShawWithoutSurfaceTensionEurJAppMath},  \cite{CaflischOrellanaSiegel-1990LocalizedApproxMethodVorticalFLows-SIMA}, \cite{CaflischHowisonSiegel-2004GlobalExistenceSingularSolMuskatProblem-CPAM}. Global existence of solutions with small data was shown in \cite{CaflischHowisonSiegel-2004GlobalExistenceSingularSolMuskatProblem-CPAM} and short time existence of solutions with large data in an appropriate Sobolev space was established in \cite{Ambrose-2004WellPoseHeleShawWithoutSurfaceTensionEurJAppMath}.  In \cite{CordobaGancedo-2007ContourDynamics-CMP},  the Muskat problem was formulated as an  integro-differential  equation for the  \emph{gradient} of the free surface function. This formulation was then used to show that near a sufficiently regular stable solution, the equation linearizes to the 1/2-heat equation, and \cite{CordobaGancedo-2007ContourDynamics-CMP}  showed existence of solutions in this regime. This formulation was subsequently used to establish many well-posedness and regularity results, including \cite{ConstantinCordobaGancedoStrain-2013GlobalExistenceMuskat-JEMS}, \cite{ConstantinGancedoShvydkoyVicol-2017Global2DMuskatRegularity-AIHP}, \cite{CordobaCordobaGancedo-2011InterfaceHeleShawMuskat-AnnalsMath}.

The next batch of results that exploit the integro-differential structure are
\cite{AlazardNguyen-2021CauchyProbMuskatNonLipData-CPDE},
\cite{AlazardNguyen-2021OnCauchyProbMuskatPart2CriticalData-AnnPDE},
\cite{Cameron-2019WellPosedTwoDimMuskat-APDE}, 
\cite{Cameron-2020WellPose3DMuskatMediumSlope-ArXiv}, 
\cite{CordobaGancedo-2009MaxPrincipleForMuskat-CMP}.   In \cite{CordobaGancedo-2009MaxPrincipleForMuskat-CMP}, the Muskat problem  was rewritten as a fully nonlinear integro-differential equation on $f$ itself, instead of $\partial_x f$ (as previously), and in  $d=1$ has  the form 
\begin{align*}
	\partial_t f = \int_\real \del_y f(x,t)K_f(y,t)dy.
\end{align*}
Here $K_f\geq0$ is a kernel that depends on $f$ and has the same structure as operators with the GCP and the Hamilton-Jacobi-Bellman equations mentioned above.  The integro-differential equation  for $f$ (as opposed to $\partial_x f$) was used in \cite{CordobaGancedo-2009MaxPrincipleForMuskat-CMP} to show non-expansion of the Lipschitz norm of solutions with nice enough data.   This structure was subsequently utilized in \cite{Cameron-2019WellPosedTwoDimMuskat-APDE}, \cite{Cameron-2020WellPose3DMuskatMediumSlope-ArXiv} to study well-posedness for Lipschitz data as well as regularizing effects. The nonlinear parabolic structure given by the nonlinear Dirichlet-to-Neumann operator was used for short time well-posedness in \cite{NguyenPausader-2020ParadifferentialWellPoseMuskatARMA}. We conclude by mentioning the recent works \cite{zlatos20242dI} and \cite{zlatos20242dII} that establish local-in-time well-posedness, and study singularities, of the two-phase Muskat problem on the half plane.

%%%%%%%%%%%%%%%%%%%%%%%%%%%%%%%%%%%%%%%%%%%%%%%%%
%%%%%%%%%%%%%%%%%%%%%%%%%%%%%%%%%%%%%%%%%%%%%%%%%
%%%%%%%%%%%%%%%%%%%%%%%%%%%%%%%%%%%%%%%%%%%%%%%%%
%%%%%%%%%%%%%%%%%%%%%%%%%%%%%%%%%%%%%%%%%%%%%%%%%
%%%%%%%%%%%%%%%%%%%%%%%%%%%%%%%%%%%%%%%%%%%%%%%%%
%%%%%%%%%%%%%%%%%%%%%%%%%%%%%%%%%%%%%%%%%%%%%%%%%
%%%%%%%%%%%%%%%%%%%%%%%%%%%%%%%%%%%%%%%%%%%%%%%%%
%%%%%%%%%%%%%%%%%%%%%%%%%%%%%%%%%%%%%%%%%%%%%%%%%
%%%%%%%%%%%%%%%%%%%%%%%%%%%%%%%%%%%%%%%%%%%%%%%%%

\section{Notation}\label{sec:Notation}

Here we collect some notation that will be used throughout the paper.

\begin{itemize}
	
	\item We write a point in $\mathbb{R}^{d+1}$ as $X = (x,x_{d+1})\in \mathbb{R}^d\times \mathbb{R}$. We will use  upper case letters to denote points in $\mathbb{R}^{d+1}$, and lower case letters to denote points in $\mathbb{R}^d$. The $(d+1)$-th unit vector in $\R^{d+1}$ is denoted by $e_{d+1}$.

	\item For $X\in \mathbb{R}^{d+1}$ and $r>0$, we denote by $B_r(X)$ the open ball in $\mathbb{R}^{d+1}$ of radius $r$ centered at $X$.  When it is important to distinguish between a ball in $\real^{d+1}$ and $\real^d$, we will use respectively $B_r^{d+1}(X)$ and $B_r^d(x)$.
	
	\item We use $ BUC(\real^d)$ to denote the set of bounded and uniformly continuous functions on $\real^d$.
	    \item For a function $f:\real^d\rightarrow \real$, we denote its subgraph by $D_f$ and its graph by $\Gamma_f$: see \eqref{eqIntro:DfAndGamf}.

    \item For $Y = (y,f(y))\in \Gamma_f$, we use  $n =n_f= n(y) = n(Y)$ to denote the outward unit normal vector to $D_f$ at $Y$, as defined in (\ref{eqIntro:NormalVector}). We denote the inward unit normal vector to $D_f$ by
    \begin{equation}
    \label{eq:nu}
    \nu =\nu_f=\nu(y) = \nu(Y):= -n(Y).
    \end{equation} 	
	
    \item For $l\in \mathbb{N}$ and $\Omega \subset \mathbb{R}^l$, we denote by
    $C^k(\Omega)$ the set of functions having all derivatives of order less than or equal to $k$ continuous in $\Omega$ for $k\in \mathbb{N}$ or $k=\infty$.

    \item For $\gamma\in (0,1]$ we denote the H\"older space as follows:
    \begin{align*}
        C^{1,\gamma}(\Omega) = \left\lbrace
        f:\Omega \to \mathbb{R}: \Vert f \Vert_{L^\infty(\Omega)} + \Vert \nabla f\Vert_{L^\infty(\Omega)} + [\nabla f]_{C^{0,\gamma}(\Omega)} < \infty 		 \right\rbrace		
    \end{align*}	
    where 
    \begin{align*}
        [\grad f]_{C^{0,\gamma}(\Omega)} = \sup \left\lbrace \frac{ |\grad f(x)- \grad f(y)|}{|x-y|^\gamma}: x,y\in \Omega, x\neq y \right\rbrace.
    \end{align*}

\item The function spaces $\K(\gam,m)$, $\K^*(\gam,m)$, and $\K(\gam,m, x_0, a, p)$
 are defined, respectively, in \eqref{eqAuxOp:TheSetMatchcalK}, \eqref{eq:K*}, and \eqref{eqPtWise:PtWiseKGamM}.

    \item For $x\in \mathbb{R}^d$ and a function $u:\mathbb{R}^d\to \mathbb{R}$, we define the translation operator $\tau_x$ that acting on $u$ as $\tau_xu(\cdot) = u(\cdot + x)$.

	\item The Global comparison property (GCP) is in Definition \ref{defAuxOp:GCP}.
	\item The inf and sup convolutions of a function $u$ are denoted by $u_\ep
	$ and $u^\ep$: see Definition \ref{defViscSol:InfSupConvoution}.
	\item The upper and lower semicontinuous envelopes of a function $u$ are denoted $u^*$ and $u_*$: see Definition \ref{defViscSol:USCLSCEnvelopes}.

\end{itemize}

%%%%%%%%%%%%%%%%%%%%%%%%%%%%%%%%%%%%%%%%%%%%%%%%%
%%%%%%%%%%%%%%%%%%%%%%%%%%%%%%%%%%%%%%%%%%%%%%%%%
%%%%%%%%%%%%%%%%%%%%%%%%%%%%%%%%%%%%%%%%%%%%%%%%%
%%%%%%%%%%%%%%%%%%%%%%%%%%%%%%%%%%%%%%%%%%%%%%%%%
%%%%%%%%%%%%%%%%%%%%%%%%%%%%%%%%%%%%%%%%%%%%%%%%%
%%%%%%%%%%%%%%%%%%%%%%%%%%%%%%%%%%%%%%%%%%%%%%%%%
%%%%%%%%%%%%%%%%%%%%%%%%%%%%%%%%%%%%%%%%%%%%%%%%%
%%%%%%%%%%%%%%%%%%%%%%%%%%%%%%%%%%%%%%%%%%%%%%%%%
%%%%%%%%%%%%%%%%%%%%%%%%%%%%%%%%%%%%%%%%%%%%%%%%%

\section{The Hele-Shaw operator, $H$, and its properties}\label{sec:AuxiliaryOperatorProperties}

In this section we give the precise definition  and list of properties of the operator $H$ introduced in (\ref{eqIntro:H}).    These will then be used for well-posedness of (\ref{eqIntro:HeleShawHJB}) in Section \ref{sec:ViscositySolsAndMainProof}.

\subsection{Function spaces}
We begin by introducing several function spaces that we will employ throughout the work.  For $\gamma\in (0,1)$ and $m>0$, we define the open convex set 
    \begin{align}
    \label{eqAuxOp:TheSetMatchcalK}
        \mathcal{K}(\gamma,  m) := \left\lbrace f\in C^{1,\gamma}(\mathbb{R}^d):  \Vert f\Vert _{C^{1,\gamma}(\mathbb{R}^d)} < m\right\rbrace.
	\end{align}
If at every point, a function enjoys a $C^{1,\gam}$ expansion from above, i.e.
\begin{align}\label{eqAuxOp:C1GamFromAbove}
	\textnormal{there exists}\ p,\ \textnormal{with}\ \abs{p}\leq m,\ \ \textnormal{and}\ \ 
	f(x+h)  \leq f(x) + p\cdot h + m\abs{h}^{1+\gam},
\end{align}
then we call it ``$C^{1,\gam}$-semi-concave''. 
We will denote the set of functions satisfying (\ref{eqAuxOp:C1GamFromAbove}) by $\K^*(\gam,m)$:
\begin{align}
\label{eq:K*}
	\K^*(\gam,m) = \left\lbrace f\in C^{0,1}(\real^d)\ :\ \norm{f}_{C^{0,1}}\leq m\ \textnormal{and}\  (\ref{eqAuxOp:C1GamFromAbove})\ \textnormal{holds for}\ f\ \textnormal{and}\ m  \right\rbrace.
\end{align}

\begin{definition}\label{defAuxOp:PunctuallyC1Gam}
	We say that $f$ is \emph{punctually $C^{1,\gam}$ at $x$}, which we denote as $f\in \left( m\text{-} C^{1,\gam}(x) \right)$, provided that $f$ is differentiable at $x$ and there exists $r>0$, $p\in\real$, with
	\begin{align*}
		\textnormal{for}\ \ \abs{y}\leq r,\ \ 
		\abs{f(x+y)-f(x)-\grad f(x)\cdot y}\leq m\abs{y}^{1+\gam}.
	\end{align*}
\end{definition}

Finally, we will want to list certain subsets of $\K(\gam,m)$ in which the functions are punctually $C^{1,\gam}$ and  share the same value and the same gradient at a particular point.  To this end, we use the  notation:
\begin{align}\label{eqPtWise:PtWiseKGamM}
	\K(\gam,m,x_0, a, p) = 
	\left\lbrace f\in \K^*(\gam,m)\ :\ f\ \textnormal{is}\ m\text{-}C^{1,\gam}(x_0),\ \textnormal{with}\ 
	f(x_0)=a,\ \grad f(x_0)=p 
	\right\rbrace.
\end{align}

\subsection{The basic properties}\label{sec:BasicPropertiesH}

We give a rigorous definition of $W_f$ introduced in (\ref{eqIntro:Wf}). The properties we collect here are modifications and improvements of the ones for a related pressure function studied in \cite[Section 5]{ChangLaraGuillenSchwab-2019SomeFBAsNonlocalParaboic-NonlinAnal}, of which $U$ defined in (\ref{eqIntro:HS2}) is a special case. Since the  results in \cite{ChangLaraGuillenSchwab-2019SomeFBAsNonlocalParaboic-NonlinAnal} concern a pressure that solves a fully nonlinear uniformly elliptic equation (with a slightly different lower boundary condition), while here $W_f$  is harmonic,  many of the arguments simplify and can be improved upon considerably.

\begin{definition}\label{defAuxOp:DefOfWf}
	Given $f\in C^{0,1}(\real^d)$, let $W_f\in C^2(D_f)\intersect C^{0}(\overline{D}_f)$ be the unique solution of 
	\begin{align}\label{eqAuxOp:PotentialWf}
\begin{cases}
		\Delta W_f = 0\ &\textnormal{in}\ D_f\\
		W_f = 0\ &\textnormal{on}\ \Gam_f \ \textnormal{and}\ \norm{W_f-\ell}_{L^\infty(D_f)}<\infty,
\end{cases}
        \end{align}
where $\ell(x,x_{d+1})=-x_{d+1}$.  (Existence and uniqueness of $W_f$ are listed in Propositions  \ref{propAppendix:weak-max-principle}, \ref{propAppendix:existence-Perron}.)

\end{definition}

\begin{rem}
	If there exists $r>0$ and $X\in\Gam_f$ with $f\in C^{1,\gam}(B_r(x))$, then $W_f\in C^{1,\gam}(B_{r/2}(X)\intersect D_f)$ (recall that $X=(x,f(x))$, and we note the implicit use of the same notation for $B_r$ in both $\real^d$ and $\real^{d+1}$).  
\end{rem}

We  now define the operator $H$, introduced in (\ref{eqIntro:H}).

\noindent

\begin{definition}\label{defAuxOp:TheOperatorH}
The operator, $H: C^{1,\gam}(\real^d)\to C^{0}(\real^d)$, is defined for $f\in C^{1,\gam}$ and $x\in\real^d$  as
\begin{align}\label{eqCheckAssumpt:DefOfHeleShawOp}
	H(f,x) =\left(\sqrt{1+\abs{\grad f}^2}\right) \partial_{\nu_f} W_f(x, f(x)),
\end{align}
where $\nu_f$ is the inward normal as in (\ref{eq:nu}).

\end{definition}

%%%%%%%%%%%%%%%%%%%%%%%%%%%%%%%%%%%%%%%%%%%%%%%%%
%%%%%%%%%%%%%%%%%%%%%%%%%%%%%%%%%%%%%%%%%%%%%%%%%
%%%%%%%%%%%%%%%%%%%%%%%%%%%%%%%%%%%%%%%%%%%%%%%%%
%%%%%%%%%%%%%%%%%%%%%%%%%%%%%%%%%%%%%%%%%%%%%%%%%

The first property is basic local regularity for $H$ in $C^{1,\gam}$, which follows standard local estimates for harmonic functions combined with the definition of $H$, and so we omit its proof.

\begin{proposition}\label{propAuxOp:BaiscRegularityForHC1GamToCgam}
	If $f\in C^{0,1}(\real^d)\intersect C^{1,\gam}(B_{2R}(x_0))$, then $H(f,\cdot) \in C^\gam(B_R(x_0))$.
\end{proposition}

One of the most important properties of $H$ and $M$ for our work is what we call the global comparison property (GCP). Roughly, to say that an operator, $J$, satisfies the GCP means that $J$  preserves the ordering of  functions on $\mathbb{R}^d$ at any points where their graphs coincide. See Section \ref{ss:GCP} for a discussion of the role of the GCP in other works.

\begin{definition}[Global comparison property]\label{defAuxOp:GCP}
     We say an operator $J: C^{1,\gamma}(\mathbb{R}^d)\to C^{0}(\mathbb{R}^d)$ has the \emph{global comparison property} (hereafter, GCP) if the following holds:
         \begin{align*}
        &\textnormal{ for any } x_0\in \R^d, u,v\in C^{1,\gamma}(\mathbb{R}^d) \textnormal{  such that }u\leq v \textnormal{ in }\mathbb{R}^d \textnormal{ and } u(x_0) = v(x_0),\\
		 &\quad \textnormal{we have } J(u,x_0)\leq J(v,x_0).
    \end{align*}
    We say $J$ satisfies GCP at $x_0$ or with respect to $x_0$ if the above holds for a fixed $x_0\in \mathbb{R}^d$ only.
\end{definition}

Next, we list some basic properties of $W_f$ and $H$.

\begin{proposition}[Basic Properties]\label{propAuxOp:HProperties} Let $W_f$ and $H$ be as in Definitions \ref{defAuxOp:DefOfWf} and \ref{defAuxOp:TheOperatorH}.
	
	\begin{enumerate}[(i)]
		\item $f\mapsto W_f$ is monotonic: if $f,g\in C^{0,1}(\real^d)$ and $f\leq g$, then  $W_f\leq W_g$ holds in $D_f$.
		\item $H$ enjoys the global comparison property.
		\item $H$ is translation invariant: if $f\in C^{1,\gam}(\real^d)$, then 
		\begin{align}\label{eqAuxOp:HPropertiesTranslationInvariance}
			H(\tau_z f, x) = H(f, x+z), \qquad x,z\in \mathbb{R}^d.
		\end{align}
		\item $H$ is invariant under  addition of constants: if $f\in C^{1,\gam}(\real^d)$ and $c\in \real$, then
		\begin{align}\label{eqAuxOp:HPropertiesInvarianceByAddConstant}
			H(f+c, x) = H(f, x), \qquad x\in \mathbb{R}^d, c\in \mathbb{R}.
		\end{align}
	\end{enumerate}
\end{proposition}

\begin{proof}[Proof of Proposition \ref{propAuxOp:HProperties}]

	To establish (i), we note that $f \leq g$ 
	implies $D_f \subseteq D_g$. By definition $W_g$, 
	\begin{align*}
		W_g \geq 0\ \ \text{in}\ \ D_g.
	\end{align*}    
    As a consequence $W_g \geq 0 = W_f$ on $\Gamma_f$, and so
    Proposition 
    \ref{propAppendix:weak-max-principle} gives $	W_g - W_f \geq 0\ \  \text{in}\ \ D_f$. 	To verify the GCP, suppose $f,g\in 
 	C^{1,\gamma}(\mathbb{R}^d)$ are such that $f\leq g$ in $\mathbb{R}^d$ 
 	and there exists $x_0\in \mathbb{R}^d$ with $f(x_0) = g(x_0)$. We 
 	shall show that
    \begin{align*}
        H(f,x_0) \leq H(g,x_0).
    \end{align*}
    Let $X_0 = (x_0,f(x_0)) = (x_0,g(x_0)) \in \Gamma_f\cap \Gamma_g$. 		
    Since $\nabla f(x_0) = \nabla g(x_0)$ we have $\nu_f(X_0) = 
    \nu_g(X_0)$ and $X_0 + s\nu_f(X_0) \in D_f$ for $s > 0$. Using the 
    fact that $W_f\leq W_g$ we have
    \begin{align*}
        W_f(X_0) = W_g(X_0), \qquad\text{and}\qquad 
        W_f(X_0 + s\nu_f(X_0)) \leq W_g(X_0 + s\nu_f(X_0))
    \end{align*}
    for $s>0$. Therefore
    \begin{align*}
        \frac{W_f(X_0 + s\nu_f(X_0)) - W_f(X_0)}{s} \leq  \frac{W_g(X_0 + 
        s\nu_f(X_0)) - W_g(X_0)}{s} \qquad \text{for all}\;s > 0.
    \end{align*}
    Letting $s\to 0^+$ we obtain $\partial_{\nu_f}W_f(X_0) \leq 
    \partial_{\nu_f}W_g(X_0)$. Thus from the definition of $H$ in \eqref{eqCheckAssumpt:DefOfHeleShawOp} we obtain $H(f,x_0) \leq H(g,x_0)$ thanks to $\nabla f(x_0) = \nabla g(x_0)$.

	Next we verify (\ref{eqAuxOp:HPropertiesTranslationInvariance}). To 
	this end, fix $f\in C^{1,\gam}(\mathbb{R}^d)$, $z\in \mathbb{R}^d$, and 
	let $W_f$ solve \eqref{eqAuxOp:PotentialWf}. Let
	\begin{align*}
		D_{\tau_zf} = \{ (x,x_{d+1}) \in \mathbb{R}^d\times \mathbb{R}: 
			x_{d+1} < f(x+z)\},
	\end{align*}
	and its boundary $\Gamma_{\tau_z f} = \{ (x, f(x+z)): x\in \mathbb{R}
	^d \}$. We then define
    \begin{align*}
        (\tau_z W_f) (x,x_{d+1}) := W_f(x+z, x_{d+1}), \qquad (x,x_{d+1}) 
        		\in D_{\tau_z f}.
    \end{align*}
    This is well-defined since $(x,x_{d+1}) \in D_{\tau_z f}$ provided 
    that $(x+z, x_{d+1}) \in D_f$. It follows that 
    \begin{align*}
    \begin{cases}
    	\Delta (\tau_z W_f) = 0\ &\textnormal{in}\ D_{\tau_z f}\\
    		\tau_z W_f = 0\ &\textnormal{on}\ \Gamma_{\tau_z f} \ 
    		\textnormal{and}\ \norm{\tau_z W_f-\ell}_{L^\infty(D_{\tau_z f})}
    		<\infty. 
    \end{cases}
    \end{align*}
    By uniqueness, we 
    obtain $\tau_z 
    		W_f(x) = W_{\tau_z f}(x)$
    for $x\in \mathbb{R}^d$, from which we conclude,
    \begin{align*}
    H(f(\cdot+z), x) = H(f, x+z), \qquad \textnormal{for all}\;x\in 
    \mathbb{R}^d,
		\end{align*}
	as desired.

	To verify (\ref{eqAuxOp:HPropertiesInvarianceByAddConstant}), let 
	$\tilde{f}(x) = f(x) + c$ for $x\in \mathbb{R}^d$. Let
	 \begin{align*}
 		W(x,x_{d+1}) = W_{\tilde{f}}(x, x_{d+1}+c), \qquad (x,x_{d+1})\in 
 		D_f.
	 \end{align*}
	 This is well-defined since $(x, x_{d+1}+c) 
	\in D_{\tilde{f}}$ if and only if $(x,x_{d+1}) \in D_f$. By uniqueness 
	for \eqref{eqAuxOp:PotentialWf} we have $W\equiv W_f$ in $D_f$, i.e.,
	\begin{align*}
    		 W_{\tilde{f}}(x,x_{d+1}+c) = 
    		 W_f(x,x_{d+1}), \qquad\textnormal{if}\; x_{d+1} < f(x).
	\end{align*}
	Hence, it follows that for $(x,f(x))\in\Gam_f$,
	\begin{align*}
		\partial_{\nu_{ \tilde f}} W_{\tilde f}(x, \tilde{f}(x)) = 
		\partial_{\nu_f}W_f(x,f(x		)),
	\end{align*}
	which implies $H(\tilde f, x) = H(f,x)$, concluding the proof.
\end{proof}

The next lemma regarding a barrier function is straightforward from the local $L^\infty$-$C^{1,\gam}$ estimates for harmonic functions, and we do not include a proof.

\begin{lemma}\label{lemAuxOp:BarrierFuns}
	If $R>0$, $w\in \K(\gam,m)$, and the function $b=b_{w,R}$ is defined as the unique solution of
	\begin{align*}
		\begin{cases}
			\Delta b = 0\ &\text{in}\ D_w\intersect B_R\\
			b=0\ &\text{on}\ \Gam_w\intersect B_{R/2}\\
			b(X) = \textnormal{smooth and radially increasing} &\text{on}\ \Gam_w \intersect \left( B_R\setminus B_{R/2} \right)\\
			b(X) = 1\ &\text{on}\ D_w\intersect \partial B_{R},
		\end{cases}
	\end{align*}
	then there exists a positive constant $C=C(\gam,m,d,R)$, so that
	\begin{align*}
		\norm{b}_{C^{1,\gam}(B_{R/4})}\leq C.
	\end{align*}
\end{lemma}

\begin{rem}\label{remAuxOp:TheWForTheBarrierLemma}
	A typical choice of $w$ in Lemma \ref{lemAuxOp:BarrierFuns} would be a $C^{1,\gam}$ approximation of the function, for $c_1,c_2, c_3\in\real$, $x_0\in\real^d$, and $p\in\real^d$,
	\begin{align*}
		\phi(x) = \pm \min\{ c_1 + p\cdot(x-x_0) + c_2\abs{x-x_0}^{1+\gam},   c_3  \}.
	\end{align*}
	For example, if $f\in\K^*(\gam,m)$, then $\phi$, and the subsequent $w$, can be chosen so that
	\begin{align*}
		w\in\K(\gam, 2m),\ \ w\geq f,\ \ \textnormal{and}\ \ w(x_0) = f(x_0).
	\end{align*}
\end{rem}

These barrier functions can be used to show that for $f\in \K^*(\gam,m)$, $W_f$ enjoys global Lipschitz estimates.  First, we show a $L^\infty$ bound on $W_f$.

\begin{lemma}\label{lemAuxOp:UniformLInfinityBounds}
	If $f\in C^{0,1}(\real^d)$ and $W_f$ is as in (\ref{eqAuxOp:PotentialWf}), then, for $\ell(x,x_{d+1})=-x_{d+1}$,
	\begin{align*}
		\ell + \inf_{\real^d} f \leq W_f \leq \ell + \sup_{\real^d} f
	\end{align*}
\end{lemma}

\begin{proof}[Proof of Lemma \ref{lemAuxOp:UniformLInfinityBounds}]
	We see that for $c_1 = \inf_{\real^d} f$ and $c_2 = \sup_{\real^d} f$, 
	\begin{align*}
		W_{c_1} = \ell + c_1\ \ \ \text{and}\ \ \ 
		W_{c_2} = \ell + c_2.
	\end{align*} 
	Furthermore, the monotonicity of $W$ in Proposition \ref{propAuxOp:HProperties} gives
	\begin{align*}
		W_{c_1}\leq W_f \leq W_{c_2},
	\end{align*}
	which is what was claimed.
\end{proof}

\begin{lemma}\label{lemAuxOp:WfLipOnK*Part1}
	If $f\in \K^*(\gam,m)$, then there exists a constant, $C=C(\gam,m,d)$, so that for any unit vector, $e\in \real^{d+1}$,
	\begin{align*}
		\textnormal{for all}\ \ X\in D_f,\textnormal{and}\ h\in\real\ \textnormal{such that}\  X+he\in D_f,\ \ 
		\abs{W_f(X+he) - W_f(X)}\leq C\abs{h}.
	\end{align*}
\end{lemma}

\begin{proof}[Proof of Lemma \ref{lemAuxOp:WfLipOnK*Part1}]

	Let $e = (\hat{e}, e_{d+1}) \in \mathbb{R}^d\times\mathbb{R}$. We can write	
	\begin{align}\label{eq:lemAuxOp-W-shift}
		W_f(X+he) = W_{\hat{f}}(X),\qquad X\in D_{\hat{f}},
	\end{align}
	where
	\begin{align*}
		\hat f = \tau_{h \hat{e} }f - h e_{d+1}, \qquad\text{i.e.,}\qquad \hat{f}(x) = f(x+h\hat{e}) - he_{d+1}.
	\end{align*}
        Indeed, we have
        \begin{align*}
            (x,x_{d+1}) \in D_{\hat{f}} &\qquad\Longleftrightarrow \qquad x_{d+1} < \hat{f}(x) = f(x+h\hat{e}) - he_{d+1} \\
            &\qquad\Longleftrightarrow \qquad x_{d+1} + he_{d+1} < f(x+he)\\
            &\qquad\Longleftrightarrow \qquad (x + h\hat{e}, x_{d+1} + he_{d+1}) \in D_f \\ 
            &\qquad\Longleftrightarrow \qquad X + he \in D_f. 
        \end{align*}
        Therefore we can define
        \begin{align*}
            W(X) = W(x,x_{d+1}) = W_f(X+he), \qquad X\in D_{\hat{f}}.
        \end{align*}
        Then $W$ is harmonic in $D_{\hat{f}}$ with $W = 0$ on $\Gamma_{\hat{f}}$, hence by uniqueness $W \equiv W_{\hat{f}}$ and thus \eqref{eq:lemAuxOp-W-shift} follows.

		We note that by Lemma \ref{lemAuxOp:UniformLInfinityBounds}, there is a constant that depends on $m$, so that
		\begin{align*}
			&\textnormal{for}\ X\in 
			\{ (x,x_{d+1})\ :\ \min\{ f(x), \hat f(x) \} - 1 \leq x_{d+1} \leq f(x) \},\\ 
			&\ \ \ \ \ 0\leq W_f(X) \leq C(m)
		\end{align*}
		and similarly for $\hat f$.  Thus, given any $X_0\in \Gam_f$, using Remark \ref{remAuxOp:TheWForTheBarrierLemma} and Lemma \ref{lemAuxOp:BarrierFuns}, there exist $m'$ (depending only $m$),   $w\in K(\gam,m')$, and a barrier $b_{w,x_0}$, so that
		\begin{align*}
			0\leq W_f(X)\leq C(m)b_{w,x_0}(X).
		\end{align*}
		The regularity of $b_{w,x_0}$ then implies that $W_f$ grows at most linearly away from its zero set, $\Gam_f$.  A similar statement holds for $\hat f$.

	Since $f$ is Lipschitz with Lipschitz constant $m$, we see that $\abs{f-\hat f}\leq mh$. We can now use the linear growth of $W_f$ and $W_{\hat f}$ to get a bound on the values of $W_f$ and $W_{\hat f}$ on the boundary of $D_f\intersect D_{\hat f}$. That is to say that 
	\begin{align*}
		\textnormal{for}\ X\in \partial(D_f\intersect D_{\hat f}),\ \ 
		\abs{W_f(X)-W_{\hat f}(X)}\leq C(m)\abs{f(x)-\hat f(x)}.
	\end{align*} 
	To this end, we define the function $V$ via
	\begin{align*}
		V = W_f - W_{\hat f}\ \ \text{in}\ \ D_f\intersect D_{\hat f}.
	\end{align*}
	We note that $V$ is a bounded harmonic function in $D_f\intersect D_{\hat f}$, and
	\begin{align*}
		\textnormal{for}\ \ Y\in \partial \left( D_f\intersect D_{\hat f} \right) ,\ \ 
		\abs{V(Y)}\leq Cm\abs{h}.
	\end{align*}
	Thus by the maximum principle, we have
	\begin{align*}
		\textnormal{for all}\ \ X\in D_f\intersect D_{\hat f},\ \ \abs{V(X)}\leq Cm\abs{h},
	\end{align*}
	which implies the result of the lemma.

\end{proof}

Lemma \ref{lemAuxOp:WfLipOnK*Part1} has two immediate corollaries.

\begin{corollary}\label{corAuxOp:WfLipOnK*Part2}
	If $f\in \K^*(\gam,m)$, then $W_f\in C^{0,1}(D_f)$, and for $C=C(\gam,m,d)$,
	\begin{align*}
		\norm{\grad W_f}_{L^\infty(D_f)}\leq C.
	\end{align*}
\end{corollary}

\begin{corollary}\label{corAuxOp:HDefinedAEWhenFinK*}
	If $f\in \K^*(\gam,m)$, then for $a.e.\ x\in\real^d$, $H(f,x)$ is well defined.
\end{corollary}

The next basic result should not be conflated with that in Corollary \ref{corAuxOp:WfLipOnK*Part2}.  In Corollary \ref{corAuxOp:WfLipOnK*Part2}, it is very convenient that $W_f=0$ on $\Gam_f$, whereas the next result applies for functions with $C^{1,\gam}$ decay at $X_0$.

\begin{lemma}\label{lemAuxOp:BoundaryParabolaEstimate}
	If $w\in \K^*(\gam,m)\intersect (m\text{-}C^{1,\gam}(x_0))$, and for some $1>r_0>0$, $V$ satisfies 
	\begin{align*}
		\begin{cases}
			\Delta V = 0\ &\textnormal{in}\ D_w\intersect B_{r_0}(X_0),\\
			\abs{V(X)}\leq C_1\abs{X-X_0}^{1+\gam}\ &\textnormal{on}\ 
			\partial \left( D_w\intersect B_{r_0}(X_0) \right),
		\end{cases}
	\end{align*}
	then there exists a positive constant $C=C(\gam,m,d)$ so that
	\begin{align*}
		\abs{ V(X_0 + s\nu_w)}\leq C\cdot C_1\cdot s .
	\end{align*}
\end{lemma}

\begin{proof}[Proof of Lemma \ref{lemAuxOp:BoundaryParabolaEstimate}]	
	First, we note that under the assumption that $w\in \K^*(\gam,m)\intersect (m\text{-}C^{1,\gam}(x_0))$, there exists $r_1>0$ and $w^+$, so that
	\begin{align}\label{eqAuxOp:BoundaryParabolaProofAssumptionOnWPlus}
		&w^+\in \K(\gam,2m), \nonumber\\ 
		&\textnormal{with}\  w(x_0) = w^+(x_0),\ \ 
		\grad w(x_0) = \grad w^+(x_0), \nonumber \\
		&\textnormal{for}\ x\in \real^d,\ w(x)\leq w^+(x), \nonumber\\
		&\textnormal{and for}\ x\in B_{r_1}(x_0),\ 
		w^+(x)-w(x)\geq m\abs{x-x_0}^{1+\gam}.
	\end{align}

	Next, we will create a barrier from above for $V$ in a possibly smaller ball at $X_0$.  To this end, consider the function, $U^+$,
	\begin{align*}
		\begin{cases}
			\Delta U^+ = 0 &\text{in}\ D_{w^+}\intersect B_1\\
			U^+(X) = \abs{X-X_0}^{1+\gam}\ &\text{on}\ \partial(D_{w^+}\intersect B_1).
		\end{cases}
	\end{align*}
	We know that $U^+\in C^{1,\gam}(D_{w^+}\intersect B_{1/2})$, and furthermore, by the Hopf lemma, 
	\begin{align*}
		\partial_{\nu_{w^+}} U^+(X_0)>0\ \ \textnormal{and}\ \ \partial_{-e_{d+1}} U^+(X_0)>0.
	\end{align*}
	Thus, there exists $r_2>0$ and $\del>0$ small enough, so that
	\begin{align*}
		\textnormal{for}\ \ X\in B_{r_2}(X_0)\intersect D_{w^+}, \quad \text{we have} \  \partial_{-e_{d+1}} U^+(X)\geq \del.
	\end{align*} 
	Immediately from the fact that $\partial_{-e_{d+1}} U^+(X)\geq 0$, we know that $U^+$ has the correct size on $\Gam_w$.  Indeed, we have from $w(x)\leq w^+(x)$,
	\begin{align}\label{eqAuxOp:BoundaryParabolaProofBarrierAtGamW}
		\textnormal{for}\ \ X\in\Gam_w\intersect B_{r_2}(X_0),\ \ 
		U^+(x,w(x))\geq U^+(x,w^+(x)) = \abs{X-X_0}^{1+\gam}.
	\end{align}
	However, to use $U^+$ as a barrier, we need to control its size from below on all of $\partial(B_{r_2}\intersect D_w)$.  This is where the uniform positivity of $\partial_{-e_{d+1}} U^+(X)$ is used, in conjunction with the strict separation of $w^+$ from $w$ in (\ref{eqAuxOp:BoundaryParabolaProofAssumptionOnWPlus}).  It is important to note at this point that even though $r_2$ and $\del$ may be small, they are universal and fixed, simply by the properties of $U^+$, which in turn depends on $w^+$, which can be chosen to depend only on $\gam$, $m$, and $d$.

	We note that, for a dimensional constant, 
	\begin{align*}
		\inf_{(x,x_{d+1})\in \partial B_{r_2}(X_0)\intersect D_w} w^+(x) - x_{d+1}\geq c m (r_2)^{1+\gam},
	\end{align*}
	which can be seen from the fact that $w^+(x) - x_{d+1}$ is smallest when $\abs{X-X_0}=r_2$ and $X\in\Gam_w$.
	Since $\partial_{-e_{d+1}} U^+(X)\geq \del$ and $U^+\geq 0$ on $\Gam_{w^+}$, we then have
	\begin{align}\label{eqAuxOp:BoundaryParabolaProofBarrierInsideDw}
		\textnormal{for}\ \ X\in \partial B_{r_2}(X_0)\intersect D_w,\ \ 
		U^+(X)\geq \del c m (r_2)^{1+\gam}.
	\end{align}
	For simplicity, we may assume that $\del c m (r_2)^{1+\gam}\leq 1$.

	We finally notice that the assumptions on $V$ and that $0<r_0<1$, we have
	\begin{align*}
		\textnormal{for}\ \ X\in B_{r_0}(X_0)\intersect D_w,\ \ \abs{V(X)}\leq C_1.
	\end{align*}
	Thus, combining estimates (\ref{eqAuxOp:BoundaryParabolaProofBarrierAtGamW}) and (\ref{eqAuxOp:BoundaryParabolaProofBarrierInsideDw}), the functions, $\frac{\pm C_1}{\del c m (r_2)^{1+\gam}} U^+$, serve as barriers for $V$ in $B_{r_2}(X_0)$.  That is to say,
	\begin{align*}
		\textnormal{for}\ \ X\in B_{r_2}(X_0)\intersect D_w,\ \ 
		\frac{-C_1}{\del c m (r_2)^{1+\gam}} U^+(X)
		\leq V(X)
		\leq \frac{C_1}{\del c m (r_2)^{1+\gam}} U^+(X).
	\end{align*}
	 Furthermore, tracking the dependence of $U^+$ on $w^+$, we see that $\norm{U^+}_{C^{1,\gam}(D_{w^+}\intersect B_{1/2})}\leq C(m,\gam,d)$.   We can consolidate all of these constants into on $C=C(\gam,m,d)$.
	Thus, since
	\begin{align*}
		 V(X_0) = U^+(X_0) = 0,
	\end{align*}
	we conclude that 
	\begin{align*}
		\abs{V(X_0+s\nu)}\leq C_1\cdot CU^+(X_0+s\nu) \leq C_1\cdot C s.
	\end{align*}
	This gives the desired outcome of the lemma.
\end{proof}

\begin{rem}
\label{rem:barrier} 
	We note that the proof of Lemma \ref{lemAuxOp:BoundaryParabolaEstimate} did not use the full assumption that $w\in\K^*(\gam,m)$.  We only used the fact that there exists some $\tilde r$ such that $w$ is $C^{1,\gam}$-semi-concave in $B_{\tilde r}(x_0)$.  The assumption $w\in \K^*(\gam,m)$ appeared for convenience, and it suffices for our needs.
\end{rem}

%%%%%%%%%%%%%%%%%%%%%%%%%%%%%%%%%%%%%%%%%%%%%%%%%
%%%%%%%%%%%%%%%%%%%%%%%%%%%%%%%%%%%%%%%%%%%%%%%%%
%%%%%%%%%%%%%%%%%%%%%%%%%%%%%%%%%%%%%%%%%%%%%%%%%
%%%%%%%%%%%%%%%%%%%%%%%%%%%%%%%%%%%%%%%%%%%%%%%%%

\subsection{Pointwise evaluation of $H$}\label{sec:PointwiseEvaluationH}

The operator $H$ enjoys some more subtle properties, including the fact that it can be defined classically at points where $f\in C^{0,1}(\real^d)$ and may only be punctually $C^{1,\gam}$ at a point $x$. We establish this in Corollary \ref{corPtWise:NormalDeriv}.

 An analogous result was an important part of the regularity theory for free boundary problems, and a modern reference is \cite[Lemma 11.17]{CaffarelliSalsa-2005GeometricApproachtoFB}.  This was also used in a fundamental way in \cite{ChangLaraGuillenSchwab-2019SomeFBAsNonlocalParaboic-NonlinAnal}, where it appeared as \cite[Lemma 5.11]{ChangLaraGuillenSchwab-2019SomeFBAsNonlocalParaboic-NonlinAnal}.  Here we adapt the corresponding proofs from \cite{CaffarelliSalsa-2005GeometricApproachtoFB},  \cite{ChangLaraGuillenSchwab-2019SomeFBAsNonlocalParaboic-NonlinAnal} by adding a  bit of extra precision to their statements (necessary in our setting) and present some subsequent stability results.

We recall that in (\ref{eqPtWise:PtWiseKGamM}) we defined $\K(\gam,m,x_0, a, p)$ as the subset of $\K(\gam,m)$ comprised of functions that are punctually $C^{1,\gam}$ and share the same value $a$ and the same gradient $p$ at the point $x_0$.

We establish that once $m$, $x_0$, $a$, $p$ are fixed, one can select a particular harmonic function to compute the normal derivative for all $W_f$ for $f\in \K(\gam,m,x_0,a,p)$ (see Corollary \ref{corPtWise:NormalDeriv}).  These harmonic functions will be called $V_w$, as defined here.

\begin{definition}\label{defPtWise:VwHarmonic}
	Let $w\in \K^*(\gam,m)$, with $\inf w>0$, and $V_w$ be the harmonic function defined as
	\begin{align}\label{eqPtWise:HarmonicBarrierNiceDomain}
		\begin{cases}
			\Delta V_w = 0\ &\textnormal{in}\ D_w\\
			V_w = 0\ &\textnormal{on}\ \Gam_w\\
			V_w = 1\ &\textnormal{on}\ \Gam_0= \{ (x,x_{d+1})\ :\ x_{d+1}=0 \}.
		\end{cases}
	\end{align}
\end{definition}

This next lemma is the main result for this section.

\begin{lemma}[Pointwise evaluation]\label{LemPtWise:NonTanBehavior}
Given $m$, $a\geq 1$, and $p\in\real^d$, there exist  a positive constant $r>0$ and a fixed $w\in C^{1,\gam}(B_r(x_0))\intersect C^{0,1}(\real^d)$, depending on $m$, $a$, $p$, with
	\begin{align*}
		w(x_0)=a,\ \ \grad w(x_0)=p,\ \ w\geq \frac{1}{2},\ \ \textnormal{and}\ \ \norm{w}_{C^{1,\gam}(B_r(x_0))}\leq m,\ \ \norm{w}_{C^{0,1}(\real^d)}\leq m,
	\end{align*}
so that for all 
	\begin{align}
	\label{eqPtWise:f in K a p geq 1}
		f\in\K(\gam,m,x_0,a,p)\ \textnormal{with}\ f\geq 1,
	\end{align}
	and for all 
	$U$ that satisfy for some $R_0$, 
	\begin{align*}
		U\geq 0,\ \ U(X_0)=0,\ \ \textnormal{and}\ \  \Delta U = 0,\ \textnormal{in}\ D_f\intersect B_{R_0}(X_0),
	\end{align*}
	and that grows at most linearly away from $X_0=(x_0, f(x_0))$, i.e. there exists a constant  $C_U>0$ with
	\begin{align*}
		\abs{U(X)}\leq C_U\abs{X-X_0}\ \ \textnormal{in}\ \ D_f\intersect B_{R_0}(X_0),
	\end{align*}
	then for a positive constant, $r_f>0$, the following asymptotic behavior is valid, for $\abs{X-X_0}\leq r_f$:
	\begin{align}\label{eqPtWise:NonTanBehavior}
		\textnormal{for}\ X\to X_0\ \ \textnormal{non-tangentially in}\ D_w,\ \ 
		U(X) = s_w(U)V_w(X) + o(\abs{X-X_0}),
	\end{align}
	where $V_w$ is as in (\ref{eqPtWise:HarmonicBarrierNiceDomain}) and $s_w(U)\in (0,\infty)$ is uniquely determined by $w$ and $U$. 
\end{lemma}

\begin{rem}
	We note that this results holds under less restrictive assumptions on $f$, but we have taken $f$ to be in $\K^*(\gam,m)$ simply because that is the situation in which we apply the lemma (note, $\K(\gam, m, x_0, a, p)\subset \K^*(\gam,m)$).  Indeed, following \cite[Lemma 11.17]{CaffarelliSalsa-2005GeometricApproachtoFB} directly, the result holds if $f$ is simply punctually $C^{1,\gam}(x_0)$ \emph{only from below}, not necessarily from both sides.  In that case, the behavior of $U$ takes the form of a dichotomy, where either $U$ grows super linearly from $X_0$, or if it grows at most linearly, then (\ref{eqPtWise:NonTanBehavior}) is valid.
\end{rem}

\begin{proof}[Proof of Lemma \ref{LemPtWise:NonTanBehavior}]
Consider the function $w$ defined by
\begin{align*}
	w(y) = \max \left\lbrace a + p\cdot(y-x_0) - m\abs{y-x_0}^{1+\gam}, \ \frac{1}{2} \right\rbrace.
\end{align*}
For any $f$ satisfying (\ref{eqPtWise:f in K a p geq 1}), there exists $r_f$ so that 
\begin{align*}
	\text{for}\ \ \abs{y-x_0}\leq r_f,\ \ 
	 w(y)\leq f(y)\ \ 
	\textnormal{and}\ \ w(x_0)=f(x_0)=a.
\end{align*}

We now use, for this fixed $w$, the function $V_w$ as in Definition \ref{defPtWise:VwHarmonic} to determine the non-tangential behavior of $U$.  (We note that $w$ is $C^{1,\gam}$ in a neighborhood of $x_0$, and so $V_w$ will also be $C^{1,\gam}$ in a neighborhood of $X_0$.)  To this end, let us define the sequence, $\al_k$ as
\begin{align*}
	\al_k(U) = \sup\{ s\ :\ U\geq sV_w\ \textnormal{in}\ B_{2^{-k}}(X_0)\intersect D_w \}.
\end{align*}
We note that thanks to $r_f$ as above, for $k$ large enough,
\begin{align*}
    B_{2^{-k}}(X_0)\cap D_w \subset B_{2^{-k}}(X_0) \cap D_f
\end{align*}
and so $\al_k$ is well defined for $k$ large enough.
We see also that $\al_k$ is increasing.  Furthermore, as $U$ is assumed to have at most linear growth at $X_0$,  
we see that $\al_k$ is also bounded.  We then define
\begin{align}\label{eqPtWise:SfIsSup}
	s_w(U) = \sup_k\{ \al_k(U)  \}
\end{align}

We shall first establish that the lower bound on $U$ of (\ref{eqPtWise:NonTanBehavior}) holds. To this end,  let $X_i\to X_0$ be any sequence converging to $X_0$. 
Let $k_i$ be a non-decreasing subsequence with the property that
\begin{align}
\label{eqAuxOp:Xi-X0 for lb}
	\abs{X_i-X_0}\leq 2^{-k_i}.
\end{align}
We see by the definition of $\al_{k_i}$, as well as by (\ref{eqAuxOp:Xi-X0 for lb}), that 
\begin{align*}
	U(X_i)\geq \al_{k_i}V_w(X_i) = s_w V_w(X_i) + (\al_{k_i}-s_f)V_w(X_i).
\end{align*}
As $V_w$ is globally Lipschitz in $\overline{D}_w$, we see that there is a positive $C=C(\gam,m,d,a,p)$, so that
\begin{align*}
	V_w(X_i)-V_w(X_0) = V_w(X_i)\leq C\abs{X_i-X_0}.
\end{align*}
Thus,
\begin{align*}
	U(X_i)\geq s_w V_w(X_i) + C(\al_{k_i}-s_w)\abs{X_i-X_0}.
\end{align*}
Hence, one half of the equality in (\ref{eqPtWise:NonTanBehavior}) has been established.

For the reverse inequality, we argue by contradiction.  To this end, we suppose that there exists a sequence $X_i\to X_0$ non-tangentially in $D_w$, and yet, there exists $c_0>0$, with
\begin{align*}
	U(X_i)\geq s_w V_w(X_i) + c_0\abs{X_i-X_0}.
\end{align*}

Let us now take a non-decreasing subsequence $k_i$ with
	\begin{align*}
		2^{-k_i-2}\leq \abs{X_i-X_0}\leq 2^{-k_i-1},
	\end{align*}
	so that 
	\begin{align*}
		B_{2^{-k_i-1}}(X_i)\subset B_{2^{-k_i}}(X_0).
	\end{align*}
As $\al_{k_i}\leq s_w$, we have
\begin{align*}
	U\geq  s_w  V_w(X_i) + c_0\abs{X_i-X_0}\geq \al_{k_i}V_w(X_i) + c_0\abs{X_i-X_0}.
\end{align*}
Let us consider the function, $\Tilde V$, given by
\begin{align*}
	\Tilde V_i = U - \al_{k_i}V_w,
\end{align*}
which, by definition of $\alpha_{k_i}$, is non-negative and harmonic in $D_w \intersect B_{2^{-k_i}}(X_0)$. 
We note that at $X_i$, we have
\begin{align*}
	\Tilde V_i(X_i) \geq c_0\abs{X_i-X_0}.
\end{align*}
Let us call 
\begin{align*}
	r_i = \abs{X_i-X_0},
\end{align*}
The non-tangential nature of the convergence means that there is a choice of $c_1>0$,  so that 
\begin{align*}
	B_{c_1r_i}(X_i)\subset B_{2^{-k_i}}(X_0)\intersect D_w.
\end{align*}
Since $\Tilde V_i\geq 0$ we can apply the Harnack inequality to see that for a universal $c_2$ (with $0<c_2<1$),
\begin{align*}
	\Tilde V_i \geq c_2c_0\abs{X_i-X_0}\ \ \textnormal{in}\ \ B_{\frac{1}{2}c_1r_i}(X_i).
\end{align*} 
We can now use a covering argument and repeated use of the Harnack inequality to extend this estimate.  Indeed, there exists a positive constant $L$, depending only on $c_1$ and $d$, such that for each $i$, there exist $Y^i_l\in D_f$, for $l=1,\dots,L$,  with the property that
\begin{align*}
	B_{c_1r_i}(Y^i_l)\subset D_f\ \ \textnormal{and}\ \ 
	\{ X\ :\ 2^{-k_i-1}\leq d(X,\Gam_f)\leq 2^{-k_i}  \}\intersect B_{2^{-k_i}}(X_0)
	\subset \Union_{l=1}^L B_{\frac{1}{2}c_1r_i}(Y^i_l).
\end{align*}
Using the Harnack inequality up to possibly $L$ times, we see that for $\tilde c_2 = (c_2)^L$,
\begin{align*}
	\textnormal{for}\ \ Z\in \{ X\ :\ 2^{-k_i-1}\leq d(X,\Gam_f)\leq 2^{-k_i}  \}\intersect B_{2^{-k_i}}(X_0),\ \ 
	\Tilde V_i(Z) \geq \tilde c_2c_0\abs{X_i-X_0}.
\end{align*}
Thus, using a standard barrier (e.g. $\tilde c_2c_0\abs{X_i-X_0} b_w$, where $b_w$ is as in Lemma \ref{lemAuxOp:BarrierFuns}, for $R=c2^{-k_i}$), we  obtain, for a universal $c_3$,
\begin{align*}
	\Tilde V_i(Z) \geq c_3\tilde c_2 c_0 d(Z,\Gam_w)\ \ \textnormal{for}\ \ Z\in B_{2^{-k_i}}(X_0)\intersect D_w.
\end{align*}
We note that $c_3$ is universal because the factor $\abs{X_i-X_0}/R$ is independent of $i$, and is bounded away from $0$.
Consolidating constants, with $c=c_3\tilde c_2 c_0$, and returning to $U$, we see that
\begin{align*}
	U(Z)\geq \al_{k_i} V_w(Z) + c d(Z,\Gam_w)\ \ \textnormal{for}\ \ Z\in B_{2^{-k_i}}(X_0)\intersect D_w,
\end{align*}
which also gives for an arbitrary $\al>0$,
\begin{align*}
	U(Z)\geq   \al V_w(Z) + (\al_{k_i} - \al) V_w(Z) + c d(Z,\Gam_w)\ \ \textnormal{for}\ \ Z\in B_{2^{-k_i}}(X_0)\intersect D_w.
\end{align*}
Given the global Lipschitz estimates that $V_w$ enjoys, we know that there exists a positive constant $C_w=C(\gam,m,d,a,p)$ so that
\begin{align*}
	\norm{\grad V_w}_{L^\infty(D_w)}\leq C_w,
\end{align*}
and hence,
\begin{align*}
	\abs{\al_{k_i}-\al}V_w(Z)\leq C_w\abs{\al_{k_i}-\al} d(Z,\Gam_w) .
\end{align*}
Let us now choose $\alpha=  s_w  +\frac{c}{4C_w}$ and $i$ large enough so that $\alpha_{k_i}>  s_w-\frac{c}{4C_w}$, 
so that
\begin{align*}
	C_w(\al-\al_{k_i})<\frac{1}{2}c.
\end{align*} 
Hence, for this particular $\al>s_w(U)$ and $i$ fixed we see that
\begin{align*}
	U(Z)\geq \al V_w(Z) + \frac{c}{2}d(Z,\Gam_w)\ \geq \al V_w(Z) \ \textnormal{for}\ \ Z\in B_{2^{-k_i}}(X_0)\intersect D_w.
\end{align*}
This now contradicts the fact that $\al_{k_i}$ is the largest such constant for which $U$ is above $sV_w$ in $B_{2^{-k_i}}(X_0)$. 
We thus conclude that (\ref{eqPtWise:NonTanBehavior}) indeed holds.
\end{proof}

An immediate corollary is that $W_f$ has a classically defined normal derivative at points where $f$ is $m$-$C^{1,\gamma}(x_0)$, and $H(f)$ is classically defined as well.

\begin{corollary}\label{corPtWise:NormalDeriv}
	If $x_0$ is fixed and $f\in C^{0,1}(\real^d)$ is punctually $C^{1,\gam}(x_0)$, then $W_f$ has a normal derivative at $X_0=(x_0, f(x_0))$, defined classically as
	\begin{align*}
		\partial_{\nu_f} W_f(X_0) = \lim_{t\to0^+}\frac{1}{t}W_f(X_0+ t\nu_f). 
	\end{align*}
	Furthermore, $H(f,x_0)$ is well defined, just as in Definition \ref{defAuxOp:TheOperatorH},
	\begin{align*}
		H(f,x_0) = \left( \sqrt{1+\abs{\grad f(x_0)}^2} \right)\partial_{\nu_f} W_f(x_0,f(x_0)).
	\end{align*}
\end{corollary}

\begin{proof}[Proof of Corollary \ref{corPtWise:NormalDeriv}]
	Let $w$ be as given by Lemma \ref{LemPtWise:NonTanBehavior}. We note $w\in C^{1,\gam}(B_r(x_0))$, and thus $V_w\in C^{1,\gam}(B_{r/2}(X_0))$.
	When it happens that $f\geq 1$, then the desired conclusion follows directly from Lemma \ref{LemPtWise:NonTanBehavior} (and in fact, $\partial_{\nu_f} W_f(X_0) =s_w(W_f) \partial_{\nu_f} V_w(X_0)$).  But we recall that the construction of $W_f$ is invariant by the addition of constants, i.e. for any constant $c$, 
	\begin{align*}
		W_{f+c}(X) = W_f(X-c)\ \ \textnormal{in}\ \ D_{f+c}.
	\end{align*}
	Thus, we see that after adding the constant $c=\norm{f}+1$ to $f$, the function $f+c$ satisfies the assumptions of Lemma \ref{LemPtWise:NonTanBehavior}.  By assumption, $\grad f(x_0)$ exists, and so combined with the existence of $\partial_{\nu_f} W_f(X_0)$, $H$ is well defined.
\end{proof}

Next, we use Corollary \ref{corPtWise:NormalDeriv} to extend the GCP (Definition \ref{defAuxOp:GCP} and Proposition \ref{propAuxOp:HProperties}-(ii)) to functions that are only punctually $C^{1,\gamma}$.

\begin{lemma}[Punctual GCP]\label{LemPtWise:PunctialGCP}  Let $x_0\in \R^d$ be fixed, and let $f,g\in C^{0,1}(\R^d)$ be functions such that $f$ and $g$ are punctually $C^{1,\gamma}(x_0)$, with $f(x_0) = g(x_0)$ and $f\leq g$ in $\R^d$. Then, 
\begin{equation*}
    H(f,x_0)\leq H(g,x_0).
\end{equation*}
\end{lemma}
\begin{proof} In view of Corollary \ref{corPtWise:NormalDeriv}, $H(f,x_0)$ and $H(g,x_0)$ are well-defined. Without loss of generality, we can also assume $f,g\geq 1$. By assumption we have $\nabla (f-g)(x_0) = 0$. Let $m>0$ and $X_0 = (x_0, f(x_0))$ as usual, we define 
\begin{align*}
    \nu_w = \nu_f(X_0) = \nu_g(X_0), \qquad a = f(x_0) = g(x_0), \qquad p = \nabla f(x_0) = \nabla g(x_0). 
\end{align*}
By Lemma \ref{LemPtWise:NonTanBehavior} there exists $r>0$ and $w\in C^{1,\gamma}(B_r(x_0))$ with $w \leq f\leq g$ in $B_r(x_0)$ and  $w(x_0) = a$, $\nabla w(x_0) = p$ such that
\begin{align*}
    W_f(X) = s_w(W_f)V_w(X) + o(|X-X_0|) 
    \qquad\text{and}\qquad
    W_g(X) = s_w(W_g)V_w(X) + o(|X-X_0|)
\end{align*}
for $X\to X_0$ non-tangentially in $D_w$. Consequently, we have
\begin{align}\label{eqPtWiseGCP:formula}
    \partial_{\nu_f}W_f(X_0) = s_w(W_f) \partial_{\nu_w}V_w(X_0)
    \qquad \text{and} \qquad
    \partial_{\nu_g}W_g(X_0) = s_w(W_g) \partial_{\nu_w}V_w(X_0)
\end{align}
where $V_w$ is defined in Definition \ref{defPtWise:VwHarmonic}. Recalling the definition of $s_w(W_f)$ and $s_w(W_g)$ from the proof of Lemma \ref{LemPtWise:NonTanBehavior}, we have 
\begin{align*}
\begin{cases}
    \al_k(W_f) = \sup\{ s\ :\ W_f\geq sV_w\ \textnormal{in}\ B_{2^{-k}}(X_0)\intersect D_w \} \\
    \al_k(W_g) = \sup\{ s\ :\ W_g\geq sV_w\ \textnormal{in}\ B_{2^{-k}}(X_0)\intersect D_w \}.
\end{cases}
\end{align*}
Since $W_f\leq W_g$ due to $f\leq g$, it follows that $\alpha_k(W_f)\leq \alpha_k(W_g)$, which implies that 
\begin{align*}
    s_w(W_f) = \sup_k \alpha_k(W_f) \leq \sup_k \alpha_k(W_g) = s_w(W_g).
\end{align*}
Using this in \eqref{eqPtWiseGCP:formula} we obtain that $H(f,x_0)\leq H(g,x_0)$.
\end{proof}

\begin{lemma}\label{lemPtWise:NormalDerivLinftyDependence}
	There exists a positive constant $C=C(\gam,m,d)$  so that if $f$ and $g$ are such that  $f\leq g$, with $f,g\in \K^*(\gam,m)\intersect m\text{-}C^{1,\gam}(x_0)$, with $f(x_0)=g(x_0)$ and $\grad f(x_0)= \grad g(x_0)$, then
	\begin{align*}
		0 \leq \partial_\nu W_g(X_0) - \partial_\nu W_f(X_0)
		\leq  C\norm{f-g}_{L^\infty(\real^d)},
	\end{align*} 
	where $X_0=(x_0, f(x_0))$. Consequently, given the standing assumptions about $f$ and $g$, we have
	\begin{align*}
		0 \leq H(g,x_0) - H(f,x_0)
		\leq  C\norm{f-g}_{L^\infty(\real^d)}. 
	\end{align*}

\end{lemma}

\begin{proof}[Proof of Lemma \ref{lemPtWise:NormalDerivLinftyDependence}]
	First, we remark that  Corollary \ref{corPtWise:NormalDeriv} implies that    $\partial_\nu W_f(X_0)$ and $ \partial_\nu W_g(X_0)$   are well-defined.	Next, we note the first inequality is immediate from the GCP.  Let us address the second inequality. We recall that by assumption, $f\leq g$. From Lemma \ref{lemAuxOp:WfLipOnK*Part1}, we see that for $C=C(\gam,m,d)$,
		\begin{align*}
			\textnormal{for}\ X\in\Gam_f,\ \ 
			0\leq W_g(X) - W_f(X) \leq C(g(x)-f(x)),
		\end{align*}
		and hence by the comparison principle (Proposition \ref{propAppendix:weak-max-principle}), 
		\begin{align*}
			\norm{W_f - W_g}_{L^\infty(D_f)}\leq C\norm{f-g}_{L^\infty(\real^d)}.
		\end{align*}
		Let us call $\ep = C\norm{f-g}_{L^\infty(\real^d)}$.
	Thus since $f\in \K^*(\gam,m)$, there is a barrier $b$, as in Lemma \ref{lemAuxOp:BarrierFuns}, with  $\|b\|_{C^{1,\gam}(\real^d)}$ depending only on $\gam$, $m$, and $d$, so that for some $r>0$,
	\begin{align*}
		\textnormal{in}\ \ D_f\intersect B_r(X_0),\ \ 
		0\leq W_g - W_f\leq \ep b.
	\end{align*} 
The fact that $b$ is Lipschitz yields that there exists a positive constant  $C_1=C_1(\gam,m,d)$ such that, for $t>0$ small enough, 
	\begin{align*}
		0\leq W_g(X_0 + t\nu) - W_f(X_0 + t\nu)\leq \ C_1\ep t.
	\end{align*}
		
	After consolidating constants, all of which only depended on $\gam$, $m$, $d$, we see that there exists a positive constant $C_2=C_2(\gam,m,d)$ with
	\begin{align*}
		\partial_\nu W_g(X_0) - \partial_\nu W_f(X_0) \leq C_2 \norm{f-g}_{L^\infty(\R^d)}.
	\end{align*}
Finally, recalling the definition of $H$ as in \eqref{eqCheckAssumpt:DefOfHeleShawOp}, and noting that $\nabla f(x_0) = \nabla g(x_0)$,  we deduce the desired conclusion.
\end{proof}

%%%%%%%%%%%%%%%%%%%%%%%%%%%%%%%%%%%%%%%%%%%%%%%%%
%%%%%%%%%%%%%%%%%%%%%%%%%%%%%%%%%%%%%%%%%%%%%%%%%
%%%%%%%%%%%%%%%%%%%%%%%%%%%%%%%%%%%%%%%%%%%%%%%%%
%%%%%%%%%%%%%%%%%%%%%%%%%%%%%%%%%%%%%%%%%%%%%%%%%

\subsection{The Lipschitz property}\label{subsection:Lipschitz}

Next we establish that $H$ is a locally Lipschitz as a mapping $C^{1,\gam}(\real^d)\to C^0(\real^d)$. The main lemmas that will be used to establish this result will be useful in their own right.   We begin with:

\begin{lemma}\label{lemAuxOp:PreparingForHLipschitz-LocalEstimateK*} 
Given $m>0$ and $R>0$, there exists a positive constant, $C=C(\gam,m,d)$, so that if 
	\begin{enumerate}[(i)]
		
		\item $f,g\in \K^*(\gam,m)\intersect \left(m\text{-}C^{1,\gam}(x_0)\right)$,
		
		\item $f(x_0)=g(x_0)$ and $\grad f(x_0) = \grad g(x_0)$,
		
		\item $f-g\in C^{1,\gam}(B_R(x_0))$,
		
		\item $U_f,U_g$ solve
		\begin{align*}
			\begin{cases}
				\Delta U_f=0\ &\textnormal{in}\ D_f\intersect B_R(X_0)\\
				U_f=0\ &\textnormal{on}\ \Gam_f\intersect B_R(X_0),
			\end{cases}
			\ \ \ \ \textnormal{and}\ \ \ \
			\begin{cases}
				\Delta U_g=0\ &\textnormal{in}\ D_g\intersect B_R(X_0)\\
				U_g=0\ &\textnormal{on}\ \Gam_g\intersect B_R(X_0),
			\end{cases}
		\end{align*}
		
	\end{enumerate}
	where $X_0=(x_0, f(x_0)$, then
\begin{align*}
    \abs{\partial_{\nu_f} U_f(X_0) - \partial_{\nu_f}U_g(X_0) } \leq 
	C(\gam, m,d)\left( \norm{f-g}_{C^{1,\gam}(B_R(x_0))} + \frac{1}{R^{1+\gam}}\norm{U_f-U_g}_{L^\infty(D_f\intersect D_g\intersect B_R(X_0) )} \right).
\end{align*}

\end{lemma}

An immediate consequence of Lemma \ref{lemAuxOp:PreparingForHLipschitz-LocalEstimateK*} is the following proposition relating to $H(f,x_0)$.

\begin{proposition}\label{propAuxOp:HLipschitz-K*Ingredients} 
	Given $m$, and $R>0$, there exists a positive constant, $C=C(\gam,m,d)$, so that if 
		\begin{enumerate}[(i)]
		
			\item $f,g\in \K^*(\gam,m)\intersect \left(m\text{-}C^{1,\gam}(x_0)\right)$
		
			\item $f(x_0)=g(x_0)$ and $\grad f(x_0) = \grad g(x_0)$,
		
			\item $f-g\in C^{1,\gam}(B_R(x_0))$
		
		\end{enumerate}
		then
	\begin{align*}
	    \abs{H(f,x_0)-H(g, x_0)} \leq 
		C(\gam, m,d)\left( \norm{f-g}_{C^{1,\gam}(B_R(x_0))} + \frac{1}{R^{1+\gam}}\norm{f-g}_{L^\infty(\real^d)} \right).
	\end{align*}

\end{proposition}

Proposition \ref{propAuxOp:HLipschitz-K*Ingredients} is the main ingredient in the Lipschitz proof for $H$.  We state this here, and then will provide the proof of both afterwards.

\begin{theorem}\label{thmAuxOp:HPreciseLipStatementC1Gam}
	If $R_0>0$ is fixed, $R>R_0$, and $f,g\in \K^*(\gam,m)\intersect C^{1,\gam}(B_{2R}(x_0))$, then there exists a postive constant $C(\gam,m,d, R_0)$ so that 
	\begin{align*}
		\norm{H(f,\cdot) - H(g,\cdot)}_{L^\infty (B_R(x_0))} \leq
		C(\gam, m,d, R_0)\left( \norm{f-g}_{C^{1,\gam}(B_{2R}(x_0))} +  \norm{f-g}_{L^\infty(\real^d)} \right).
	\end{align*}
\end{theorem}

\begin{corollary}\label{CorAuxOp:HLocallyLipC1Gam}
	$H$ is locally Lipschitz as a mapping $H: C^{1,\gam}(\real^d)\to C^0(\real^d)$.
\end{corollary}

\begin{proof}[Proof of Lemma \ref{lemAuxOp:PreparingForHLipschitz-LocalEstimateK*}]
	
	We consider the function $V=U_f-U_g$ in the domain $D_f\intersect D_g\intersect B_R(X_0)$.  We note that for the function $w$ given by 
	\begin{align*}
		w = \min\{f,g\},
	\end{align*}
	we have
	\begin{align*}
		D_f\intersect D_g = D_w\ \ \text{and}\ \ 
		\partial \left( D_f\intersect D_g \right) = \Gam_w.
	\end{align*}
	Taking a minimum of functions in $\K^*(\gam,m)$ preserves this property, and so also $w\in \K^*(\gam,m)$. And, as $f, g\in m-C^{1,\gamma}(x_0)$ with $f(x_0)=g(x_0)$ and $\grad f(x_0) = \grad g(x_0)$, we see that $w\in m\text{-}C^{1,\gam}(x_0)$.

	We note that for $x\in B_R(x_0)$, since $f-g\in C^{1,\gam}(B_R(x_0))$, and since $f,g\in \K^*(\gam,m)$, by Lemma \ref{lemAuxOp:WfLipOnK*Part1}, for $X\in \Gam_w\intersect B_R(x_0)$,
	\begin{align*}
		\abs{U_f(X) - U_g(X)} \leq C(\gam,m,d)\abs{f(x)-g(x)}\leq C(\gam,m)\norm{f-g}_{C^{1,\gam}(B_R(x_0))}\abs{X-X_0}^{1+\gam}.
	\end{align*}
	Thus, for $X\in \partial(D_w\intersect B_R(x_0))$,
	\begin{align*}
		\abs{V(X)}\leq C(\gam,m,d)\left( \norm{f-g}_{C^{1,\gam}(\real^d)}  
		+ \frac{\norm{U_f-U_g}_{L^\infty(B_R\intersect D_w)}}{R^{1+\gam}}\right)
		\abs{X-X_0}^{1+\gam}.
	\end{align*}
	We can then apply Lemma \ref{lemAuxOp:BoundaryParabolaEstimate} to $V$ to conclude the desired inequality for $\abs{\partial_\nu U_f(X_0) - \partial_\nu U_g(X_0) }$.
	
\end{proof}

\begin{proof}[Proof of Proposition \ref{propAuxOp:HLipschitz-K*Ingredients}]
	
	We use the same notation as in the proof of Lemma \ref{lemAuxOp:PreparingForHLipschitz-LocalEstimateK*}, specifically, $w=\min\{f,g\}$.
	Lemma \ref{lemAuxOp:WfLipOnK*Part1} for each of $W_f$ and $W_g$ shows that there exists  a positive constant $C$ depending on $\gam$ and $m$ such that,
	\begin{align*}
		\textnormal{for}\ X\in\Gam_w,\ \ \abs{W_f(X)-W_g(X)}\leq C\norm{f-g}_{L^\infty(\R^d)}.
	\end{align*}
	Hence, by the comparison result, Proposition \ref{propAppendix:weak-max-principle}, we see that 
	\begin{align*}
		\norm{W_f-W_g}_{L^\infty(D_f\intersect D_g)}\leq \norm{W_f-W_g}_{L^\infty(\Gam_w)}
		\leq C\norm{f-g}_{L^\infty(\real^d)}.
	\end{align*}
	Furthermore, from the assumption that $f,g\in (m\text{-}C^{1,\gam}(x_0))$, that $f(x_0)=g(x_0)$, and $\grad f(x_0)=\grad g(x_0)$, we have that $w\in (m\text{-}C^{1,\gam}(x_0))$.  Also, by assumption, $f-g\in C^{1,\gam}(B_R(x_0))$.  We can then invoke Lemma \ref{lemAuxOp:PreparingForHLipschitz-LocalEstimateK*}.
\end{proof}

\begin{proof}[Proof of Theorem \ref{thmAuxOp:HPreciseLipStatementC1Gam}]

We split the proof into a few  steps for clarity. We would like to point out that the following steps are intended to create a situation to which Lemma \ref{lemAuxOp:PreparingForHLipschitz-LocalEstimateK*} can be applied.

\paragraph{\textbf{Step 0.}} A translation and shift to attain $x=0$ and $f(0)=g(0)=0$.

By Proposition \ref{propAuxOp:HProperties}, we know that $H$ is translation invariant and invariant by the addition of constants.  Thus, without loss of generality, we will prove that
\begin{align}\label{eqLip:OnlyNeedXAtZero}
	\abs{H(f,0)-H(g,0)}\leq C \left(
	\norm{f-g}_{C^{1,\gam}(B_{R})} +  
	\norm{f-g}_{L^\infty(\real^d)}
	\right),
\end{align}
under the addition assumption that $f(0)=g(0)$.  Then, after taking the translation into account, we will obtain the estimate for $\norm{H(f,\cdot) - H(g,\cdot)}_{L^\infty(B_R(x_0))}$.

\paragraph{\textbf{Step 1.}}  A smallness assumption on $f-g$ that depends on $m$, $\gam$, $d$.

It will be useful in the subsequent steps to have $f-g$ appropriately small, which comes from using a rotation of $\real^{d+1}$ in our argument.  The choice of $\ep$ happens in step 2.  To this end, given $m$, and given a fixed $\ep>0$, we note that it suffices to show that (\ref{eqLip:OnlyNeedXAtZero}) holds whenever
\begin{align}\label{eqLip:OnlyNeedXAtZero-small-eps}
	\norm{f-g}_{C^{1,\gam}(B_R)} + 
	\norm{f-g}_{L^\infty(\mathbb{R}^d)}  < \varepsilon.
\end{align}
Indeed, by Corollary \ref{corAuxOp:WfLipOnK*Part2}, since $f,g\in \K^*(\gamma, m)$, there exists $C(\gamma, m, d)$ such that
\begin{align}\label{eqLipschitz-boundWf-Wg}
   \Vert \nabla W_f\Vert_{L^\infty(\overline{D}_f)} \leq C, \qquad    \Vert \nabla W_g\Vert_{L^\infty(\overline{D}_g)} \leq C.
\end{align}
 Therefore $|H(f,0)|\leq C$ and $|H(g,0)|\leq C$. For $f,g$ such that \eqref{eqLip:OnlyNeedXAtZero-small-eps} is not true then
\begin{align*}
   |H(f,0) - H(g, 0)| &\leq |H(f,0)| + |H(g,0)| \\
   &\leq \left(2C\varepsilon^{-1}\right)\varepsilon \leq \left(2C\varepsilon^{-1}\right)
   \left( \norm{f-g}_{C^{1,\gam}(B_R)} + \norm{f-g}_{L^\infty(\mathbb{R}^d)} \right).
\end{align*}

We reiterate that all dependence on the choice of $\ep$ appears in step 2.

\paragraph{\textbf{Step 2.}} Apply a rotation to reduce to $\grad f(0)=\grad g(0)$.

Let $\mathcal{R}:\real^{d+1}\to\real^{d+1}$ be  
 the unique rotation such that
\begin{align*}
	\mathcal{R}\nu_g = \nu_f\ \ \ \textnormal{and}\ \ \
	\mathcal{R} v = v\ \ \textnormal{for}\ \ v\in\left(\textnormal{span}(\nu_f,\nu_g)\right)^{\perp}.
\end{align*}
We note that $\ep$ can be chosen small enough so that $\abs{\grad f(0) - \grad g(0)}\leq \ep$ implies that the image of $\Gamma_g$ via $\mathcal{R}$ inside of $B_{R_0}$ is  a graph of a bounded function.   Furthermore, this function can be extended to all of $\real^d$ is a way that does not expand the $\K^*(\gam, m)$ property by more than a multiple of $m$.  More precisely, for 
$\varepsilon=\varepsilon(m, \gamma, d, R_0)>0$ small enough, there exists a function $w$ so that
\begin{align}\label{eqAuxOp:LipProofConditionsOnw}
	&w \in \K^*(\gam, 2m)\intersect (2m)\text{-} C^{1,\gam}(0),
	\ \ \textnormal{and}\ \ 
	f-w\in C^{1,\gam}(B_{R_0}) \nonumber\\ 
	&\textnormal{with}\ \  
	\norm{f-w}_{C^{1,\gam}(B_{R_0})}\leq 2\norm{f-g}_{C^{1,\gam}(B_{R_0})},
\end{align}
and also 
\begin{align*}
	\Rmc(\Gam_g)\intersect B_{R_0} = \Gam_w \intersect B_{R_0} . 
\end{align*}
  Indeed, by construction of  $\mathcal{R}$, we see
  \begin{align*}
    \norm{ \mathcal{R} - \mathrm{Id}} \leq \vert \nu_f-\nu_g\vert \leq  |\nabla f(0) - \nabla g(0)|.
\end{align*}
As a consequence, there exists $C=C(R_0)>0$, so that
\begin{align*}
	\norm{g-w}_{L^\infty(B_{R_0})}\leq C\norm{\mathcal{R} - \Id}\leq C\abs{\grad f(0)-\grad g(0)},
\end{align*}
and thus a smallness assumption on $\abs{\grad f(0) - \grad g(0)}<\ep$ can be used to obtain (\ref{eqAuxOp:LipProofConditionsOnw}).

\medskip

\noindent
\paragraph{\textbf{Step 3.}} An estimate $\abs{\partial_{\nu_f}W_f(0) - \partial_{\nu_g} W_g(0)}$.

We introduce the following function, $V$:
\begin{align}\label{eqLipschizt:def_V}
	V:D_w\intersect B_{R_0}\to\real,\ \ \ \ V(X) = W_g\left(\mathcal{R}^{-1}X\right).
\end{align}
By the definition of $\mathcal{R}$ and $w$, $V$ is harmonic in $D_w\intersect B_{R_0}$, and by construction, $\partial_{\nu_f}V(0) = \partial_{\nu_g}W_g(0)$. Therefore
\begin{align*}
	\partial_{\nu_f} W_f(0) - \partial_{\nu_g} W_g(0)
	&= \partial_{\nu_f}W_f(0) - \partial_{\nu_f} V(0).
\end{align*}
Furthermore, by construction, $f$, $w$, $W_f$, and $V$ satisfy the assumptions of Lemma \ref{lemAuxOp:PreparingForHLipschitz-LocalEstimateK*} for $r={R_0}$.  Thus,
\begin{align*}
	\abs{\partial_{\nu_f}W_f(0) - \partial_{\nu_f} V(0)}
	&\leq C(\gam,m,d)\left(
	\norm{f-w}_{C^{1,\gam}(B_{R_0})} +  \norm{W_f - V}_{L^\infty(D_f\intersect D_w\intersect B_{R_0})}\right) \\
	&\leq C(\gam,m,d)\left(
	\norm{f-g}_{C^{1,\gam}(B_{R_0})} +  \norm{W_f - V}_{L^\infty(D_f\intersect D_w\intersect B_{R_0})}
	\right),
\end{align*}
where we've used  (\ref{eqAuxOp:LipProofConditionsOnw}) to obtain the final inequality. 
We will now establish,
\begin{align}
\label{eq:Lip to do}
\norm{W_f - V}_{L^\infty(D_f\intersect D_w\intersect B_{R_0})}&\leq C(\gam,m,d)\left(\norm{f-g}_{C^{1,\gam}(B_{R_0})} +  \|f-g\|_{L^\infty(\real^d)}\right);
\end{align}
together with the previous estimate, and the fact that $\partial_{\nu_f} V(0) = \partial_{\nu_g}W_g(0)$, will complete the proof.

To this end, we first note that from the fact that $f,g\in \K^*(\gam, m)$ and Lemma \ref{lemAuxOp:WfLipOnK*Part1}, we have that
\begin{align*}
	\textnormal{for}\ \ X\in \partial \left( D_f\intersect D_g \right),\ \ 
	\abs{W_f-W_g}\leq C(\gam,m,d)\norm{f-g}_{L^\infty}.
\end{align*}
Hence by the comparison result in Proposition \ref{propAppendix:weak-max-principle}, we have
\begin{align}\label{eqAuxOp:LipProofWfMinusWg}
	\norm{W_f-W_g}_{L^\infty(D_f\intersect D_g)}\leq C(\gam,m,d)\norm{f-g}_{L^\infty(\real^d)}.
\end{align}
 
To establish (\ref{eq:Lip to do}) we first consider  $X=(x,x_{d+1})\in D_f\intersect D_w\cap B_{R_0}\intersect D_g$. Since $g\in\K^*(\gam,m)$  Corollary \ref{corAuxOp:WfLipOnK*Part2} yields  $\norm{\grad W_g}_{L^\infty(\real^d)}\leq C(\gam,m,d)$, and thus
\begin{align*}
	|W_g(X)-V(X)|=\abs{W_g(X) - W_g(\Rmc^{-1}X)}&\leq C(\gam,m,d)\abs{X - \Rmc^{-1} X}\\
	&\leq C(\gam,m,d) \norm{\Rmc}\abs{X} \\
	&\leq C(\gam,m,d, R_0) \abs{\grad f(0) - \grad g(0)},
\end{align*}
where to obtain the final inequality we used that $|X|\leq R_0$ and $R_0$ is a fixed parameter that does not depend on $m$, $\gamma$, and $d$. This estimate, together with  (\ref{eqAuxOp:LipProofWfMinusWg}), imply,
\begin{align}
\label{eq:onDg}
\|W_f -V\|_{L^\infty (D_f\intersect D_w\cap B_{R_0} \intersect D_g)}&\leq \|W_f -W_g\|_{L^\infty (D_f\intersect D_w\intersect D_g)}+\|W_g -V\|_{L^\infty (D_f\intersect D_w\intersect D_g)} \nonumber \\
&\leq C(\gam, m, d, R_0) \left(\norm{f-g}_{C^{1,\gam}(B_{R_0})} + \|f-g\|_{L^\infty(\real^d)}\right).
\end{align}

We continue  with establishing (\ref{eq:Lip to do}). Let $X=(x,x_{d+1}) \in D_f\intersect D_w\cap B_{R_0}\cap D_g^c$, so that 
\begin{align*}
	g(x) < x_{d+1} \leq f(x).
\end{align*}
Letting 
 $\hat X = (x,g(x))\in\overline D_g$, we note, \begin{align*}
	\abs{ (W_f - V)(X) - (W_f-V)(\hat X) } 
	&\leq \norm{\grad (W_f-V)}_{L^\infty(B_{R_0})}\abs{X-\hat X}\\
	&\ \ = \norm{\grad (W_f-V)}_{L^\infty(B_{R_0})}\abs{f(x)-g(x)}\\
	&\ \ \leq C(\gam,m,d)\abs{f(x)-g(x)},
\end{align*}
where in the last inequality, we used Corollary \ref{corAuxOp:WfLipOnK*Part2} for $W_f$ and standard $L^\infty\text{-}C^{0,1}$ estimates for the harmonic function $V$. Using this estimate, together with (\ref{eq:onDg}), yields,
\begin{align*}
|(W_f-V)(X)| &\leq |(W_f-V)(X)-(W_f-V)(\hat X)|+|(W_f-V)(\hat X)| \\
&\leq C(\gam,m,d)\abs{f(x)-g(x)} + C(\gam, m, d, R_0) \left(\norm{f-g}_{C^{1,\gam}(B_{R_0})} + \|f-g\|_{L^\infty(\real^d)}\right),
\end{align*}
so that
\[
\|W_f -V\|_{L^\infty (D_f\intersect D_w\cap B_{R_0} \intersect D_g^c)}\leq C(\gam,m, d, R_0) \left(\norm{f-g}_{C^{1,\gam}(B_{R_0})} + \|f-g\|_{L^\infty(\real^d)}\right).
\]
The previous line and  (\ref{eq:onDg}) yield (\ref{eq:Lip to do}), completing the proof.
\end{proof}

%%%%%%%%%%%%%%%%%%%%%%%%%%%%%%%%%%%%%%%%%%%%%%%%%
%%%%%%%%%%%%%%%%%%%%%%%%%%%%%%%%%%%%%%%%%%%%%%%%%
%%%%%%%%%%%%%%%%%%%%%%%%%%%%%%%%%%%%%%%%%%%%%%%%%
%%%%%%%%%%%%%%%%%%%%%%%%%%%%%%%%%%%%%%%%%%%%%%%%%

\subsection{The splitting property}\label{subsection:splitting-J}

We verify a useful splitting property  of our operator $H$ in this section.  It shows how nearby a location of evaluation $H(f,x)$ only depends locally on the $C^{1,\gam}$ regularity of $f$ and the global $L^{\infty}$ behavior.   A key step is the following lemma.

\begin{lemma}\label{lemAuxOp:Jsplitting-property-equal-B2R} There exist  constants $C>0$ and $\alpha\in(0,1]$, depending only on $d$, $\gamma$, and $m$,  such that for all  $R \geq 1$, $y_0\in\mathbb{R}^d$, and all $f,g\in \K^*(\gamma, m)$, with $f-g\in C^{1,\gam}(m)$ and $f\equiv g$ in $B_{2R}(y_0)$, we have,
    \begin{align*}
        \norm{ H(f,\cdot) - H(g,\cdot)}_{L^\infty(B_{R}(y_0))} \leq C(\gamma, m, d)R^{-\alpha} \Vert f-g \Vert_{L^\infty(\mathbb{R}^d)}.
    \end{align*}
\end{lemma}

\begin{proof} 

Let $C$, $R_0$, $s_0$, and $\alpha$ be as given by  Lemma \ref{lemAppendix:BetterVersionASLemA.5}. Fix $R>R_0$. It suffices to prove that, for  any $x_0\in\R^d$ and any $f,g\in \mathcal{K}(\gamma, m)$ with $f\equiv g$ in $B_{R}(x_0)$, we have
\begin{align}\label{eq:J-splitting-x0-BR}
    |H(f, x_0) - H(g,x_0)| \leq CR^{-\alpha}\Vert f-g\Vert_{L^\infty(\mathbb{R}^d)}.
\end{align}
Indeed, suppose this holds, and we have $y_0\in \R^d$ and $f,g\in \mathcal{K}(\gamma, m)$ with $f\equiv g$ in $B_{2R}(y_0)$. Then for any $x_0\in B_R(y_0)$ we have $\tau_{x_0-y_0}(f) \equiv \tau_{x_0-y_0}(g)$ in $B_R(x_0)$. Using the translation-invariance of $H$ as well as (\ref{eq:J-splitting-x0-BR}), we obtain,
\begin{align*}
    |H(f, x_0) - H(g,x_0)| 
    &= \left|H(\tau_{x_0-y_0}f, x_0) - H(\tau_{x_0-y_0}g, x_0)\right| \\
    &\leq CR^{-\alpha}\Vert \tau_{x_0-y_0}f-\tau_{x_0-y_0}g\Vert_{L^\infty(\mathbb{R}^d)} =  CR^{-\alpha}\Vert f-g\Vert_{L^\infty(\mathbb{R}^d)},
\end{align*}
as desired.

So, let us take  any $x_0\in\R^d$ and any $f,g\in \K^*(\gamma, m)$ with $f\equiv g$ in $B_{R}(x_0)$; the remainder of the proof is devoted to showing that (\ref{eq:J-splitting-x0-BR}) holds.  To this end, let us define $V=U_f - U_g$ in the set $D_f\intersect D_g$.  We see that if we define $w= \min\{f,g\}$, then $w\in K^*(\gam,m)$, and $D_f\intersect D_g = D_w$.  Furthermore, by Lemma \ref{lemAuxOp:WfLipOnK*Part1}, 
\begin{align*}
	\textnormal{for}\ X\in\ \Gam_w,\ \ \abs{V(X)}\leq \abs{f(x)-g(x)}.
\end{align*}
Without loss of generality, we can assume $\norm{f-g}_{L^\infty}>0$.
Thus, the pair $V/\|f-g\|_{L^\infty(\real^d)}$ and $w$ satisfy the assumptions of Lemma \ref{lemAppendix:BetterVersionASLemA.5} (where actually $w$ is the $f$ in Lemma \ref{lemAppendix:BetterVersionASLemA.5}, and the upper bound in part (ii) of the lemma arises from the fact that $f$, and hence $w$, is in $\K^*$), 
yielding
\begin{align*}
	\abs{V(X_0+s\nu)}\leq \frac{Cs}{R^{\al}}\norm{f-g}_{L^\infty}.
\end{align*}
Recalling the definition of $V$ thus gives (\ref{eq:J-splitting-x0-BR}), completing the proof.
\end{proof}

With Lemma \ref{lemAuxOp:Jsplitting-property-equal-B2R} in hand,  we can do a cutoff argument to remove the condition $f\equiv g$ in $B_{R}(x_0)$.

\begin{theorem}\label{thmAuxOp:SplittingTheorem} 
There exist   constants $C>0$ and $\al\in(0,1)$, depending on $d$, $\gamma$, and $m$, such that for all $x_0\in\real^d$, $R\geq 1$, with $f,g\in \K^*(\gamma, m)$, with $f-g\in C^{1,\gam}(B_{3R}(x_0))$, 
    \begin{align*}
        \Vert H(f,\cdot) - H(g,\cdot) \Vert_{L^\infty(B_R(x_0))} \leq C\left( \Vert f-g\Vert_{C^{1,\gamma}(B_{3R}(x_0))} + R^{-\al} \Vert f-g \Vert_{L^\infty(\mathbb{R}^d)} \right).
    \end{align*}
\end{theorem}

\begin{proof} Let $\psi:\R^d\rightarrow [0,1]$ be  smooth, compactly supported  in $B_{3R}(0)$, with  $\psi \equiv 1$ in $B_{2R}(0)$, and such that $\Vert \nabla \psi\Vert_{L^\infty(\mathbb{R}^d)} \leq CR^{-1}$, $\Vert \nabla D^2\psi\Vert_{L^\infty(\mathbb{R}^d)} \leq CR^{-2}$, and $\Vert \psi\Vert_{C^{1,\gamma}(\mathbb{R}^d)}\leq C$, for a positive contant $C$ depending only on $d$. Let $ \hat{f} = \psi f+ (1-\psi)g$, so that
\begin{align*}
  \hat{f} - g = \psi(f-g). 
\end{align*}
Using the Lipschitz property of $H$ from Theorem \ref{thmAuxOp:HPreciseLipStatementC1Gam}, the fact that $\hat{f}\equiv g$ outside of $B_{3R}(x_0)$, and the previous line, we obtain
\begin{align*}
    \Vert H(\hat{f}, \cdot) - H(g, \cdot)\Vert_{L^\infty(B_{3R}(x_0))} 
    &\leq C\Vert \hat{f} - g\Vert_{C^{1,\gamma}(B_{3R}(x_0))}  \\
    &\leq C\Vert \psi (f-g)\Vert_{C^{1,\gamma}(B_{3R}(x_0))}
    \leq C\Vert f-g\Vert_{C^{1,\gamma}(B_{3R}(x_0))},
\end{align*}
where $C$ is a positive constant that  may change from line to line and depends only on $d$, $\gamma$, $m$. We also note that, to obtain the final inequality we used the definition of the $C^{1,\gamma}$ norm as well as the elementary property that  for any H\"older continuous functions $h_1$ and $h_2$, we have, 
\[
[h_1h_2]_{C^{0,\gamma}}\leq [h_1]_{C^{0,\gamma}}\Vert h_2\Vert_{L^\infty} + [h_2]_{C^{0,\gamma}}\Vert h_1\Vert_{L^\infty}.
\]

On the other hand, since $\hat{f}\equiv f$ inside $B_{2R}(x_0)$, Lemma \ref{lemAuxOp:Jsplitting-property-equal-B2R} yields,
\begin{align*}
    \Vert H(\hat{f},\cdot) - H(f, \cdot)\Vert_{L^\infty(B_{R}(x_0))} \leq C\omega(R)\Vert \hat{f}-f\Vert_{L^\infty(\mathbb{R}^d)} 
    \leq CR^{-\alpha}\Vert f-g\Vert_{L^\infty(\mathbb{R}^d)} 
\end{align*}
where the  constants $C>0$ and $\gamma\in (0,1]$ depend only on $d$, $\gamma$, and $ m$. We therefore obtain
\begin{align*}
\Vert H(f,\cdot) - H(g,\cdot) \Vert_{L^\infty(B_R(x_0))} & \leq
    \Vert H(\hat{f},\cdot) - H(f, \cdot)\Vert_{L^\infty(B_{R}(x_0))} +\Vert H(\hat{f}, \cdot) - H(g, \cdot)\Vert_{L^\infty(\mathbb{R}^d)}  \\
    & \leq C\left( \Vert f-g\Vert_{C^{1,\gamma}(B_{3R}(x_0))} + R^{-\alpha} \Vert f-g \Vert_{L^\infty(\mathbb{R}^d)} \right),
\end{align*}
as desired.
\end{proof}

Thanks to the fact that $H$ is invariant by the addition of constants, the previous result can be made more precise.

\begin{corollary}\label{corAuxOp:SplittingTheorem-OscillationVersion}
	If $x_0$, $R$, $f$, $g$, and $C$ are all as in Theorem \ref{thmAuxOp:SplittingTheorem}, then
	\begin{align*}
		 \Vert H(f,\cdot) - H(g,\cdot) \Vert_{L^\infty(B_R(x_0))} \leq C\left(
		 \osc_{B_{3R}(x_0)} (f-g) + 
		 \norm{\grad f-\grad g}_{C^{\gam}(B_{3R})} + 
		 R^{-\al} \Vert f-g \Vert_{L^\infty(\mathbb{R}^d)} 
		 \right).
	\end{align*}
\end{corollary}

\begin{proof}[Proof of Corollary \ref{corAuxOp:SplittingTheorem-OscillationVersion}]
	
	We observe that for the constant, $c$, chosen as $c = \inf_{B_{3R}(x_0)}(f-g)$, we have
	\begin{align*}
		\norm{(f-g)-c}_{L^\infty(B_{3R}(x_0))} = \osc_{B_{3R}} (f-g).
	\end{align*}
	We can then apply Theorem \ref{thmAuxOp:SplittingTheorem} to the functions $(f-c)$ and $g$.
	We note that
	\begin{align*}
		\norm{(f-c) - g}_{L^\infty(\real^d)} 
		\leq 2\norm{f-g}_{L^\infty(\real^d)}.
	\end{align*}
	This establishes the corollary.	
\end{proof}

%%%%%%%%%%%%%%%%%%%%%%%%%%%%%%%%%%%%%%%%%%%%%%%%%
%%%%%%%%%%%%%%%%%%%%%%%%%%%%%%%%%%%%%%%%%%%%%%%%%
%%%%%%%%%%%%%%%%%%%%%%%%%%%%%%%%%%%%%%%%%%%%%%%%%
%%%%%%%%%%%%%%%%%%%%%%%%%%%%%%%%%%%%%%%%%%%%%%%%%
%%%%%%%%%%%%%%%%%%%%%%%%%%%%%%%%%%%%%%%%%%%%%%%%%
%%%%%%%%%%%%%%%%%%%%%%%%%%%%%%%%%%%%%%%%%%%%%%%%%
%%%%%%%%%%%%%%%%%%%%%%%%%%%%%%%%%%%%%%%%%%%%%%%%%
%%%%%%%%%%%%%%%%%%%%%%%%%%%%%%%%%%%%%%%%%%%%%%%%%
%%%%%%%%%%%%%%%%%%%%%%%%%%%%%%%%%%%%%%%%%%%%%%%%%

\section{Viscosity solutions for (\ref{eqIntro:MuskatHJB}) and (\ref{eqIntro:HeleShawHJB}), and the proof of Theorem \ref{thmIntro:MuskatExistUnique}}\label{sec:ViscositySolsAndMainProof}

In this section we give the definition and basic properties of viscosity solutions for equations like, and including, (\ref{eqIntro:MuskatHJB}) and (\ref{eqIntro:HeleShawHJB}).  This section will culminate in the proof of Theorem \ref{thmIntro:MuskatExistUnique}, and it follows similarly to \cite[Sections 8.3 and 10]{ChangLaraGuillenSchwab-2019SomeFBAsNonlocalParaboic-NonlinAnal}, which was based on the arguments in \cite{Silv-2011DifferentiabilityCriticalHJ}.   We focus on (\ref{eqIntro:HeleShawHJB}) and we will conclude by showing that solving (\ref{eqIntro:HeleShawHJB}) is equivalent to solving (\ref{eqIntro:MuskatHJB}).  To this end, we restate this equation without initial conditions
\begin{align}\label{eqViscSol:HeleShawHJB}
	\partial_t f = H(f)\ \ \textnormal{in}\ \ \real^d\times (0,T),
\end{align}
and we will develop the viscosity solutions theory for (\ref{eqViscSol:HeleShawHJB}).

%%%%%%%%%%%%%%%%%%%%%%%%%%%%%%%%%%%%%%%%%%%%%%%%%
%%%%%%%%%%%%%%%%%%%%%%%%%%%%%%%%%%%%%%%%%%%%%%%%%
%%%%%%%%%%%%%%%%%%%%%%%%%%%%%%%%%%%%%%%%%%%%%%%%%
%%%%%%%%%%%%%%%%%%%%%%%%%%%%%%%%%%%%%%%%%%%%%%%%%

\subsection{Basic definitions and useful properties}

Due to the nonlocal nature of integro-differential equations, the definitions of test functions and viscosity solutions have to naturally balance regularity at a point with global behavior.  This is something that is not seen for local equations.  There are many variations on choices of test functions, and in most reasonable situations, they give rise to equivalent definitions of solutions.  We use a definition of test functions that is a parabolic and $C^{1,\gam}$ version of (and equivalent to) those in \cite[Definition 1, Remark 1]{BaIm-07}, \cite[Definition 2.2]{CaSi-09RegularityIntegroDiff}, \cite[Definition 2.2]{Silv-2011DifferentiabilityCriticalHJ}.   We note that in all instances of viscosity solutions  \emph{for integro-differential equations of order strictly less than 2}, of which we are aware, the natural notion of test function extends to functions which may only be punctually regular at the point of contact.  One reason we choose the definition given here is that it works well with the Perron method to give existence of solutions.

\begin{definition}[Test functions] \label{defViscSol:test-functions}

We denote by $C^{1,\gamma}(B_r(x_0))\cap C^{0,1}(\mathbb{R}^d)$ the Banach space consisting of functions in $C^{1,\gamma}(B_r(x_0))\cap C^{0,1}(\mathbb{R}^d)$ with the norm
	\begin{align*}
	\Vert f\Vert_{C^{1,\gamma}(B_r(x_0))\cap C^{0,1}(\mathbb{R}^d)} = 
	\Vert f\Vert_{C^{1,\gamma}(B_r(x))} 
	+ 
	\Vert f\Vert_{C^{0,1}(\mathbb{R}^d)}.
	\end{align*}
We say that $\phi$ is a test function at $(x,t)\in \mathbb{R}^d\times (0,T)$ if there exists $r>0$ such that
    \begin{align*}
        \phi\in C\left((t-r, t+r);  C^{1,\gamma}(B_r(x)) 
        \cap C^{0,1}(\mathbb{R}^d)\right), 
        \ \ \ \textnormal{and}\ \ \ 
        \partial_t \phi\in C\left(B_r(x)\times (t-r,t+r) \right).
    \end{align*}
For $u:\mathbb{R}^d\times (0,T) \to \mathbb{R}$, we say that $\phi$ touches $u$ from above (resp. below) at $(x,t)$ if there exists $r>0$ with
    \begin{align*}
    \begin{cases}
            u(x,t) = \phi(x,t),\\
            u(y,s) \leq \phi(y,s) \quad \textnormal{for all}\ (y,s)\in \real^d\times (t-r,t+r) \ (\textnormal{resp.}\  u(y,s) \geq \phi(y,s)).
    \end{cases}
    \end{align*}
\end{definition}

We now define viscosity solutions of (\ref{eqViscSol:HeleShawHJB}).   This definition applies to any operator similar to $H$ or $M$ that enjoys the GCP, and we will use it for (\ref{eqIntro:HeleShawHJB}) and subsequently (\ref{eqIntro:MuskatHJB}).  As such, we temporarily give a generic equation to use in the definition:
\begin{align}\label{eqViscSol:HJB-GenericJ}
	\partial_t f = J(f)\ \ \textnormal{in}\ \ \real^d\times (0,T).
\end{align}
For us, we will always use either $J=H$ or $J=M$, which will be clear from the context.

\begin{definition}[Viscosity sub and super solutions]\label{defViscSol:ViscositySolutions}

\begin{enumerate}[(i)]
    \item A function $u:\mathbb{R}^d\times (0,T)\to \mathbb{R}$ is a \emph{viscosity subsolution} of  \eqref{eqViscSol:HJB-GenericJ} if $u$ is bounded, upper semicontinuous, %(hereafter, $u\in USC$) 
    and for every $(x,t)\in \mathbb{R}^d\times (0,T)$, any $\phi$ that is a test function touching $u$ from above at $(x,t)$ satisfies 
    \begin{align*}
        \partial_t \phi(x,t) \leq J(\phi(\cdot, t),x).   
    \end{align*}
    
    \item A function $u:\mathbb{R}^d\times (0,T)\to \mathbb{R}$ is a \emph{viscosity supersolution} of  \eqref{eqViscSol:HJB-GenericJ} if $u$ is bounded, lower semicontinuous, % (hereafter $u\in LSC$) 
    and for every $(x,t)\in \mathbb{R}^d\times (0,T)$, any $\phi$ that is a test function touching $u$ from below at $(x,t)$ satisfies 
    \begin{align*}
        \partial_t \phi(x,t) \geq J(\phi(\cdot, t),x).   
    \end{align*}

    \item We say that $u \in C(\mathbb{R}^d\times (0,T))$ is a \emph{viscosity solution} of \eqref{eqViscSol:HJB-GenericJ} if it is both a subsolution and a supersolution of \eqref{eqViscSol:HJB-GenericJ}.
\end{enumerate}

\end{definition}

As alluded to above, the subsolution (supersolution) property for test functions still remains valid when the test function may not be regular in an entire neighborhood of a point of contact with the subsolution (supersolution), but rather may only have a pointwise version of regularity, as in Definition \ref{defAuxOp:PunctuallyC1Gam}. See Section \ref{ss:punctually} for a discussion of other instances of this property in the literature.

\begin{proposition}[Pointwise evaluation] \label{propViscSol:pointwise-testing} 
Let $u$ be a viscosity subsolution to \eqref{eqViscSol:HeleShawHJB} in the sense of Definition \ref{defViscSol:ViscositySolutions}, and $(x_0,t_0)\in \R^d\times (0,T)$. Let $\psi$ be a bounded function function such that
\begin{align}\label{eq:PointwiseEvaluationLipschitzOnIntervalF}
	\psi\in C^{1,\gamma}(x_0,t_0)\cap C^{0,1}(\R^d\times I)
\end{align}
where $I = (t_0-r,t_0+r)$ for some $r>0$, and $\psi$ touches $u$ from above at $(x_0,t_0)$. Then
\begin{align*}
    \partial_t \psi(x_0,t_0) \leq H(\psi(\cdot, t_0), x_0).
\end{align*}
\end{proposition}

\begin{proof}[Proof of Proposition \ref{propViscSol:pointwise-testing}]

By the assumption on $\psi$, for some $r_0\in(0,r)$,
	\begin{align*}
		\abs{\psi(x_0+h,t_0+s) - \psi(x_0,t_0) - \grad_{x,t}\psi(x_0,t_0)\cdot (h,s) }\leq m\abs{(h,s)}^{1+\gam}\qquad		\textnormal{for}\ \abs{(h,s)}<r_0,
	\end{align*}
		where $\nabla_{x,t} = (\nabla , \partial_t)$ is the total gradient in space-time. For $\rho \in (0,r_0)$, we define
	\begin{align*}
		v_\rho(y,s) = 
		\begin{cases}
			\psi(x_0,t_0) + \grad_{x,t}\psi(x_0,t_0)\cdot\big((y,s)-(x_0,t_0)\big) + m\abs{(y,s)-(x_0,t_0)}^{1+\gam}\  &\text{for}\ (y,s)\in B_\rho(x_0,t_0)\\
			\psi(y,s) &\text{for}\ (y,s)\notin B_\rho(x_0,t_0).
		\end{cases}
	\end{align*}
	Then, we define $w_\rho:\R^d\times I \to \R$ by
	\begin{align*}
		w_\rho(y,s) = \min\left\{ w(y,s)\ : w\geq v_\rho, \ w\in C^{1,\gamma}(\R^d\times I) \ \textnormal{and}\ \Vert w\Vert_{C^{0,1}(\R^d)}\leq \Vert \psi\Vert_{C^{0,1}(\R^d)} \right\}.
	\end{align*}
	Since $v_\rho\in C^{1,\gam}(B_\rho(x_0,t_0))$, we see that also $w_\rho = v_\rho$ in $B_\rho(x_0,t_0)$. We observe that 
	\begin{align*}
		w_\rho(x_0,t_0)=\psi(x_0,t_0),\ \ 
		w_\rho\geq \psi\ \ \textnormal{on}\ \ \real^d\times I,
		\ \ \textnormal{and}\ \ 
		\lim_{\rho\to 0} \norm{w_\rho(\cdot,t)-\psi(\cdot,t)}_{L^\infty(\real^d)}=0,
	\end{align*}
for $t\in I$. Additionally, $w_\rho$ can be chosen so that $\partial_t w_\rho$ is continuous on $I$. That is to say, $w_\rho$ is a valid test function (Definition \ref{defViscSol:test-functions}).  By construction, $w_\rho$ touches $u$ from above at $(x,t)$. Thus, the definition of subsolution gives
	\begin{align*}
		\partial_t \psi(x_0,t_0) = \partial_t w_\rho(x_0,t_0) \leq H(w_\rho(\cdot,t_0), x_0).
	\end{align*} 
	By construction, $w_\rho$ and $\psi$ satisfy the assumptions of Lemma \ref{lemPtWise:NormalDerivLinftyDependence}, which implies that
	\begin{align*}
		\lim_{\rho\to 0} \abs{H(w_\rho(\cdot,t_0),x_0) - H(\psi(\cdot,t_0),x_0)}=0.
	\end{align*}
	This concludes the proposition.
\end{proof}

\begin{definition}[Inf and sup convolutions]\label{defViscSol:InfSupConvoution}
	
	For a bounded function $u(x,t): \mathbb{R}^d\times [0,T) \to \mathbb{R}$ we define the sup convolution and the inf convolution as follows.
\begin{align*}
    u^\ep(x,t) &= \sup_{(y,s)\in \mathbb{R}^d\times [0,T)} \left(u(y,s) -  \frac{|t-s|^2+|x-y|^2}{2\varepsilon}\right),\\
    u_\varepsilon(x,t) &= \inf_{(y,s)\in \mathbb{R}^d\times [0,T)} \left(u(y,s) +  \frac{|t-s|^2+|x-y|^2}{2\varepsilon}\right).
\end{align*}
\end{definition}

It is well-known that $u^\varepsilon$ is semiconvex and $u_\varepsilon$ is semiconcave, thus they are continuous. Furthermore, they are uniformly Lipschitz with Lipschitz constant $C\varepsilon^{-1}$ for $C>0$ depending only on $\Vert u\Vert_{L^\infty(\mathbb{R}^d\times [0,T))}$.  We refer the to User's Guide to Viscosity Solutions \cite{CrandalIshiLions-92UsersGuide} for a complete description of their properties.  We also note that although the regularization by inf and sup convolutions predates the notion of viscosity solutions, they have been very useful in the existence and uniqueness theory, dating back to \cite{Jensen-1988UniquenessARMA}, \cite{JeLiSo-88UniquenessSecondOrder}.  This type of regularization works particularly well for translation invariant equations, like (\ref{eqViscSol:HeleShawHJB}), as demonstrated in the next result. A more precise result holds for more general equations that may not be translation invariant, such as in \cite{JeLiSo-88UniquenessSecondOrder}, but since (\ref{eqViscSol:HeleShawHJB}) is both translation invariant and invariant by the addition of constants, we are in a considerably easier situation.

\begin{proposition}\label{propViscSol:subsolution-of-sup-convolution} If $u$ and $v$ are bounded and respectively an upper semicontinuous subsolution and a lower semicontinuous supersolution to \eqref{eqViscSol:HeleShawHJB}, then so are their sup-convolution, $u^\varepsilon$, and inf-convolution, $v_\ep$.
\end{proposition}

\begin{proof}[Proof of Proposition \ref{propViscSol:subsolution-of-sup-convolution}] 
	
	We will only include the proof for $u^\ep$.  The result for $v_\ep$ follows analogously.
	We note that $u^\ep(x)$ can be rewritten as
	\begin{align*}
		u^\ep(x,t) = \sup_{z\in\real^d,\ t+\tau\in [0,T)}\left(  u(x+z, t+\tau) - \frac{1}{2\ep}\left(\abs{z}^2 + \abs{\tau}^2\right)  \right).
	\end{align*}
	The translation invariance and invariance by the addition of constants for  $H$, given in Proposition \ref{propAuxOp:HProperties} --- appropriately applied to test functions --- shows that $u^\ep$ and $u$ solve the same equation.
\end{proof}

%%%%%%%%%%%%%%%%%%%%%%%%%%%%%%%%%%%%%%%%%%%%%%%%%
%%%%%%%%%%%%%%%%%%%%%%%%%%%%%%%%%%%%%%%%%%%%%%%%%
%%%%%%%%%%%%%%%%%%%%%%%%%%%%%%%%%%%%%%%%%%%%%%%%%
%%%%%%%%%%%%%%%%%%%%%%%%%%%%%%%%%%%%%%%%%%%%%%%%%

\subsection{The comparison result}

The comparison result we present (Proposition \ref{propViscSol:Comparison-HeleShawHJB}, below) is an adaptation of that in \cite[Proposition A.6]{Silv-2011DifferentiabilityCriticalHJ}, which was also implemented in \cite[Section 8.3]{ChangLaraGuillenSchwab-2019SomeFBAsNonlocalParaboic-NonlinAnal}.  We note that there are two versions of the comparison result in \cite{Silv-2011DifferentiabilityCriticalHJ} --- in our context, one, \cite[Corollary 3.4]{Silv-2011DifferentiabilityCriticalHJ}, would be appropriate if $H$ happened to be a globally Lipschitz operator on $C^{1,\gam}$ (which it is not), and the other one, \cite[Proposition A.6]{Silv-2011DifferentiabilityCriticalHJ} is required in the case that $H$ is only locally Lipschitz on $C^{1,\gam}$ (which is our setting).

Before we get to the comparison result, it will be useful to note that for a test function $\psi$, the map $(x,t)\mapsto H(\psi(\cdot,t),x)$ is continuous.  In the typical instances of viscosity solutions,  this is usually immediate from the fact that the relevant operator has an explicit formula.  

\begin{proposition}[Continuity on test functions]\label{propViscSol:ContinuityOfH} 
    Suppose $(x_0, t_0)\in \mathbb{R}^d\times (0,T)$ and $\psi$ is a test function at $(x_0,t_0)$ in the sense of Definition \ref{defViscSol:ViscositySolutions}, then the function $(x,t)\mapsto H(\psi(\cdot, t), x)$ is continuous near $(x_0,t_0)$.
\end{proposition}

\begin{proof}[Proof of Proposition \ref{propViscSol:ContinuityOfH}] It suffices to show the continuity at $(x_0,t_0)$. Let $(\hat{x}, \hat{t}) \in B_r(x_0,t_0)$. We have
\begin{align*}
	&\abs{ H(\psi(\cdot,\hat{t}),\hat{x}) - H(\psi(\cdot,t_0), x_0)}\\
	&\qquad\qquad \leq
	\abs{ H(\psi(\cdot,\hat{t}), \hat{x}) - H(\psi(\cdot,t_0), \hat{x}) }
	+ \abs{ H(\psi(\cdot,t_0), \hat{x}) - H(\psi(\cdot,t_0, x_0) } = J_1 + J_2.
\end{align*}
The result then follows from Theorem \ref{thmAuxOp:HPreciseLipStatementC1Gam} for $J_1$, and from Proposition \ref{propAuxOp:BaiscRegularityForHC1GamToCgam} for $J_2$.
\end{proof}

\begin{rem}[A bump function]\label{remAuxOp:bump-function} 
	It will be useful to record a particular bump function for some of the next results.
		Let us define for $x\in \mathbb{R}^d$ the function
		\begin{align*}
		    \Phi(x) = \frac{|x|^2}{1+|x|^2} \qquad\text{and}\qquad 
		    	\Phi_R(x) = \Phi\left(\frac{x}{R}\right)\qquad\text{for}
		    	\;R>1.
		\end{align*}
		The function $\Phi_R$ satisfies $0\leq \Phi_R\leq 1$, and for any $\gamma\in (0,1)$,
		\begin{align*}
		    \Vert \Phi_R\Vert_{L^\infty(B_T(x_0))} \leq \frac{T^2}{R^2}, 
		    \qquad
		    \Vert \nabla \Phi_R\Vert_{L^\infty(\mathbb{R}^d)} \leq 
		    \frac{1}{R}, 
		    \qquad \Vert\nabla \Phi_R\Vert_{C^{0,\gamma}(\mathbb{R}^d)} 
		    \leq \frac{C(d,\gamma)}{R^{1+\gamma}}.
		\end{align*}
\end{rem}

Now we can give the main result of this section.

\begin{proposition}[Comparison property] \label{propViscSol:Comparison-HeleShawHJB}
	
	If $f$ and $g$ are bounded and respectively an upper semicontinuous subsolution and a lower semicontinuous supersolution in the viscosity sense to \eqref{eqViscSol:HeleShawHJB}, and have the following local uniform ordering at $t=0$,
	\begin{align}\label{eqViscSol:InitialDataLocalUniformOrdering}
		\forall\ \al>0,\ \exists\ \del>0,\ \forall\ \abs{x-y}<\del,\ \abs{s-0}<\del,\  \abs{t-0}<\del,\
		f(x,s)\leq g(y,t) + \al, 
	\end{align}
	  then $f(x,t) \leq g(x,t)$ for all $(x,t)\in \mathbb{R}^d\times [0,T)$.
\end{proposition}

\begin{rem}
\label{rem:initial condition}
We note the condition (\ref{eqViscSol:InitialDataLocalUniformOrdering}) was used in \cite[Appendix]{Silv-2011DifferentiabilityCriticalHJ} to establish existence for a certain class of integro-differential equations.
	The reason for the extra condition in (\ref{eqViscSol:InitialDataLocalUniformOrdering}) is a bit subtle.  If if was known, a priori, that $f$ and $g$ had more regularity, then (\ref{eqViscSol:InitialDataLocalUniformOrdering}) would not be necessary.  Or, if it was true that $H$ were globally Lipscshitz on $C^{1,\gam}(\real^d)$, (\ref{eqViscSol:InitialDataLocalUniformOrdering}) would not be necessary.  However, since $H$ is only locally Lipschitz on $C^{1,\gam}$, (\ref{eqViscSol:InitialDataLocalUniformOrdering}) is necessary.

\end{rem}

\begin{rem}
	The extra assumption of (\ref{eqViscSol:InitialDataLocalUniformOrdering}) is missing from the work \cite{ChangLaraGuillenSchwab-2019SomeFBAsNonlocalParaboic-NonlinAnal}.  The stated result remains true, but the presentation misses the importance of (\ref{eqViscSol:InitialDataLocalUniformOrdering}).
\end{rem}

In order to clarify the role of the assumption (\ref{eqViscSol:InitialDataLocalUniformOrdering}),  we split Proposition \ref{propViscSol:Comparison-HeleShawHJB} into two separate lemmas.  The first one shows that even if $f$ and $g$ don't satisfy \eqref{eqViscSol:InitialDataLocalUniformOrdering}, their inf and sup convolutions still satisfy a typical comparison result.

\begin{lemma}\label{lemViscSol:ComparisonPart1-OrderingSupInfConvolution}
	If $f$ and $g$ are bounded and respectively a upper semicontinuous subsolution and a lower semicontinuous supersolution in the viscosity sense to \eqref{eqViscSol:HeleShawHJB}, and if $f^\ep$ and $g_\ep$ are respectively their sup and inf convolutions from Definition \ref{defViscSol:InfSupConvoution}, then
	\begin{align}\label{eqViscSol:ComparisonForInfSupConvolution}
		\sup_{\real^d\times [0,T)}( f^\ep - g_\ep) \leq 
		\sup_{\real^d}( f^\ep(\cdot,0) - g_\ep(\cdot,0) ).
	\end{align}
\end{lemma}

The second result states that under the additional regularity in (\ref{eqViscSol:InitialDataLocalUniformOrdering}), the comparison result can be transferred back to $f$ and $g$.

\begin{lemma}\label{lemViscSol:ComparisonPart2-InitialData}
	If $f$ and $g$ are bounded and respectively upper semicontinuous and lower semicontinuous, and if $f$ and $g$ satisfy (\ref{eqViscSol:InitialDataLocalUniformOrdering}), then
	\begin{align}\label{eqViscSol:LimitOfInfSupConvolutionInitialData}
		\limsup_{\ep\to0}  \left( \sup_{\real^d} f^\ep(\cdot,0) - g_\ep(\cdot,0) \right)
		\leq 0 .
	\end{align}
\end{lemma}

We will first show how Lemmas \ref{lemViscSol:ComparisonPart1-OrderingSupInfConvolution} and \ref{lemViscSol:ComparisonPart2-InitialData} give Proposition \ref{propViscSol:Comparison-HeleShawHJB}.  Then we will prove Lemmas \ref{lemViscSol:ComparisonPart1-OrderingSupInfConvolution} and \ref{lemViscSol:ComparisonPart2-InitialData} afterwards.

\begin{proof}[Proof of Proposition \ref{propViscSol:Comparison-HeleShawHJB}]
	
	We take $f^\ep$ and $g_\ep$ to be respectively the sup and inf convolutions of $f$ and $g$.  Lemmas \ref{lemViscSol:ComparisonPart1-OrderingSupInfConvolution} and \ref{lemViscSol:ComparisonPart2-InitialData} combine to give
	\begin{align*}
		\limsup_{\ep\to0}\sup_{\real^d\times[0,T)}(f^\ep - g_\ep)\leq 0.
	\end{align*} 
	However, by construction, $(f-g)\leq f^\ep - g_\ep$, and thus we conclude that $\sup_{\real^d\times[0,T)} f-g\leq 0$.
\end{proof}

	Now we can prove Lemmas \ref{lemViscSol:ComparisonPart1-OrderingSupInfConvolution} and \ref{lemViscSol:ComparisonPart2-InitialData}.
	
\begin{proof}[Proof of Lemma \ref{lemViscSol:ComparisonPart1-OrderingSupInfConvolution}]
	
	Let $\ep>0$, and $f^\ep$ be the sup-convolution of $f$ and $g_\ep$ the inf-convolution of $g$, as in Definition \ref{defViscSol:InfSupConvoution}. By Proposition \ref{propViscSol:subsolution-of-sup-convolution} we see that $f^\varepsilon, g_\varepsilon$ are respectively a subsolution and supersolution to \eqref{eqViscSol:HeleShawHJB} in the viscosity sense.

We will obtain (\ref{eqViscSol:ComparisonForInfSupConvolution}) by contradiction. Let us define $M$ as
\begin{align*}
	M = 
	\sup_{\real^d\times[0,T)}(f^\ep - g_\ep) - \sup_{\real^d\times\{0\}} (f^\ep - g_\ep),
\end{align*} 
and our contradiction assumption is that $M>0$.
There exist  $x_0\in\real^d$, $h>0$ small enough (depending on $M$ and $T$), and a constant $C_h$ such  that, for $\psi$ given by
\begin{align*}
    \psi(x,t) = C_h +  M\Phi_{R}(x-x_0) + ht, 
\end{align*}
the function $f^\ep - g_\ep - \psi$ attains a zero maximum at $(x_R, t_R)$ with $t_R>0$.  This fact follows from  Lemma \ref{lemViscSol:BumpToAttainMax}, which we state (and provide its elementary proof)  after completing the present proof.

Using the semi-convexity and semi-concavity of $f^\varepsilon$ and $g_\varepsilon$, combined with the fact that $\psi$ is smooth, we deduce that $f^\varepsilon, g_\varepsilon$ have classical derivatives with respect to $x,t$ at $(x_R,t_R)$ and furthermore $f^\varepsilon$ and $g_\varepsilon$ are $C^{1,1}(x_R, t_R)$. Hence, by Lemma \ref{LemPtWise:NonTanBehavior}, $H(f^\varepsilon(\cdot, t_R), x_R)$ and $H(g_\varepsilon(\cdot, t_R), x_R)$ are well defined. Furthermore, at $(x_R,t_R)$, by Proposition \ref{propViscSol:pointwise-testing}, we have 
    \begin{align*}
        H(g_\varepsilon(\cdot, t_R), x_R) \leq \partial_t g_\varepsilon(x_R,t_R)
		\ \ \ \ \textnormal{and}\ \ \ \ 
		\partial_t f^\varepsilon(x_R,t_R) \leq H(f^\varepsilon(\cdot, t_R), x_R).
    \end{align*}
	Next, thanks to the fact that $f^\varepsilon - g_\varepsilon - \psi$ attains a maximum,
    \begin{align*}
        \partial_t (g_\varepsilon + \psi)(x_R,t_R) = \partial_t f^\varepsilon(x_R,t_R),
    \end{align*}
	and we can also invoke the punctual GCP (Lemma \ref{LemPtWise:PunctialGCP}) for $f^\varepsilon \leq g^\varepsilon + \psi$, to obtain
    \begin{align*}
        H(f^\varepsilon(\cdot, t_R), x_R) \leq H\big(g_\varepsilon(\cdot, t_R)+\psi(\cdot, t_R), x_R\big).
    \end{align*}
	Combining this, we obtain
	\begin{align*}
	    H\left( g_\varepsilon(\cdot, t_R), x_R \right) + \partial_t \psi(x_R,t_R) \leq  H\Big(g_\varepsilon(\cdot, t_R)+\psi(\cdot, t_R), x_R\Big),
	\end{align*}
	which, upon  rearranging becomes,
	\begin{align*}
		0<
	    h  
	    &\leq H\big(g_\varepsilon(\cdot, t_R)+\psi(\cdot, t_R), x_R\big) - H\left(g_\varepsilon(\cdot, t_R), x_R \right)\\
	    &=H\big(g_\varepsilon(\cdot, t_R)+ M\Phi_R(\cdot-x_0), x_R\big) - H\left(g_\varepsilon(\cdot, t_R), x_R \right). 
	    \end{align*}
	    Let  $\tilde{R}>0$. We apply Corollary \ref{corAuxOp:SplittingTheorem-OscillationVersion}  to bound the right-hand side of the previous line to find, for a constant $C$ independent of $R$ and $\tilde{R}$,
	    	    \begin{align*}
	    0<h&\leq C\left( \osc_{B_{3\tilde{R}}(x_R)} M \Phi_R(\cdot -x_0)+ 
		 \norm{\grad \Phi_R(\cdot-x_0)}_{C^{\gam}(B_{3\tilde{R}}(x_R))} +\tilde{R}^{-\alpha}\right).
	    \end{align*}
Using the definition of $\Phi_R$ we obtain,
	    	    \begin{align*}
	    0<h&\leq C  \left(\osc_{B_{3\tilde{R}R^{-1}}(x_R)} M \Phi_1(\cdot -x_0)+
		 R^{-1}\norm{\grad \Phi_1(\cdot-x_0)}_{C^{\gam}(B_{3\tilde{R}}(x_R))} +\tilde{R}^{-\alpha}\right).
	    \end{align*}   
	    Let us now choose $\tilde{R}$ large enough so that 
	    \[
	    C\tilde{R}^{-\alpha}<\frac{h}{4}.
	    \]
	 Now that $\tilde{R}$ is fixed, we choose $R$ large enough so that 
	 \[
	 C\left(\osc_{B_{3\tilde{R}R^{-1}}(x_R)} M \Phi_1(\cdot -x_0) +R^{-1}\norm{\grad \Phi_1(\cdot-x_0)}_{C^{\gam}(B_{3\tilde{R}}(x_R))} \right) <\frac{h}{4}.
	 \]
	 Combining the three previous inequalities yields $0<h<\frac{h}{2}$, which is the desired contradiction.
\end{proof}

\begin{proof}[Proof of Lemma \ref{lemViscSol:ComparisonPart2-InitialData}]

Let $\al>0$ be given.  By the assumption on $f(\cdot,0)$ and $g(\cdot,0)$, let $\del>0$ be such that (\ref{eqViscSol:InitialDataLocalUniformOrdering}) holds.  Without loss of generality, let us assume that the supremum of initial values is attained, i.e.
\begin{align*}
	\sup_{\real^d} f(\cdot,0) - g(\cdot) = f(\hat x,0) - g(\hat x,0).
\end{align*}
As $f$ is upper semicontinuous and $g$ is lower semicontinuous and both are bounded, we know that the sup and inf for $f^\ep$, $g_\ep$, are always attained; let those points be $(x_\ep,s_\ep)$ and $(y_\ep,t_\ep)$ respectively for $f^\ep$ and $g_\ep$ at $(\hat x,0)$.  One of the properties of the inf and sup convolutions is that
\begin{align*}
	\lim_{\ep\to0} \frac{\abs{x_\ep-x_0}^2 + (s_\ep - t_0)^2}{2\ep} = 0
	\ \ \ \ \textnormal{and}\ \ \ \
	\lim_{\ep\to0} \frac{\abs{y_\ep-x_0}^2 + (t_\ep - t_0)^2}{2\ep} = 0.
\end{align*}
In particular, 
\begin{align*}
	(x_\ep,s_\ep)\to (\hat x,0)\ \ \ \textnormal{and}\ \ \ 
	(y_\ep,t_\ep)\to (\hat x,0).
\end{align*}
Therefore, for $\ep$ small enough, depending upon $\al$ and $\del$, we obtain 
\begin{align*}
	\abs{x_\ep - \hat x} < \del,\ \
	\abs{y_\ep - \hat x} < \del,\ \ 
	s_\ep < \del,\ \ 
	t_\ep < \del.
\end{align*}
Thus, invoking (\ref{eqViscSol:InitialDataLocalUniformOrdering}), we see that
\begin{align*}
	\sup_{\real^d} f^\ep(\cdot,0)-g_\ep(\cdot,0) = f(x_\ep, s_\ep) - g(y_\ep, t_\ep)
	\leq \al.
\end{align*}
Since $\al>0$ was arbitrary, we conclude (\ref{eqViscSol:LimitOfInfSupConvolutionInitialData}).
\end{proof}

\begin{lemma}\label{lemViscSol:BumpToAttainMax} Let $u: \mathbb{R}^d\times [0,T)\to \mathbb{R}$ be upper semicontinuous, bounded from above, and satisfy
\begin{align*}
    M = \sup_{\mathbb{R}^d \times [0,T)} u - \sup_{\mathbb{R}^d } u(\cdot,0) > 0,
\end{align*}
Let $\Phi_R$ be the bump function from Remark \ref{remAuxOp:bump-function}. 
There exists some $(x_0,t_0)$, such that for all $0<h<\frac{M}{4T}$, for all $R>0$ (independent of $h$ and $M$), there exists $(x_h,t_h)$ so that the function
\begin{align*}
	\psi(x,t) = u(x,t) - M\Phi_R(x - x_0) - ht
\end{align*}  
attains a maximum at $(x_h,t_h)$, and $t_h\in (0,T]$.
\end{lemma}

\begin{proof}[Proof of Lemma \ref{lemViscSol:BumpToAttainMax}]

	Let us first choose $(x_0,t_0)$ to be such that 
	\begin{align*}
		u(x_0,t_0) \geq \sup_{\real^d\times [0,T)} u - \frac{1}{4}M.
	\end{align*}
	We first note that this implies $t_0>0$, which is immediate from the fact that
	\begin{align}\label{eqViscSol:ProofOfAttainMaxUBiggerThanInitialData}
		u(x_0,t_0) \geq \sup_{\real^d} u(\cdot,0) + \frac{3}{4}M.
	\end{align}
	We see then that 
	\begin{align*}
		\psi(x_0,t_0) = u(x_0,t_0)  - ht_0
		\geq \sup_{\real^d} u(\cdot,0) + \frac{3}{4}M - ht
		\geq\sup_{\real^d} u(\cdot,0) + \frac{3}{4}M - hT,
	\end{align*}
	and so for $h T\leq \frac{M}{4}$, 
	\begin{align}
	\label{eq:psi at x0}
		\psi(x_0,t_0) \geq \sup_{\real^d} u(\cdot,0) + \frac{1}{2}M
		\geq \sup_{\real^d} \psi(\cdot,0) + \frac{1}{2}M.
	\end{align}
	From the definition of $\Phi_R$, for $\abs{x}$ large enough, 
	\begin{align*}
		\psi(x,t) \leq u(x,t) - \frac{3}{4}M -ht 
		\leq \sup_{\real^d} \psi(\cdot,t) - \frac{3}{4}M \leq \psi(x_0,t_0)-\frac14 M.
	\end{align*}
	Therefore, there exists $\tau>0$ such that,
	\begin{align*}
		\sup_{\real^d\times[0,T)} \psi  = \sup_{B_\tau(0)\times[0,T)} \psi.
	\end{align*}
	Since $\psi$ is upper semicontinuous, it will  attain its maximum over the compact set $\overline{B}_\tau\times[0,T]$. This maximum will be at least as large as $\psi(x_0,t_0)$; thus, by (\ref{eq:psi at x0}), it must occur in $(0,T]$.
\end{proof}

%%%%%%%%%%%%%%%%%%%%%%%%%%%%%%%%%%%%%%%%%%%%%%%%%
%%%%%%%%%%%%%%%%%%%%%%%%%%%%%%%%%%%%%%%%%%%%%%%%%
%%%%%%%%%%%%%%%%%%%%%%%%%%%%%%%%%%%%%%%%%%%%%%%%%
%%%%%%%%%%%%%%%%%%%%%%%%%%%%%%%%%%%%%%%%%%%%%%%%%

\subsection{Existence}

We  prove existence via the Perron method for viscosity solutions.  Our presentation is almost exactly as that in \cite{ImbertSilvestre-2013IntroToFullyNonlinearParabolic} and \cite{Silv-2011DifferentiabilityCriticalHJ}, which are in turn adaptations of the techniques of \cite{Ishii-1987PerronForHamiltonJacobiDUKE}.   We include most of the details here for the sake of completeness.   

The choice to use the Perron method places requirements on the definition of the test functions and viscosity sub and super solutions.  Specifically, it is the reason why we require sub (super) solutions to be defined for only upper semicontinuous (lower semicontinuous) functions in Definition \ref{defViscSol:ViscositySolutions}, and also it is the reason why test functions should be classical in an entire neighborhood of a point of contact instead of only punctually $C^{1,\gam}$ at that point.

We have broken the  proof into smaller individual steps.  The first part gives the details of how the supremum (or infimum) of a general family of subsolutions (supersolutions) preserves the subsolution (supersolution) property.  Since this is useful for reasons other than the Perron method, we state it as its own result.  We will use the upper and lower semicontinuous envelopes of a function $u$, defined as:
\begin{definition}[Semicontinuous envelopes]\label{defViscSol:USCLSCEnvelopes}
	For a function, $u$, bounded from above, we define the upper semicontinuous envelope of $u$ as
	\begin{align*}
		u^*(x,t) = 
		\lim_{r\to 0} \left(
		\sup_{(y,s)\in B_r(x)\times (t-r,t+r)} u(y,s)
		\right),
	\end{align*}
	and for a function, $v$, bounded from below, the lower semicontinuous envelope as
	\begin{align*}
		v_*(x,t) = 
		\lim_{r\to 0} \left(
		\inf_{(y,s)\in B_r(x)\times (t-r,t+r)} v(y,s)
		\right).
	\end{align*}
\end{definition}

\begin{proposition}[General sup/inf of family of sub/super solutions]\label{propViscSol:inf-supersln}  
	Let  $\{u_i\}_{i\in \I}$ and $\{v_i\}_{i\in \I}$ be collections of functions with
\begin{align*}
    u(x,t) = \sup_{i\in \I}  u_i(x,t) < \infty,\
	\ \ \ \textnormal{and}\ \ \    
	v(x,t) = \inf_{i\in \I} v_i(x,t) > -\infty.
\end{align*}
	If $\{u_i\}_{i\in \I}$,  $\{v_i\}_{i\in \I}$ are families of viscosity subsolutions and supersolutions, respectively, to \eqref{eqViscSol:HeleShawHJB}, 
then  $u^*$ and $v_*$ (as in Definition \ref{defViscSol:USCLSCEnvelopes}) are respectively a viscosity subsolution and supersolution to \eqref{eqViscSol:HeleShawHJB}.
\end{proposition}

\begin{proof}[Proof of Proposition \ref{propViscSol:inf-supersln}]  
	
	We will prove that $u^*$ is a viscosity subsolution.  The proof that $v_*$ is a supersolution follows analogously.

	 Let $\psi$ be a test function for $u^*$ at $(x,t)$ (from Definition \ref{defViscSol:test-functions}).  We will show that $\partial_t \psi(x,t) \leq H(\psi(\cdot, t), x)$.  A key step in most versions of this result is to work with test functions that have a strict max with $u^*$.  The reason for this is that the strict max property for $u^*$ can then be transferred back to the actual subsolutions, $u_i$, but at a nearby point to $(x,t)$.

To this end, we will create a strict max for $u^*-\psi$.  For $R> 0$, define $\tilde \psi$ as
\begin{align*}
    \tilde{\psi}(y,s) = \psi(y,s) + \Phi_R(y-x) + |t-s|^2.
\end{align*}
Then $u^* - \tilde{\psi}$ has a strict maximum over $\overline{B}_r(x)\times [t-r,t+r]$ at $(x,t)$.  By construction of $\tilde \psi$ and $u^*$, there exists sequences, $u_{i_n}$, $(y_n, s_n)$, $c_n$ so that
\begin{align*}
	&u_{i_n} - (\tilde \psi + c_n)\ \ \textnormal{has a local max at}\ \ (y_n,s_n),\\
	&\textnormal{and}\ \ \lim_{n\rightarrow\infty} u_{i_n}(y_n,s_n) = u^*(x,t),\\
	&\textnormal{with}\ \ \lim_{n\rightarrow\infty} (y_n,s_n) = (x,t)\ \ \textnormal{and}\ \ 
	\lim_n c_{n\rightarrow\infty} = 0.
\end{align*}
(The proof of the previous fact is standard, and we omit it.)

The definition of subsolution can then be applied to $u_{i_n}$ to obtain
\begin{align*}
    \partial_t\psi(y_n, s_n)+ 2(s_n-t) 
    = \partial_t \tilde\psi(y_n, s_n) 
    \leq H\left(\tilde\psi(\cdot, s_n), y_n\right) = H\left(\psi(\cdot, s_n) + \phi_R(\cdot -x), y_n\right).
\end{align*}
Here we used the translation invariance and invariance by addition of constants properties of $H$ from Proposition \ref{propAuxOp:HProperties}. Letting $R\to \infty$ and using Corollary \ref{corAuxOp:SplittingTheorem-OscillationVersion}, we obtain 
\begin{align*}
    \partial_t\psi(y_n, s_n)+ 2(s_n-t)  \leq H\left(\psi(\cdot, s_n), y_n\right).
\end{align*}
Taking the limit as $(y_n,s_n)\to (x,t)$ we deduce from Proposition \ref{propViscSol:ContinuityOfH} that 
\begin{align*}
    \partial_t \psi(x,t) \leq H(\psi(\cdot, t), x),  
\end{align*}
and therefore $u^*$ is a viscosity subsolution to \eqref{eqViscSol:HeleShawHJB}.
\end{proof}

As mentioned above, the previous result is one of the steps in building solutions via the Perron method.  Here we give this as our existence result.

\begin{proposition}[Existence via Perron method]\label{propViscSol:Existence}
	
    If $f_0\in BUC(\real^d)$, then there exists a unique viscosity solution $f$ that solves \eqref{eqViscSol:HeleShawHJB} with $f(\cdot, 0) = f_0$ on $\real^d$.
	
\end{proposition}

\begin{proof} We divide the proof into several steps for clarity. \medskip

\paragraph{\textbf{Step 1.}} Building barriers and constructing a maximal subsolution.

Since $f_0$ is uniformly continuous on $\mathbb{R}^d$, for any $\varepsilon>0$ there exists $\rho(\varepsilon)$ such that if $|x-y| < \rho(\varepsilon)$, then $|f_0(x) - f_0(y)| < \varepsilon$. Let $\eta\in C_c^\infty(\mathbb{R}^d)$ be such that $\mathrm{supp}(\eta)\subset B_1(0)$, $\eta(0) = 1$ and $0\leq \eta\leq 1$. For $x_0\in \mathbb{R}^d$ and $\rho > 0$ we define 
\begin{align*}
    \eta_{x_0,\rho}(x) = \eta\left(\frac{x-x_0}{\rho}\right).    
\end{align*}
Let $\rho = \rho(\varepsilon)$, we define the upper and lower functions for $f_0$ at a point $x_0\in \mathbb{R}^d$ by
\begin{align*}
    U_{x_0,\varepsilon}(x) &= \eta_\rho(x)f_0(x_0) + \left(1-\eta_\rho(x)\right) \sup_{\mathbb{R}^d}(f_0) + \varepsilon, \\
    L_{x_0,\varepsilon}(x) &= \eta_\rho(x)f_0(x_0) + \left(1-\eta_\rho(x)\right) \inf_{\mathbb{R}^d}(f_0) - \varepsilon.
\end{align*}
It is evident that $L_{x_0,\varepsilon} \leq f_0 \leq U_{x_0,\ep}$ on $\mathbb{R}^d$. Since $L_{x_0,\varepsilon}$ and $U_{x_0,\ep}$ are smooth, by Corollary \ref{corAuxOp:WfLipOnK*Part2}, there exists $C_\varepsilon$ such that 
\begin{align*}
     \norm{ H(L_{x_0,\varepsilon}, \cdot)}_{L^\infty(\mathbb{R}^d)} \leq C_\varepsilon
	 \ \ \ \ \textnormal{and}\ \ \ \ 
	 \norm{ H(U_{x_0,\varepsilon}, \cdot)}_{L^\infty(\mathbb{R}^d)} \leq C_\varepsilon.
\end{align*}
Let us define 
\begin{align*}
	\psi_{x_0,\varepsilon}^+(x,t) = U_{x_0,\varepsilon}(x) + C_\varepsilon t
	\ \ \ \ \textnormal{and}\ \ \ \
    \psi_{x_0,\varepsilon}^-(x,t) = L_{x_0,\varepsilon}(x) - C_\varepsilon t. 
\end{align*}
By choice of $C_\ep$, $ \psi_{x_0,\varepsilon}^-$ and $\psi_{x_0,\ep}^+$ are respectively a classical subsolution and supersolution to \eqref{eqViscSol:HeleShawHJB}.

In order to invoke the comparison result, Proposition \ref{propViscSol:Comparison-HeleShawHJB}, we need to work with sub/super solutions that respect the local uniform initial ordering in (\ref{eqViscSol:InitialDataLocalUniformOrdering}).  One way to enforce this in a straightforward way is to work only with subsolutions that have the following property:
\begin{align}\label{eqViscSol:Perron-SubSolInitialDataModulus}
	\textnormal{given}\ u,\ \textnormal{there exists a modulus,}\ \om_u,\ \textnormal{with}\
	u(y,t)\leq f_0(x) + \om_u( \abs{y-x} + t ).
\end{align}

We define
\begin{align*}
    \mathcal{S} = 
    \left\{ 
    u\ \text{is a subsolution to}\ (\ref{eqViscSol:HeleShawHJB})\ 
    \textnormal{and also satisfies}\ (\ref{eqViscSol:Perron-SubSolInitialDataModulus}) 
    \right\}.
\end{align*}
The set $\mathcal{S}$ is nonempty since $\psi^-_{x_0,\varepsilon}\in \mathcal{S}$ for any $x_0\in \mathbb{R}^d$ and $\varepsilon>0$.  We also note that constants are solutions to (\ref{eqViscSol:HeleShawHJB}), and so by Proposition \ref{propViscSol:Comparison-HeleShawHJB}, for  $C=\sup_{\real^d} f_0$ and $u\in \mathcal{S}$, we have $u\leq C$.  Thus the set is uniformly bounded. Let 
\begin{align*}
    w(x,t) =  \sup \left\lbrace u(x,t): u\in \mathcal{S}\right\rbrace. 
\end{align*}
Defining $w^*$ as the upper semicontinuous envelope of $w$, Proposition \ref{propViscSol:inf-supersln} gives that  $w^*$ is a viscosity subsolution to \eqref{eqViscSol:HeleShawHJB}.

\paragraph{\textbf{Step 2.}}  The initial condition for $w^*$: we will show $w^*$ satisfies (\ref{eqViscSol:Perron-SubSolInitialDataModulus}), in order to deduce $w^*\in \S$.

Let $u\in \S$. Then $u$ satisfies  (\ref{eqViscSol:Perron-SubSolInitialDataModulus}). And, by construction, $\psi^+_{x, \ep}$ also satisfies (\ref{eqViscSol:Perron-SubSolInitialDataModulus}),  for any $\ep>0$ and $x\in \real^d$. In particular, the pair $u, \psi^+_{x, \ep}$ satisfies the condition  (\ref{eqViscSol:InitialDataLocalUniformOrdering}), and thus by the comparison result, Proposition \ref{propViscSol:Comparison-HeleShawHJB}, we have
\begin{align*}
	 u(y,t)\leq \psi^+_{x,\ep}(y,t) \text{ for all }(y,t)\in\real^d\times [0,T]. 
\end{align*}
The previous line, the definition of $w$, and the fact that $\psi^+_{x_0, \ep}$ is continuous therefore imply,
\begin{align*}
	w^*(y,t)\leq \psi^+_{x,\ep}(y,t) \text{ for all }(y,t)\in\real^d\times [0,T]. 
\end{align*}

We will now establish that $w^*$ satisfies 
\begin{align}\label{eqViscSol:PerronW*InitialInequality}
	\forall\ \al>0,\ \exists\ \del>0,\ \forall\ \abs{x-y}<\del,\ 0<t<\del,\ \ 
	w^*(y,t)\leq f_0(x) + \al,
\end{align}
which will imply that $w^*$ satisfies (\ref{eqViscSol:Perron-SubSolInitialDataModulus}).

To this end, let $\al>0$ be given. Choose $\ep=\frac{\al}{3}$.  We know that $\psi^+_{x,\ep}$ is uniformly continuous in $y$, uniformly in $x$ and $t$.  Choose $\del$ by the uniform continuity of $\psi^+_{x,\ep}$ so that
\begin{align*}
	\sup_{x,t} \psi^+_{x,\ep}(y,t) - \psi^+_{x,\ep}(x,t) < \frac{\al}{3}.
\end{align*}
Also, choose $\del$, possibly smaller so that for $0<t<\del$,
\begin{align*}
	C_\ep t <\frac{\al}{3}.
\end{align*}
Thus,
\begin{align*}
	w^*(y,t) &\leq \psi^+_{x,\ep}(y,t)
	= \psi^+_{x,\ep}(x,t) - \psi^+_{x,\ep}(x,t) + \psi^+_{x,\ep}(y,t)\\
	&=  f_0(x) + C_\ep t + \ep - \psi^+_{x,\ep}(x,t) + \psi^+_{x,\ep}(y,t)
	\leq f_0(x) + \frac{\al}{3} + \frac{\al}{3} + \frac{\al}{3}.
\end{align*}
We then conclude that (\ref{eqViscSol:PerronW*InitialInequality}) is valid, and thus  $w^*\in \S$.

\paragraph{\textbf{Step 3.}} 

 We show that $w_*$ is a supersolution.

Assume to the contrary that $w_*$ is not a supersolution. Thus, by Definitions \ref{defViscSol:test-functions} and \ref{defViscSol:ViscositySolutions}, there exists a test function  $\psi$, a point of contact $(x_0,t_0)\in \real^d\times (0,T]$, $r_1>0$, and $c_0>0$ such that
\begin{align}
\label{eq:psi<w}
&\psi(x,t)\leq w_*(x, t) \text{ for all }(x,t)\in \real^d\times (t_0-r_1,t_0+r_1), \text{ with equality at }(x_0,t_0);\\
\label{eqViscSol:perron-assume-contrary-supersol}
&    \partial _t \psi(x_0,t_0) < H(\psi(\cdot, t_0), x_0) - c_0.
\end{align}

We will now construct a  perturbation $\phi$ of $\psi$ and use it to create a subsolution $\tilde{u}$. We will show  $\tilde{u}\in\S$ and that $\tilde{u}$ can be constructed so that it is strictly larger than $w$ at points near $(x_0,t_0)$, which will be a contradiction.  To this end, we let $\Phi=\Phi_1$ be the function in Remark \ref{remAuxOp:bump-function}, take parameters $\del>0$ and $\alpha>0$, to be determined, and define
\begin{align*}
    \phi(x,t) = \psi(x,t) +\del - \alpha \left( \Phi_1(x-x_0) + |t-t_0|^2 \right).
\end{align*}

Using the invariance of $H$ by the addition of constants, followed by  Theorem \ref{thmAuxOp:HPreciseLipStatementC1Gam}, we find, for a positive constant $C_1$, independent of $c_0$, $\alpha$, and  $\delta$,
\begin{align}
\nonumber
|H(\psi(\cdot, t_0), x_0)- H(\phi(\cdot, t_0), x_0)| &=  |H(\psi(\cdot, t_0), x_0)-H(\psi(\cdot, t_0)- \alpha  \Phi_1(\cdot-x_0), x_0)|\\
\label{eq:Hpsi-Hphi} & \leq C_1\norm{ \alpha \Phi_1 }_{C^{1,\gam}(\real^d)}\leq \frac{c_0}{8},
\end{align}
where the final inequality follows by choosing 
\begin{align*}
    \alpha = \frac{c_0}{8(C_1+1)\norm{ \Phi_1 }_{C^{1,\gam}(\real^d)}}.    
\end{align*}
Next, the continuity of $\partial_t\phi$ and Proposition \ref{propViscSol:ContinuityOfH} imply that there exists $r_1>0$ such that, on $B_{r_1}(x_0)\times (t_0-r_1, t_0+r_1)$, 
\begin{align}
\label{eq:phit-H cont}
|\partial_t\phi(x_0, t_0)-H(\phi(\cdot, t_0), x_0)-\partial_t\phi(x, t)-H(\phi(\cdot, t), x)| \leq \frac{c_0}{8}.
\end{align}
Let $r=\min\{r_0, r_1, \frac{t_0}{2}\}$. Then $t_0-r>0$, and on $B_{r}(x_0)\times (t_0-r, t_0+r)$, we have,
\begin{align} \nonumber 
\partial_t\phi(x, t)-H(\phi(\cdot, t), x)&\leq \partial_t\phi(x_0, t_0)-H(\phi(\cdot, t_0), x_0)+\frac{c_0}{8}\\
&\label{eq:phi subsol}\leq \partial_t\psi(x,t)-2\alpha (t-t_0) - H(\psi(\cdot, t_0), x_0)+\frac{c_0}{4}\leq -\frac{c_0}{2},
\end{align} 
where the first inequality follows from (\ref{eq:phit-H cont}), the second from (\ref{eq:Hpsi-Hphi}), and the third from (\ref{eqViscSol:perron-assume-contrary-supersol}) as well as by our choice of $\alpha$.

Now, let us fix $\delta = \frac{1}{2}\left(\alpha\frac{r}{1+r} +r\right)$. Using with the  fact that $\psi-\phi$ is radially increasing, we find,
\begin{align*}
%\label{eqViscSol:Perron-PhiBelowPsi}
\inf_{(\real^d\times[0,T])\setminus(B_r(x)\times (t_0-r,t_0+r))} (\psi-\phi)\geq 	\inf_{\partial(B_r(x)\times (t_0-r,t_0+r))} \left(  \psi - \phi \right) =-\delta +\alpha\frac{r}{1+r} +r>\frac{\delta}{2}>0.
\end{align*}
Combining this with the inequality from (\ref{eq:psi<w}) yields,
\begin{align}
\label{eq:w>phi}
w^*(x,t)\geq w(x,t)\geq w_*(x,t)\geq \psi(x,t)> \phi(x,t) \text{ on }(\real^d\times[0,T])\setminus(B_r(x)\times (t_0-r,t_0+r)).
\end{align}
We  now define $\tilde{u}$ via
\begin{align*}
	\tilde u = \max( w^*, \phi).
\end{align*}
We claim that $\tilde{u}$ is a subsolution of \eqref{eqViscSol:HeleShawHJB} on all of $\real^d\times (0,T]$. Indeed, recall that $w^*$ is a subsolution of \eqref{eqViscSol:HeleShawHJB} on $\real^d\times (0,T]$. The inequality 
 (\ref{eq:w>phi}) implies that $\tilde{u}\equiv w^*$ outside of $B_r(x)\times (t_0-r,t_0+r)$; therefore, $\tilde{u}$ is a subsolution outside of $B_r(x)\times (t_0-r,t_0+r)$. On the other hand, we have shown in (\ref{eq:phi subsol}) that $\phi$ is a subsolution of \eqref{eqViscSol:HeleShawHJB} on $B_r(x)\times (t_0-r,t_0+r)$; since a maximum of two subsolutions is a subsolution, we conclude that $\tilde{u}$ is a subsolution on $B_r(x)\times (t_0-r,t_0+r)$ as well.  
 
Finally, we recall that our choice of $r$ implies $t_0-r>0$. Together with (\ref{eq:w>phi}), this yields $\tilde{u}\equiv w^*$ in a neighborhood of $\real^d\times\{0\}$.  Since we have already shown that $w^*$ obeys the local uniform initial condition in (\ref{eqViscSol:Perron-SubSolInitialDataModulus}), we deduce that $\tilde{u}$ does as well, and thus we conclude  $\tilde u\in \S$, as desired,

In order to obtain the desired contradiction, we will find points nearby $(x_0,t_0)$ where $\tilde u > w$, which will contradict the definition of $w$, as $\tilde u\in S$.  Indeed, we observe that by the definition of test function, $\psi(x_0,t_0) = w_*(x_0,t_0)$. 
From the definition of $w_*$, let us take a sequence, $(y_n,s_n)$, so that
\begin{align*}
	\lim_{n\to\infty} w(y_n,s_n) = w_*(x_0,t_0).
\end{align*}
Thus, taking $n$ large enough, 
\begin{align*}
	w(y_n,s_n) \leq w_*(x_0,t_0) + \frac{\del}{4}.
\end{align*}
By continuity of $\phi$, and the definition of $\tilde u$, we see that for $n$ large enough, 
\begin{align*}
	\tilde u(y_n,s_n) \geq \phi(y_n,s_n) \geq \phi(x_0,t_0) - \frac{\del}{4} 
	= w_*(x_0,t_0) + \frac{3\del}{4} > w(y_n,s_n).
\end{align*}
Thus, we have contradicted the maximality of $w$.  This concludes that $w_*$ is a supersolution of (\ref{eqViscSol:HeleShawHJB}).

\paragraph{\textbf{Step 4.}} Conclusion: $w_* = w^* = w$ is a solution of (\ref{eqViscSol:HeleShawHJB}) with $w(\cdot,0)=f_0$.

We have already shown that $w^*$ is a subsolution that also enjoys (\ref{eqViscSol:Perron-SubSolInitialDataModulus}).  Since for all $x_0$ and $\ep>0$, we know that $\psi^-_{x_0,\ep}\in\S$, we conclude,  just as was done in Step 2, that there is some modulus with 
\begin{align*}
	w_*(y,t)\geq f_0(x) - \om_{w_*}(\abs{x-y}+t).
\end{align*}
Thus, we have that $w^*$ is a subsolution, $w_*$ is a supersolution, and they obey
\begin{align*}
	w^*(x,s)\leq w_*(y,t) + \om_{w^*}(\abs{x-y}+s) + \om_{w_*}(\abs{x-y}+t).
\end{align*}
Thus, $w_*$ and $w^*$ have the required local uniform initial condition to apply the comparison result, Proposition \ref{propViscSol:Comparison-HeleShawHJB}.  We then conclude that $w^*\leq w_*$,
which was the goal.
\end{proof}

%%%%%%%%%%%%%%%%%%%%%%%%%%%%%%%%%%%%%%%%%%%%%%%%%
%%%%%%%%%%%%%%%%%%%%%%%%%%%%%%%%%%%%%%%%%%%%%%%%%
%%%%%%%%%%%%%%%%%%%%%%%%%%%%%%%%%%%%%%%%%%%%%%%%%
%%%%%%%%%%%%%%%%%%%%%%%%%%%%%%%%%%%%%%%%%%%%%%%%%

\subsection{Preservation of modulus} An immediate consequence of the comparison principle and the translation invariance of $H$ is the preservation of a modulus of continuity for solutions. This is a classical result for viscosity solutions when the equation is translation invariant.

\begin{proposition}[Modulus]\label{propViscSol:preservation-modulus} 
	
	Let $f_0$ be bounded and uniformly continuous on $\mathbb{R}^d$ with a modulus $\omega$. Let $f\in C^0(\real^d\times[0,T))$ be the unique viscosity solution of \eqref{eqViscSol:HeleShawHJB} with initial condition $f(\cdot, 0)=f_0$. Then  $f(\cdot, t)$ has the same modulus $\omega$ for $t\in [0,T)$.
\end{proposition}

\begin{proof}[Proof of Proposition \ref{propViscSol:preservation-modulus}]

	Let $h\in \mathbb{R}^d$ be fixed. Due to the translation invariance of $H$, we have, 
\begin{align*}
    \partial_ t f(x+h, t) = H(f(\cdot, t), x+h) =  H(\tau_hf(\cdot, t), x);
\end{align*}
so, we obtain that $(x,t)\mapsto \tau_hf(x,t)$ is a viscosity solution of \eqref{eqViscSol:HeleShawHJB} with initial data 
\begin{align*}
    \tau_hf(x,0) = f_0(x+h) \leq f_0(x) + \omega(|h|).    
\end{align*}
Let $\phi(x,t) = f(x,t) + \omega(|h|)$.  Due to the invariance by the addition of constants we see that $\phi$ is a viscosity solution to \eqref{eqViscSol:HeleShawHJB} with initial data 
\begin{align*}
    \phi(x,0) = f_0(x) + \omega(|h|) \geq f_0(x+h) = \tau_hf(x,0).
\end{align*}
By the construction in Proposition \ref{propViscSol:Existence}, $f$ satisfies the condition (\ref{eqViscSol:InitialDataLocalUniformOrdering}), and hence also does $\phi$.
Thus, by the comparison result (Proposition \ref{propViscSol:Comparison-HeleShawHJB}) we obtain 
\begin{align*}
    \tau_hf(x,t) = f(x+h, t) \leq \phi(x,t) = f(x,t) + \omega(|h|)
\end{align*}
for all $(x,t)\in \mathbb{R}^d\times [0,T)$. We can repeat the argument with $\psi(x,t) = -\omega(|h|) + f(x,t)$ to deduce that $\abs{ f(x+h,t) - f(x,t) } \leq \omega(|h|)$. 
\end{proof}

%%%%%%%%%%%%%%%%%%%%%%%%%%%%%%%%%%%%%%%%%%%%%%%%%
%%%%%%%%%%%%%%%%%%%%%%%%%%%%%%%%%%%%%%%%%%%%%%%%%
%%%%%%%%%%%%%%%%%%%%%%%%%%%%%%%%%%%%%%%%%%%%%%%%%
%%%%%%%%%%%%%%%%%%%%%%%%%%%%%%%%%%%%%%%%%%%%%%%%%

\subsection{Concluding existence and uniqueness for (\ref{eqViscSol:HeleShawHJB})}

It will be useful to collect each of the main results in this section into one theorem that concisely states the existence and uniqueness for (\ref{eqViscSol:HeleShawHJB}).  Combining Propositions \ref{propViscSol:Comparison-HeleShawHJB} (comparison), \ref{propViscSol:Existence} (existence), and \ref{propViscSol:preservation-modulus} (modulus), we obtain the following.

\begin{theorem}\label{thmViscSol:HeleShawCollectedResults}
	Given any $f_0\in BUC(\real^d)$, there exists a unique viscosity solution, $f\in C^0(\real^d\times [0,T))$, to 
	\begin{align}\label{eqViscSol:HeleShawHJB-WithInitialCondition}
		\begin{cases}
			\partial_t f = H(f)\ &\text{in}\ \real^d\times(0,T)\\
			f(\cdot,0) = f_0\ &\text{on}\ \real^d\times\{0\}.
		\end{cases}
	\end{align}
	Furthermore, if $f_0$ has a modulus of continuity, $\om$, then for each $t\in[0,T)$, $f(\cdot, t)$ has the same modulus of continuity.
\end{theorem}

%%%%%%%%%%%%%%%%%%%%%%%%%%%%%%%%%%%%%%%%%%%%%%%%%
%%%%%%%%%%%%%%%%%%%%%%%%%%%%%%%%%%%%%%%%%%%%%%%%%
%%%%%%%%%%%%%%%%%%%%%%%%%%%%%%%%%%%%%%%%%%%%%%%%%
%%%%%%%%%%%%%%%%%%%%%%%%%%%%%%%%%%%%%%%%%%%%%%%%%

\subsection{Proof of Theorem \ref{thmIntro:MuskatExistUnique}}\label{subsecVS:ProofOfMainTheorem}

As suggested in Section \ref{sec:MainObservation}, we will establish Theorem \ref{thmIntro:MuskatExistUnique} by showing that there is an equivalence between viscosity solutions of (\ref{eqIntro:MuskatHJB}) and (\ref{eqViscSol:HeleShawHJB-WithInitialCondition}).  We use Definition \ref{defViscSol:ViscositySolutions} with $J=M$ for the definition of solution to (\ref{eqIntro:MuskatHJB}). We recall that $M$ is defined in (\ref{eqIntro:M def}).

Towards Theorem \ref{thmIntro:MuskatExistUnique}, we will establish the following.

\begin{proposition}\label{propViscSol:EquivalenceSolsHeleShawMuskat}
	If $f_0\in BUC(\real^d)$, $f\in C^0(\real^d\times[0,T))$, and $g(x,t)=f(x,t)+t$, then $f$ is a viscosity solution to (\ref{eqIntro:MuskatHJB}) if and only if $g$ is a viscosity solution of (\ref{eqViscSol:HeleShawHJB-WithInitialCondition}).
\end{proposition}

Proposition \ref{propViscSol:EquivalenceSolsHeleShawMuskat} will be immediate from the definitions of viscosity solutions, combined with the following observation.

\begin{lemma}\label{lemViscSol:MuskatAndHeleShawOperatorsEquivalent}
	If $f\in C^{0,1}(\real^d)$, $W_f$ is from Definition \ref{defAuxOp:DefOfWf}, $\Phi_{f,f}$ is as in (\ref{eqIntro:DtoNPotential}), and $\ell(x) = -x_{d+1}$, then 
	\begin{enumerate}[(i)]
		\item $\displaystyle \Phi_{f,f} = W_f - \ell$,
		
		\item for $X_0\in\Gam$, 
		\begin{align*}
			\partial_\nu W_f(X_0)\ \textnormal{exists}\ \ \iff\ \ 
			\partial_\nu \Phi_{f,f}(X_0)\ \textnormal{exists},
		\end{align*} 
		
		\item for all $x_0$ such that $H(f,x_0)$ exists,
		\begin{align*}
			M(f,x_0) = H(f,x_0) - 1.
		\end{align*}
	\end{enumerate}
	
\end{lemma}

\begin{proof}[Proof of Lemma \ref{lemViscSol:MuskatAndHeleShawOperatorsEquivalent}]
	We note that for $f\in C^{0,1}(\real^d)$, the existence and uniqueness of $W_f$ and $\Phi_{f,f}$ is given in the Appendix, Propositions \ref{propAppendix:weak-max-principle} and \ref{propAppendix:existence-Perron}.  Since $\ell$ is $C^2$, parts (i), (ii), (iii) are immediate from the presentation given in Section \ref{sec:MainObservation}.
\end{proof}

Now we can give the proof of Proposition \ref{propViscSol:EquivalenceSolsHeleShawMuskat}.

\begin{proof}[Proof of Proposition \ref{propViscSol:EquivalenceSolsHeleShawMuskat}]
	Given the definitions of viscosity sub and super solutions, it suffices to establish the equivalence for test functions. This is because $f$ and $g$ differ by the addition of a smooth function.  So, for any test function, $\phi$, as in Definition \ref{defViscSol:test-functions}, we see that $\phi$ is a test function for $f$ if and only if $\phi + t$ is test function for $g$.  Thus, for example, the subsolution property for $f$ implies $\partial_t \phi \leq M(\phi)$ and thus
	\begin{align*}
		\partial_t (\phi +t)\leq M(\phi+t) +1 = H(\phi + t). 
	\end{align*}
	The reverse implication is immediate.  The property for supersolutions is also immediate.
\end{proof}

We can now give the proof of Theorem \ref{thmIntro:MuskatExistUnique}.

\begin{proof}[Proof of Theorem \ref{thmIntro:MuskatExistUnique}]
	This is immediate from Theorem \ref{thmViscSol:HeleShawCollectedResults} and Proposition \ref{propViscSol:EquivalenceSolsHeleShawMuskat}.
\end{proof}

%%%%%%%%%%%%%%%%%%%%%%%%%%%%%%%%%%%%%%%%%%%%%%%%%
%%%%%%%%%%%%%%%%%%%%%%%%%%%%%%%%%%%%%%%%%%%%%%%%%
%%%%%%%%%%%%%%%%%%%%%%%%%%%%%%%%%%%%%%%%%%%%%%%%%
%%%%%%%%%%%%%%%%%%%%%%%%%%%%%%%%%%%%%%%%%%%%%%%%%
%%%%%%%%%%%%%%%%%%%%%%%%%%%%%%%%%%%%%%%%%%%%%%%%%
%%%%%%%%%%%%%%%%%%%%%%%%%%%%%%%%%%%%%%%%%%%%%%%%%
%%%%%%%%%%%%%%%%%%%%%%%%%%%%%%%%%%%%%%%%%%%%%%%%%
%%%%%%%%%%%%%%%%%%%%%%%%%%%%%%%%%%%%%%%%%%%%%%%%%
%%%%%%%%%%%%%%%%%%%%%%%%%%%%%%%%%%%%%%%%%%%%%%%%%

\section{ Further discussion of related results }\label{sec:ComparisonWithOtherResults}

In the first subsection below, we demonstrate that the potential $\Phi_{f,g}$, which is central to our work here, is the same one as appears in several other works on the Muskat problem. In subsection \ref{ss:NotionsOfSolutions}, we discuss the various notions of solutions and how they are related.
In the final two subsections, we provide a more detailed account of where two of the main elements of our proof --- the GCP and the punctual evaluation property --- appear previously in the literature.

\subsection{The potential $\Phi_{f,g}$ in the literature}
\label{ss:potential}

As described in Section \ref{ssIntro:Muskat}, the works \cite{AlazardMeunierSmets-2020LyapounovCauchyProbHeleShaw-CMP},  \cite{DongGancedoNguyen-2023MuskatWP3D-arXiv}, \cite{DongGancedoNguyen-2023MuskatOnePhase2DWellPose-CPAM} establish well-posedness results for integro-differential equations that arise from the reduction of 
 the one-phase Muskat problem  to the equation for the graph of the boundary.
 In this section, we demonstrate that the operators defined in those works are the same as the operator $G(f)g$ used here. Indeed, in the present work, the operator $G(f)g$ is defined via the potential $\Phi_{f,g}$, the unique solution to (\ref{eqIntro:DtoNPotential}). The integro-differential operators studied in  \cite{AlazardMeunierSmets-2020LyapounovCauchyProbHeleShaw-CMP},  \cite{DongGancedoNguyen-2023MuskatWP3D-arXiv}, \cite{DongGancedoNguyen-2023MuskatOnePhase2DWellPose-CPAM} are defined via the potential $\Psi_{f,g}$, the unique variational solution that satisfies
\begin{align}\label{eqOtherResults:PotentialDGN}
	\Delta \Psi_{f,g} = 0\ \text{in}\ D_f,\ \ \ \Psi_{f,g}=g \text{ on }\Gamma_f,\ \ \  
	\Psi_{f,g}\in L^1_{\textnormal{loc}}(D_f),\ \ \ \textnormal{and}\ \ \ \norm{\grad \Psi_{f,g}}_{L^2(D_f)}<\infty, 
\end{align}
(see \cite[Proposition 2.6]{DongGancedoNguyen-2023MuskatOnePhase2DWellPose-CPAM}). We now demonstrate that, for periodic $f$, $g$ (the setting of \cite{AlazardMeunierSmets-2020LyapounovCauchyProbHeleShaw-CMP},  \cite{DongGancedoNguyen-2023MuskatWP3D-arXiv}, \cite{DongGancedoNguyen-2023MuskatOnePhase2DWellPose-CPAM}), the potentials $\Phi_{f,g}$ and $\Psi_{f,g}$ agree.

\begin{proposition}\label{propOtherResults:SamePotentials}
	If $f\in C^{0,1}(\real^d)$ and $g\in C^0(\Gam_f)$ are both periodic, $\Phi_{f,g}$ is defined as in (\ref{eqIntro:DtoNPotential}), and $\Psi_{f,g}$ is the corresponding potential defined in (\ref{eqOtherResults:PotentialDGN}), then
	\begin{align*}
		\Phi_{f,g} = \Psi_{f,g}.
	\end{align*}
\end{proposition}

We have the following immediate corollary:

\begin{corollary}
	The normalized versions of the operators in \cite{AlazardMeunierSmets-2020LyapounovCauchyProbHeleShaw-CMP}, \cite{DongGancedoNguyen-2023MuskatOnePhase2DWellPose-CPAM}, \cite{DongGancedoNguyen-2023MuskatWP3D-arXiv} all coincide with $G(f)g$ defined here.
\end{corollary}

For the convenience of the reader, we now provide the proof of Proposition \ref{propOtherResults:SamePotentials}, using well-known, classical techniques. First we recall the following fact and its proof:

\begin{lemma}\label{lemma:SublinearGrowthHarmonic}    
        Let $\mathbb{H}=\{(x,x_{d+1}) \in \mathbb{T}^d\times \mathbb{R}\ :\ x_{d+1}<0\}$ and $u\in C^2(\mathbb{T}^{d}\times \mathbb{R})$ such that 
    $\nabla u\in L^2(\mathbb{H})$ where $X = (x,x_{d+1})\in \mathbb{T}^d\times \mathbb{R}$, then
    	\begin{align}\label{eq:lemma:SublinearGrowthHarmonic:def-sublinear}
		\lim_{|y|\to \infty}\left( \sup_{x\in \mathbb{T}^d} \frac{|u(x,y)|}{|y|} \right) = 0.
	\end{align}
\end{lemma}
\begin{proof} We have, for any $x\in \real^d$,
\begin{align*}
	|u(x,y)| 
	&\leq |u(x,0)| + \int_{y}^0 |u_{x_{d+1}}(x,\xi)|\ d\xi  \leq C + \left(\int_{y}^0 |u_{x_{d+1}}(x,\xi)|^2\ d\xi\right)^{1/2} \left(\int_y^0 d\xi\right)^{1/2} \\ &= C+ \sqrt{|y|}\left(\int_{y}^0 |u_{x_{d+1}}(x,\xi)|^2\ d\xi\right)^{1/2} \leq C+\sqrt{|y|} \|\nabla u\|_{L^2(\mathbb{H})},
\end{align*}
from which the conclusion (\ref{eq:lemma:SublinearGrowthHarmonic:def-sublinear}) follows.
\end{proof}

\begin{proof}[Proof of Proposition \ref{propOtherResults:SamePotentials}] By Lemma \ref{lemma:SublinearGrowthHarmonic}, we observe that $\Psi_{f,g}$ exhibits sublinear behavior, as specified in \eqref{eq:lemma:SublinearGrowthHarmonic:def-sublinear}. To streamline our discussion, we assume, without loss of generality, that $f\geq 1$. This assumption implies $\mathbb{H} = \{(x,x_{d+1}): x\in \mathbb{T}^d, x_{d+1} < 0\}\subset D_f$. Consequently, both $\Phi_{f,g}$ and $\Psi_{f,g}$ are elements of $C^2(\overline{\mathbb{H}})$. Let us introduce a new function, denoted as $v$, which satisfies the equation:
\begin{align*}
	\Delta v = 0\;\textnormal{in}\;\mathbb{H}, \qquad v = \Psi_{f,g}\;\textnormal{on}\;\partial\mathbb{H}, \qquad \textnormal{and}\qquad  \Vert v\Vert_{L^\infty(\mathbb{H})} < \infty.
\end{align*}
The existence and uniqueness of such a function $v$ follows from Propositions \ref{propAppendix:weak-max-principle} and \ref{propAppendix:existence-Perron}, or by just using the convolution with the Poisson kernel for half-space. Define 
\begin{align*}
    w = \Psi_{f,g} - v\ \ \text{in}\;\overline{\mathbb{H}},
\end{align*}
and note $w = 0$ on $\partial \mathbb{H}$. We can then use reflection to extend $w$ to a harmonic function $\tilde{w}$ on $\mathbb{T}^d\times \mathbb{R}$, where $\tilde{w} = w$ in $\mathbb{H}$. Notably, $\tilde{w}$ is sublinear due to the boundedness of $v$ and Lemma \ref{lemma:SublinearGrowthHarmonic} applied to $\Psi_{f,g}$. The classical gradient estimate for harmonic functions yields:
\begin{align*}
	|\nabla \tilde{w}(X)| \leq \frac{C}{R} \Vert \tilde{w}\Vert_{L^\infty(B_R(X))}
\end{align*}
for any $X \in \mathbb{T}^d\times \mathbb{R}$ and $R>0$. Let $R\to \infty$ we conclude that $\tilde w$ must be a constant due to its sublinear growth at infinity. Consequently, $\tilde{w} = 0$. This leads to the deduction that $\Psi_{f,g} = w + v$ is bounded, establishing $\Psi_{f,g}\equiv \Phi_{f,g}$ by the uniqueness property in Proposition \ref{propAppendix:weak-max-principle} for bounded solutions.
\end{proof}

Finally, we remark on  several instances of the use of $G(f)g$ in the literature on water waves. In the work \cite{Lannes-2005WellPoseWaterWave-JAMS}, the corresponding potential $\Phi_{f,g}\in H^{k+2}$ is defined as a variational solution for $f\in H^{k+3/2}$, $g\in H^{k+1/2}$.  In the work \cite{Wu-1999WellPosedWaterWaveInSobolevSpace-JAMS}, the potential $\Phi_{f,g}$, defined for $f\in C^{0,1}(\real^d)$ and $g\in L^2(\Gam_f)$, is the unique smooth solution whose boundary values are determined by the non-tangential maximal function, given by \cite{Dahlberg-1977EstimatesHarmonicMeasARMA}, \cite{JerisonKenig-1982BoundaryValueProblemsLipDomains-Chapter}.

\subsection{Various notions of solutions among \cite{AlazardKoch-2023HeleShawSemiFlow-ArXiV}, \cite{AlazardMeunierSmets-2020LyapounovCauchyProbHeleShaw-CMP}, \cite{DongGancedoNguyen-2023MuskatOnePhase2DWellPose-CPAM}}
\label{ss:NotionsOfSolutions}

We now expand upon the discussion 
initiated in Remark \ref{remark} on the two notions of weak solutions available for the one-phase Muskat problem: viscosity solutions and variational solutions.
There is a more complete list of references on this topic in Sections \ref{sss:HS} and \ref{ssIntro:Muskat}.

We note that, in this work and  \cite{AlazardKoch-2023HeleShawSemiFlow-ArXiV}, it is used that, in the graph setting, classical solutions to the Muskat and Hele-Shaw problems, (\ref{eqIntro:MuskatInTermsOfOmegatNormalVelocity}) and (\ref{eqIntro:HS1}),  differ by a linear function. In addition, each of the two problems has an associated Hamilton-Jacobi-Bellman equation: (\ref{eqIntro:MuskatHJB}) and (\ref{eqIntro:HeleShawHJB}), respectively. Although viscosity solutions are defined  for (\ref{eqIntro:HS1}) and (\ref{eqIntro:HeleShawHJB}), it is 
 crucial to note that these are two different notions of solution.  Indeed, on the one hand, (\ref{eqIntro:HS1}) is a free boundary problem whose spatial variable is in $\real^{d+1}$, and the notion of viscosity solution is designed to be compatible with the free boundary condition (e.g. that given in \cite{Kim-2003UniquenessAndExistenceHeleShawStefanARMA}).  On the other hand, (\ref{eqIntro:HeleShawHJB}) is a nonlinear integro-differential equation with no free boundary condition (a nonlinear version of the $1/2$-heat equation, as was made precise in \cite{ChangLaraGuillenSchwab-2019SomeFBAsNonlocalParaboic-NonlinAnal}), and so viscosity solutions are based on the existing integro-differential theory (e.g. \cite{BaIm-07}, \cite{CaSi-09RegularityIntegroDiff}, \cite{Silv-2011DifferentiabilityCriticalHJ}, \cite{ChangLaraGuillenSchwab-2019SomeFBAsNonlocalParaboic-NonlinAnal}).  An equivalence between the different notions of viscosity solution to (\ref{eqIntro:HS1}) and (\ref{eqIntro:HeleShawHJB}) was given in \cite{ChangLaraGuillenSchwab-2019SomeFBAsNonlocalParaboic-NonlinAnal}.  Although this relationship between different notions of viscosity solutions has not been extended to the one-phase Muskat problem, namely \eqref{eqIntro:MuskatInTermsOfOmegatNormalVelocity} and  (\ref{eqIntro:MuskatHJB}), the standard notion of viscosity solutions for integro-differential equations (closest to those in \cite{Silv-2011DifferentiabilityCriticalHJ})   were used for (\ref{eqIntro:MuskatHJB}) in \cite{DongGancedoNguyen-2023MuskatOnePhase2DWellPose-CPAM}, \cite{DongGancedoNguyen-2023MuskatWP3D-arXiv}, as well as in the present work.  We note that one important property of viscosity solutions to either (\ref{eqIntro:HS1}) or (\ref{eqIntro:MuskatHJB})  is that the equation can be evaluated in a pointwise sense whenever the viscosity solution is locally regular enough (see subsection \ref{ss:punctually} below, for the case of (\ref{eqIntro:MuskatHJB})).

A different notion of solution to (\ref{eqIntro:HS1}) (which as above also gives solutions to (\ref{eqIntro:MuskatInTermsOfOmegatNormalVelocity})), going back to \cite{ElliottJanovsky-1981VariationalApproachHeleShaw}, relies upon an auxiliary problem that is variational in nature, utilizing an obstacle problem, which pre-dates the viscosity solutions of, e.g. \cite{Kim-2003UniquenessAndExistenceHeleShawStefanARMA}.  When solutions are regular enough, a solution of (\ref{eqIntro:HS1}) can be used to construct a solution to (\ref{eqIntro:HeleShawHJB}), and hence also (\ref{eqIntro:MuskatHJB}).      This is the variational inequality / obstacle problem approach is used in \cite{AlazardKoch-2023HeleShawSemiFlow-ArXiV}.  A connection between the variational solutions constructed in \cite{ElliottJanovsky-1981VariationalApproachHeleShaw} and viscosity solutions in \cite{Kim-2003UniquenessAndExistenceHeleShawStefanARMA} is utilized for a homogenization problem in \cite{KimMellet-2009HomogHeleShawARMA}.

These variational solutions to the Hele-Shaw problem (\ref{eqIntro:HS1}) are constructed using the transformation, 
\begin{align}
\label{eq:EJ}
w(x,t) = \int_0^t p(x,s)ds,
\end{align}
where $p$ is the pressure of the fluid, as in (\ref{eqIntro:HS1}).  If one had $w$ and it was regular enough, then $p$ can be recovered from $\partial_t w$.    This is the approach in \cite{AlazardKoch-2023HeleShawSemiFlow-ArXiV}. In one direction, \cite{AlazardKoch-2023HeleShawSemiFlow-ArXiV} establishes that if $p$ is regular (classical), then $w$ is a distributional solution of the obstacle problem
\begin{align}\label{eqOtherResults:ObstacleEqW}
	\Delta w = \Indicator_{\{w>0\}\intersect A}\ \ \ \text{where}\ \ \ 
	A = U\setminus \{ p(\cdot,0) >0 \},
\end{align}
where $U$ is a fixed background set that is open and connected (\cite[Proposition 4.3]{AlazardKoch-2023HeleShawSemiFlow-ArXiV}). In fact, given a candidate $p$ to solve, (\ref{eqIntro:HS1}), \cite[Definition 3.6]{AlazardKoch-2023HeleShawSemiFlow-ArXiV} asks that the corresponding $w$ in (\ref{eq:EJ}) should be a distributional solution of (\ref{eqOtherResults:ObstacleEqW}).   Then,  \cite{AlazardKoch-2023HeleShawSemiFlow-ArXiV}   uses variational techniques to solve this  obstacle problem for $w$ and define the solution $f$ to  (\ref{eqIntro:MuskatHJB}) implicitly by
\begin{align}\label{eqOtherResults:DefOfF-AlazardKoch}
	\{ (x,x_{d+1})\ :\ w(x,x_{d+1},t) >0  \} = 
	\{  (x,x_{d+1})\ : x_{d+1} < f(x,t)\}.
\end{align}
This relation can be found in \cite[Theorem 5.4]{AlazardKoch-2023HeleShawSemiFlow-ArXiV}.  This definition of $f$ gives the solutions referred to as a semi-flow for (\ref{eqIntro:MuskatHJB}) in \cite{AlazardKoch-2023HeleShawSemiFlow-ArXiV}.  There it is shown that if $w$, and hence $f$, are globally regular enough (i.e.  $f\in H^s(\real^d)$ with $s>\frac{d}{2}+1$), then $f$ indeed solves (\ref{eqIntro:MuskatHJB}). Furthermore, \cite[Theorem 2.3]{AlazardKoch-2023HeleShawSemiFlow-ArXiV} states that, in the case that the sub-zero sets of the auxiliary function $w$ are not  globally smooth enough to evaluate (\ref{eqIntro:MuskatHJB}) pointwise, there exists a unique semi-flow map from bounded LSC functions to bounded LSC functions, as well as from uniformly continuous data into the space of uniformly continuous functions.  The solution of (\ref{eqIntro:MuskatHJB}) is then interpreted as this semi-flow map acting on the given data.  Furthermore, after some regularizing time, these auxiliary functions, $w(t)$, become classical (their corresponding free boundary is $C^{1,\al}$), and then the correspondence between the three different solutions to three different equations can be used interchangeably to obtain one from the other, i.e. for the Hele-Shaw setting, $w$ produces $p$, which, in turn produces $f$ as the function whose graph gives the free boundary, and $f$ solves (\ref{eqIntro:HeleShawHJB}) (which, as above is equivalent to solving the corresponding equation related the the one-phase Muskat problem (\ref{eqIntro:MuskatHJB})).  For data $f(\cdot,0)\in BUC(\real^d)$ (or even bounded and $LSC(\real^d)$), there will be a time for which $f$ is not a classical solution, and  when the free boundary of $w$ is not classical, and here it is not clear in what sense such an $f$ solves (\ref{eqIntro:HeleShawHJB}) (which, as used above gives a solution to (\ref{eqIntro:MuskatHJB})), and this is not addressed in \cite{AlazardKoch-2023HeleShawSemiFlow-ArXiV}.

It is not known whether the  solutions of (\ref{eqIntro:MuskatHJB}) introduced in \cite{AlazardKoch-2023HeleShawSemiFlow-ArXiV} correspond to the viscosity solutions of (\ref{eqIntro:MuskatHJB}) given here.  Thus, it seems reasonable to pursue the study of both solutions independently.  Given the correspondence proved between variational and viscosity solutions to the Hele-Shaw problem (\ref{eqIntro:HS1}) in \cite{KimMellet-2009HomogHeleShawARMA}, we would guess that there are some situations in which  variational solutions and  viscosity solutions of  (\ref{eqIntro:MuskatHJB}) coincide.

\subsection{The GCP}
\label{ss:GCP}

The  GCP (Definition \ref{defAuxOp:GCP})  is fundamental to all works on viscosity solutions, including the foundational ones of  \cite{CrandallLions-83ViscositySolutionsHJE-TAMS}, \cite{CrandallEvansLions-1984SomePropertiesOfViscositySolutionsTAMS}. In particular, the comparison property is at the core of the definition of viscosity solutions for the Hele-Shaw problem (\ref{eqIntro:HS1}) and similar free boundary problems; see \cite{Caffarelli-1987HarnackInqualityApproachFBPart1RevMatIbero}, \cite{AthanaCaffarelliSalsa-1996RegFBParabolicPhaseTransitionACTA}, \cite{Kim-2003UniquenessAndExistenceHeleShawStefanARMA}.

In the setting of nonlinear integro-differential operators, the GCP was established in \cite{ChangLaraGuillenSchwab-2019SomeFBAsNonlocalParaboic-NonlinAnal} for a class of operators including the Hele-Shaw operator $I$  (\ref{eqIntro:I}). It was established that viscosity solutions to the corresponding Hamilton-Jacobi-Bellman equation exist, are unique, and correspond to those of the associated free boundary problem, which, for $I$, is the Hele-Shaw problem (\ref{eqIntro:HS1}).  A key element of the proof were the min-max formulas for elliptic operators established in \cite{GuSc-2019MinMaxNonlocalCALCVARPDE}, \cite{GuSc-2019MinMaxEuclideanNATMA},  which provided a roadmap for establishing well-posedness.

 Subsequently, the GCP for $M$ was utilized in \cite{DongGancedoNguyen-2023MuskatOnePhase2DWellPose-CPAM}, \cite{DongGancedoNguyen-2023MuskatWP3D-arXiv} to build viscosity solutions for the Muskat equation.  The GCP is implicitly used in \cite[Propositions 2.4 and 2.6]{AlazardMeunierSmets-2020LyapounovCauchyProbHeleShaw-CMP}.  Furthermore, in the somewhat long history of viewing free boundaries a parabolic operators, one can also see the GCP in the works that use the positivity of the Raleigh-Taylor coefficient, which \emph{for the one-phase problem} is the quantity $1-\partial_{e_{d+1}}\Phi_{f,f}$.  Indeed, verifying $1-\partial_{e_{d+1}}\Phi_{f,f}>0$ in fact gives ellipticity or parabolicity of the corresponding equation, as in  \cite[Proposition 4.3 and Section 9.2]{AlazardMeunierSmets-2020LyapounovCauchyProbHeleShaw-CMP}, \cite[Lemma 4.2, Proposition 4.3]{NguyenPausader-2020ParadifferentialWellPoseMuskatARMA}.

\subsection{Pointwise  evaluation property}
\label{ss:punctually}

A very convenient result that is used in the viscosity solutions theory for integro-differential equations is that the equation holds whenever the test function is merely punctually $C^{1,\gam}$; we informally refer to this as the pointwise evaluation property.  This immediately gives that solutions will satisfy the equation in the pointwise sense when they enjoy enough regularity locally, which we state in the following lemma.

\begin{lemma} If $f$ is a viscosity solution of \eqref{eqIntro:MuskatHJB}, and there is $r_0$ for which $f$ has more regularity in the sense that
	\begin{align*}
		f \in C^{1,\gamma}(x_0,t_0)\cap C^{0,1}(\R^d\times (t_0-r_0, t_0+r_0)),
	\end{align*}
then classically, $f$ satisfies at $(x_0,t_0)$,
\begin{align*}
	\partial_t f(x_0, t_0) = M(f(\cdot, t_0), x_0).
\end{align*}

\end{lemma}

\begin{proof} From Proposition \ref{propViscSol:EquivalenceSolsHeleShawMuskat}, $g(x,t) = f(x,t) + t$ for $(x,t)\in \R^d\times [0,T)$ is a viscosity solution of \eqref{eqViscSol:HeleShawHJB-WithInitialCondition}, and also $g\in C^{1,\gamma}(x_0,t_0)$ as well. By Proposition \ref{propViscSol:pointwise-testing}, we can use $g $ as a test function against itself at $(x_0, t_0)$, yielding $\partial_t g(x_0, t_0) \leq H(g(\cdot, t_0), x_0)$. Applying a similar argument to the supersolution case in Proposition \ref{propViscSol:pointwise-testing}, we obtain  
\begin{equation*}
	\partial_t g(x_0, t_0)\geq H(g(\cdot, t_0), x_0).
\end{equation*}
Since $M(f, x_0) = H(f, x_0) - 1$ and $H(f + c, x) = H(f, x)$ (by (iii) of Proposition \ref{propAuxOp:HProperties}), the desired result follows. 
\end{proof}

Interestingly (and possibly not surprisingly, a posteriori), this matches the much earlier theory for free boundary problems, which is closely related; see, for instance, \cite[Lemma 11]{Caffarelli-1987HarnackInqualityApproachFBPart1RevMatIbero} and \cite[Lemma 11.17]{CaffarelliSalsa-2005GeometricApproachtoFB}.  We note that the analog for second order equations is not true, namely punctual $C^{1,1}$ regularity of a test function does not ensure that it satisfies the equation.  

The use of this convenient feature goes back at least to \cite[Proposition 2]{BaIm-07} and \cite[Lemma 4.3]{CaSi-09RegularityIntegroDiff}.  We note that in the explicitly integro-differential works, like \cite{BaIm-07} and \cite{CaSi-09RegularityIntegroDiff} (there are, of course, \emph{many} works on the topic), the equations always are assumed to take the form of a min-max of linear integro-differential operators, and as a consequence, the pointwise evaluation property follows in a much more straightforward fashion than it did in this work.  Furthermore, it appears as though the pointwise  evaluation result was a convenience, but not a necessity in these earlier works ---  meaning the comparison results for viscosity solutions could have been proved without it.  

For the Hele-Shaw operator $I$, mentioned here in (\ref{eqIntro:I}), the pointwise  evaluation property was proven in \cite[Lemma 5.11, Corollary 5.12]{ChangLaraGuillenSchwab-2019SomeFBAsNonlocalParaboic-NonlinAnal}, which is a small modification of \cite[Lemma 11.17]{CaffarelliSalsa-2005GeometricApproachtoFB}.  Subsequently, this was developed via a different argument and used in \cite[Corollary 2.14, Proposition 6.2]{DongGancedoNguyen-2023MuskatOnePhase2DWellPose-CPAM}, \cite[Corollary 3.4, Proposition 3.7]{DongGancedoNguyen-2023MuskatWP3D-arXiv}.  Interestingly, in  contrast to the earlier works on integro-differential equations that were not focused on free boundary problems, it is not clear whether the proofs here and in \cite{ChangLaraGuillenSchwab-2019SomeFBAsNonlocalParaboic-NonlinAnal}, \cite{DongGancedoNguyen-2023MuskatOnePhase2DWellPose-CPAM}, \cite{DongGancedoNguyen-2023MuskatWP3D-arXiv} could be completed without the pointwise  evaluation property.

%%%%%%%%%%%%%%%%%%%%%%%%%%%%%%%%%%%%%%%%%%%%%%%%%
%%%%%%%%%%%%%%%%%%%%%%%%%%%%%%%%%%%%%%%%%%%%%%%%%
%%%%%%%%%%%%%%%%%%%%%%%%%%%%%%%%%%%%%%%%%%%%%%%%%
%%%%%%%%%%%%%%%%%%%%%%%%%%%%%%%%%%%%%%%%%%%%%%%%%
%%%%%%%%%%%%%%%%%%%%%%%%%%%%%%%%%%%%%%%%%%%%%%%%%
%%%%%%%%%%%%%%%%%%%%%%%%%%%%%%%%%%%%%%%%%%%%%%%%%
%%%%%%%%%%%%%%%%%%%%%%%%%%%%%%%%%%%%%%%%%%%%%%%%%
%%%%%%%%%%%%%%%%%%%%%%%%%%%%%%%%%%%%%%%%%%%%%%%%%
%%%%%%%%%%%%%%%%%%%%%%%%%%%%%%%%%%%%%%%%%%%%%%%%%

\appendix

\section{Background results about Harmonic functions}\label{sec:Background-HarmonicFunctions}

We now provide some  well-known result for existence and uniqueness of harmonic functions in  unbounded domains of the form $D_f$:
\begin{align}\label{eqAppendix:DirichletProplemInDf}
	\begin{cases}
		\Delta u_{f,g} = 0\ &\text{in}\ D_f,\\
		u_{f,g} = g\ &\text{on}\ \Gam_f.
	\end{cases}
\end{align}
 We could not find an immediate and easily available reference, and so we included that argument here for completeness.  One reference is \cite{Ishii-1989UniqueViscSolSecondOrderCPAM}, but that work treats fully nonlinear elliptic equations on possibly unbounded domains, and as such, the presentation is notably more complicated than is needed for harmonic functions.

The first result is the weak maximum principle for (\ref{eqAppendix:DirichletProplemInDf}). We say that $u\in C^2(D_f)$ is subharmonic in $D_f$ if $-\Delta u \leq 0$ in $D_f$, and $u$ is superharmonic if $-u$ is subharmonic in $D_f$.

\begin{proposition}[Weak maximum principle on subgraph domains for bounded functions]\label{propAppendix:weak-max-principle} Let $f\in C(\mathbb{R}^d)\cap L^\infty(\mathbb{R}^d)$ and $u\in C^2(D_f)\cap C(\overline{D}_f)$ be bounded from above. Assume that $u\leq 0$ on $\Gamma_f$ and $u$ is subharmonic in $D_f$. Then $u\leq 0$ in $D_f$. 
\end{proposition}

Using the maximum principle, we can employ the classical Perron's method to construct the solutions to the Dirichlet problem \eqref{eqAppendix:DirichletProplemInDf} (see \cite[Theorems 2.12, 2.14]{GilbargTrudinger-98BookEllipticPDE}). We note that as the argument is local, the fact that we are working on an unbounded domain does not create any issues, as long as we have comparison in Proposition \ref{propAppendix:weak-max-principle}.  We thus omit the proof of Proposition \ref{propAppendix:existence-Perron}.

\begin{proposition}[Existence of bounded solutions]\label{propAppendix:existence-Perron}
	Given $f\in C^{0,1}(\mathbb{R}^d)$ and $g\in C^{0,1}(\Gamma_f)$, there exists a unique classical, bounded solution, $u_{f,g}\in C^2(D_f)\intersect C(\overline D_f)$, to (\ref{eqAppendix:DirichletProplemInDf}).
\end{proposition}

We note that the Perron's solution achieves boundary values at $x\in \Gamma_f$ in the classical sense if and only if the boundary point $x$ is regular, meaning that a local \emph{barrier} exists. We note that when $\Gamma_f$ is Lipschitz, every point on the boundary $\Gamma$ satisfies the exterior cone condition, which is enough to conclude an existence of a barrier.

\begin{proof}[Proof of Proposition \ref{propAppendix:weak-max-principle}]
	
We proceed by contradiction, and assume 
$	\sup_{D_f} u = m > 0$. 
Note that $m<+\infty$ since, by assumption, $u$ is bounded. 
Let us fix $\varepsilon>0$ so that $4\varepsilon<m$ and let $\bar{X} = (\bar{x}, \bar{x}_{d+1})\in D_f$ be such that 
\begin{align}\label{eq:Background:choiceOfX0}
	m \geq u(\bar{X}) > m-\varepsilon.
\end{align}
Denote $K = \Vert f\Vert_{L^\infty}$ and define $\delta>0$ via $\delta = \min \left\lbrace\frac{\varepsilon}{|\bar{x}_{d+1}|}, \frac{\varepsilon}{K}   \right\rbrace$.  
 Letting $A = \delta^{-1}m$, we consider the strip $S = \left\lbrace x_{d+1} > -A \right\rbrace \cap D_f$. 
Note that the boundary of $S$ is made up of two disjoint components: $\Gamma_f$ and $\{x_{d+1} = -A\}$. In addition, the definition of $\delta$ implies,
\begin{align}\label{eq:Background:propertiesOfDelta}
	\delta \bar{x}_{d+1}\geq -\varepsilon,
\end{align}
which, together with our choice of $\ep$, yields, $\bar{x}_{d+1} \geq -\delta^{-1}\varepsilon > -\delta^{-1}m = -A$; thus, we conclude $\bar X\in S$. For future use, we also note,
\begin{align}\label{eq:Background:propertiesOfDelta2}
\textnormal{for all $(x,x_{d+1})\in D_f $ we have }
    \delta x_{d+1} \leq \delta K \leq \varepsilon.
    \end{align}

Denoting $\alpha = \frac{3}{4K^2}$, we define the auxiliary function $
    \psi(x) = e^{-\alpha|x_{d+1}-2K|^2}$  for $X = (x,x_{d+1}) \in D_f$. We have
\begin{align*}
    \Delta \psi (X)&
    = 2\alpha e^{-\alpha |x_{d+1}-2K|^2}\left(2\alpha (x_{d+1}-2K)^2 -1\right).
\end{align*}
Let us now consider $X\in S$, so that, 
$    -A - 2K\leq x_{d+1} -2K \leq -K$. 
Since $-A-2K<0$, we find,
\begin{align*}
    (A+2K)^2\geq |x_{d+1} -2K|^2 \geq K^2 \textnormal{ for }X\in S.
\end{align*}
Hence, for $X\in S$,
\begin{align*}
    \Delta \psi(X) \geq 2\alpha e^{-\alpha |x_{d+1}-2K|^2}\left(  2\alpha K^2 -1\right)
    \geq 2\alpha e^{-\alpha (A+2K)^2}.
\end{align*}
Our choice of $\alpha$ thus implies that there exists $C_{A,K} >0$ with,
\begin{align}
\label{eq:Delta psi}
    \Delta \psi(X)\geq C_{A,K} \textnormal{ for all }X\in S.
\end{align}
In addition, let $\eta(X) \in C_c^\infty(B_1(0))$ be a function such that $\eta(X) = \eta(|X|)$, $0\leq \eta \leq 1$,  $\eta(0) = 1$, and $\Vert D^2\eta\Vert_{L^\infty} \leq 2$. Now, we are ready to put together our perturbation of $u$: for any $0<t<1$ we define $w$ by,
\begin{align*}
    w(X) = u(X)+\delta x_{d+1} +\ep \psi(X) +2m\eta(t(X-\bar{X})). 
\end{align*}
Using \eqref{eq:Background:choiceOfX0}, \eqref{eq:Background:propertiesOfDelta}, $\psi\geq 0$, $\eta(0)=1$, and our choice $4\varepsilon < m$ yields,
\begin{align*}
    w(\bar{X}) = u(\bar{X}) + \delta \bar{x}_{d+1} + \varepsilon \psi(\bar{X}) + 2m \eta(0) \geq (m - \varepsilon) + (- \varepsilon) +  (0) + 2m = 3m - 2\varepsilon \geq 2.5m.
\end{align*}
We shall now show 
\begin{align}\label{eq:w on bdry}
    w(X) < 2.5m \textnormal{ holds on }\partial(S\cap \{|X-\bar{X}|=t^{-1}\}).
\end{align}
Since $\bar X \in S\cap \{|X-X_0|=t^{-1}\}$, the previous two lines will 
 imply that ${w}$ has a local maximum at some point $X^*\in int (S\cap \{|X-\bar X|\geq t^{-1}\})$, so that,
\begin{align*}
	0\geq \Delta {w}(X^*) \geq 0+C_{A,K}\ep   -2m t^2 \Delta \eta (t(X^*-X_0)).
\end{align*}
The second inequality follows since, by assumption, $u$ is subharmonic in $D_f$, as well as from (\ref{eq:Delta psi}). 
 Recalling $\Vert D^2\eta \Vert_{L^\infty}\leq 2$ and letting $t\rightarrow 0$ in the previous line,  we obtain the desired contradiction.

We now establish (\ref{eq:w on bdry}). There are three cases: 

\begin{itemize}
    \item If $X\in \Gamma_f$ then $u(X) \leq 0$, thus, using \eqref{eq:Background:propertiesOfDelta2} and our choices of $\delta$ and $\ep$ yields,
    \begin{align*}
        w(X) \leq 0+ \delta K +\ep +2m \leq 0 + \varepsilon + \varepsilon +2m = 2m+2\varepsilon < 2.5m.
    \end{align*}
    \item If $X$ is such that $x_{d+1}=-A=-\delta^{-1}m$ then $u(X) + \delta x_{d+1} = u(X)-m \leq 0$ holds, by the definition of $m$. Thus, using that $\psi\leq 1$ and $\eta\leq 1$, followed by our choice of $\ep$, gives,
\begin{align*}
        w(X)\leq u(X) + \delta x_{d+1}+ \varepsilon + 2m \leq 0
        + \varepsilon +2m  < 0 + m/4 + 2m = 2.25m.
\end{align*}
    \item Finally, for $X\in S$ with $|X-X_0|=t^{-1}$, we have that the $\eta$ term vanishes (by definition of $\eta$). Therefore, \eqref{eq:Background:propertiesOfDelta2} and our choice of $\ep$ yield,
    \begin{align*}
        w(X)\leq m+\delta K +\ep +0 \leq m+\ep+m/4 + 0 \leq m+m/4+m/4 = 1.5m.
    \end{align*}
   \end{itemize}
Therefore, we see condition (\ref{eq:w on bdry}) holds, and therefore the proof is complete.
\end{proof}

We conclude the appendix with a simplified proof of a modified version of an estimate that appeared in  \cite[Proposition 4.8]{AbedinSchwab-2020RegularityTwoPhaseHSParrabolic-JFA}, which includes a more precise statement to include $f\in \K^*(\gam,M)$, instead of  $C^{1,\gam}(\real^d)$.

\begin{lemma}\label{lemAppendix:BetterVersionASLemA.5}
	Let $R>1$. There exist positive constants $C=C(\gam,d,m)$ and $\al=\al(d,m)$ so that if 
	
	\begin{enumerate}[(i)]
		\item $f\in C^{0,1}(\real^d)$, and  $\norm{\grad f}_{L^\infty}(\real^d)\leq m$,
		
		\item there exists $w\in \K(\gam,m)$, with $w\geq f$ and $w(x_0)=f(x_0)$,
		
		\item $g\in C^0(\Gam_f)$ with $\abs{g}\leq 1$ and $g=0$ in $B_R(X_0)\intersect \Gam_f$, (recall $X_0=(x_0,f(x_0))$)
	\end{enumerate}
	and for $u$ that solves
	\begin{align*}
		\begin{cases}
			-\Delta u = 0\ &\textnormal{in}\ D_f\\
			u=g\ &\textnormal{on}\ \Gam_f,
		\end{cases}
	\end{align*} 
	then $u$ enjoys the growth estimate, for $0\leq s \leq 1$
	\begin{align*}
		0\leq u(X_0+s\nu_w) \leq C\frac{s}{R^\al}.
	\end{align*}

\end{lemma}

\begin{proof}[Proof of Lemma \ref{lemAppendix:BetterVersionASLemA.5}]

	We will rescale $u$ so that the boundary data is zero in $B_1$.  To this end, we define the functions
	\begin{align*}
		f_R(y) = \frac{1}{R}f(Ry)
	\end{align*}
	and
	\begin{align*}
		u_R(Y) = u(RY).
	\end{align*}
	We see that for $Y\in B_1(X_0)$, $X=RY\in B_R(X_0)$, and
	\begin{align*}
		Y\in\Gam_{f_R}\ \ \iff\ \ X=RY\in \Gam_f.
	\end{align*}
	Thus, we see that $u_R$ solves
	\begin{align*}
		\begin{cases}
			-\Delta u_R = 0\ &\textnormal{in}\ D_{f_R}\\
			u_R=g_R\ &\textnormal{on}\ \Gam_{f_R}
		\end{cases}
	\end{align*}
	and $g_R=0$ in $B_1$.  Thus, by the H\"older regularity of harmonic functions in Lipschitz domains, we see 
	\begin{align*}
		0\leq u_R(Y)\leq C(d(Y,\Gam_{f_R}))^\al.
	\end{align*} 
	(We note this follows from flattening the domain and extending $\tilde u$ --- possible because $u=0$ on $\Gam_f\intersect B_R$ --- and then invoking the local H\"older regularity of, e.g. De Giorgi - Nash - Moser theory.)

	Unscaling this estimate, for $X=RY$, by the Lipschitz property of $f$ and $f_R$, we see that 
	\begin{align*}
		d(RY, \Gam_f) \approx Rd(Y,\Gam_{f_R}).
	\end{align*}
	Hence,
	\begin{align*}
		u(X) = u(RY) = u_R(Y) \leq C(d(Y,\Gam_{f_R}))^\al \leq C\left( \frac{d(X,\Gam)}{R} \right)^\al.
	\end{align*}
	As we only use this for $d(X,\Gam_f)\leq 1$, we see that we have the following:
	\begin{align*}
		\textnormal{for}\ d(X,\Gam)\leq 1,\ \ u\leq \frac{C}{R^\al}.
	\end{align*}
	Now, taking a barrier function from Lemma \ref{lemAuxOp:BarrierFuns}, $b$, in $D_w\intersect B_1(X_0)$,  we can conclude that
	\begin{align*}
		0\leq u(X_0+s\nu)\leq V(X_0+s\nu)\leq s\frac{C}{R^\al}.
	\end{align*}
\end{proof}

%%%%%%%%%%%%%%%%%%%%%%%%%%%%%%%%%%%%%%%%%%%%%%
%%%%%%%%%%%%%%%%%%%%%%%%%%%%%%%%%%%%%%%%%%%%%%
%%%%%%%%%%%%%%%%%%%%%%%%%%%%%%%%%%%%%%%%%%%%%%
%%%%%%%%%%%%%%%%%%%%%%%%%%%%%%%%%%%%%%%%%%%%%%
%%%%%%%%%%%%%%%%%%%%%%%%%%%%%%%%%%%%%%%%%%%%%%
%%%%%%%%%%%%%%%%%%%%%%%%%%%%%%%%%%%%%%%%%%%%%%
%%%%%%%%%%%%%%%%%%%%%%%%%%%%%%%%%%%%%%%%%%%%%%
%%%%%%%%%%%%%%%%%%%%%%%%%%%%%%%%%%%%%%%%%%%%%%
%%%%%%%%%%%%%%%%%%%%%%%%%%%%%%%%%%%%%%%%%%%%%%
%%%%%%%%%%%%%%%%%%%%%%%%%%%%%%%%%%%%%%%%%%%%%%

\bibliography{refs-MuskatSTT}
% \bibliography{./refs-MuskatSTT}
\bibliographystyle{alpha}
%%%%%%%%%%%%%%%%%%%%%%%%%%%%%%%%%%%%%%%%%%%%%%

\end{document}